\documentclass[11pt,reqno]{amsart}

\usepackage[T1]{fontenc}
\usepackage[utf8]{inputenc}
\usepackage{lmodern}
\usepackage[a4paper,margin=32mm]{geometry}
\usepackage{microtype}

\usepackage{amsmath,amssymb,mathtools}
\usepackage{mathrsfs}
\usepackage{bm}
\usepackage{tikz-cd}
\tikzcdset{
  scale cd/.style={
    every label/.append style={scale=#1},
    cells={nodes={scale=#1}}
  }
}

\usepackage{booktabs}
\usepackage{enumitem}
\usepackage{xcolor}
\usepackage{aliascnt}
\definecolor{journalblue}{RGB}{32,72,120}

\usepackage[
  colorlinks=true,
  linkcolor=journalblue,
  citecolor=journalblue,
  urlcolor=journalblue,
  pdfusetitle
]{hyperref}
\usepackage[nameinlink,capitalise,noabbrev]{cleveref}
\crefalias{subsection}{section}
\crefalias{subsubsection}{section}

\setlist{itemsep=2pt,topsep=4pt}
\allowdisplaybreaks
\numberwithin{equation}{section}

\newtheorem{theorem}{Theorem}[section]

\newaliascnt{lemma}{theorem}

\aliascntresetthe{lemma}

\newaliascnt{proposition}{theorem}
\newtheorem{proposition}[proposition]{Proposition}
\aliascntresetthe{proposition}

\newaliascnt{corollary}{theorem}

\aliascntresetthe{corollary}

\theoremstyle{definition}
\newaliascnt{definition}{theorem}
\newtheorem{definition}[definition]{Definition}
\aliascntresetthe{definition}

\newaliascnt{example}{theorem}

\aliascntresetthe{example}

\newaliascnt{assumption}{theorem}

\aliascntresetthe{assumption}

\newaliascnt{notation}{theorem}
\newtheorem{notation}[notation]{Notation}
\aliascntresetthe{notation}

\theoremstyle{remark}
\newaliascnt{remark}{theorem}
\newtheorem{remark}[remark]{Remark}
\aliascntresetthe{remark}

\crefname{theorem}{theorem}{theorems}
\Crefname{theorem}{Theorem}{Theorems}
\crefname{lemma}{lemma}{lemmas}
\Crefname{lemma}{Lemma}{Lemmas}
\crefname{proposition}{proposition}{propositions}
\Crefname{proposition}{Proposition}{Propositions}
\crefname{corollary}{corollary}{corollaries}
\Crefname{corollary}{Corollary}{Corollaries}
\crefname{definition}{definition}{definitions}
\Crefname{definition}{Definition}{Definitions}
\crefname{example}{example}{examples}
\Crefname{example}{Example}{Examples}
\crefname{assumption}{assumption}{assumptions}
\Crefname{assumption}{Assumption}{Assumptions}
\crefname{notation}{notation}{notations}
\Crefname{notation}{Notation}{Notations}
\crefname{remark}{remark}{remarks}
\Crefname{remark}{Remark}{Remarks}

\title[Weak Lie 3-groups, 2-gerbes and T-duality]{Weak Lie 3-groups, 2-gerbes over torus fibrations of type $F_1$, and T-duality}

\author{Roberto Tellez-Dominguez}
\address{%
  Department of Mathematics, CEU San Pablo University, Madrid, Spain
}
\email{rtellezd46@gmail.com}

\date{September 3, 2026}

\subjclass[2020]{Primary 53C08; Secondary 18N20, 18G45, 55R10, 81T30}
\keywords{weak Lie 3-groups, gerbes, 2-gerbes, T-duality, torus fibrations, dimensional reduction}

\begin{document}

\begin{abstract}
We construct a weak Lie 3-group $T_2\mathbb B_n^{F_1}$ from the 2-category of gerbes over $\mathbb R^n/\mathbb Z^n$ and the $\mathbb R^n/\mathbb Z^n$-action on it by pullback along translations. We also construct a homotopy equivalence between $T_2\mathbb B_n^{F_1}$ and a different Lie 3-group $T_2\mathbb D_n^{F_1}$, which admits a dimensional reduction to the homotopy equivalence of Lie 2-groups introduced by Nikolaus and Waldorf to model half-geometric T-duality. This is a first step towards establishing a higher form of T-duality for 2-gerbes over torus fibrations, relevant to supergravity and M-theory. 
\end{abstract}

\maketitle

\tableofcontents

\section{Introduction}\label{sec:introduction}

A common string-theoretic background is given by a principal $\mathbb R^n/\mathbb Z^n$-fibration $P \rightarrow X$ with a gerbe $\mathcal{L} \rightarrow P$ over it \cite{Murray:9407015,Bouwknegt:2003vb,Bouwknegt:2003zg,Bunke:2005um}. Here a gerbe is a geometric object constructed by gluing along transition $U(1)$-bundles, and classified by $H^3(P,\mathbb Z)$. In theories such as 11-dimensional supergravity or M-theory, $\mathcal{L}$ is expected to arise from transgression of a 2-gerbe over a principal $\mathbb R^{n+1}/\mathbb Z^{n+1}$-fibration $P \times_X L \rightarrow X$ \cite{Breen:1994aa,Fiorenza:2012mr,Fiorenza:2018ekd}. This 2-gerbe is a geometric object constructed by gluing along transition gerbes, and classified by $H^4(P,\mathbb Z)$. The purpose of this article is to provide the algebraic structures that control the gluing data determining both 2-gerbes over torus fibrations, and a transformation of these lifting T-duality of gerbes along transgression. This suggests the existence of an associated duality for the corresponding higher-dimensional field theories, in the style of U-duality \cite{Hull:1994ys} or other forms of higher T-duality \cite{Fiorenza:2018ekd,Chatzistavrakidis:2021buscher}.

The study of T-duality started with the observation that, with the right choices of background fields, string-theoretic backgrounds $(P,\mathcal{L})$ and $(\hat{P},\widehat{\mathcal{L}})$ that are non-homeomorphic may nevertheless lead to equivalent quantum field theories. Topological T-duality focuses on the study of questions that are independent of the fields, such as which backgrounds $(P,\mathcal{L})$ admit a corresponding T-dual $(\hat{P},\widehat{\mathcal{L}})$, how many non-homeomorphic T-duals may exist, or how these may be constructed \cite{Bouwknegt:2003vb,Bunke:2005sn,Bunke:2005um}. One of the first results was the following. The fibration $\mathbb R^n/\mathbb Z^n \rightarrow P \rightarrow X$ determines, via the Serre spectral sequence, a filtration of $H^3(P,\mathbb Z)$ of the form $\pi^{\ast}H^3(X,\mathbb Z) = F_3 \subset F_2 \subset F_1 \subset F_0 = H^3(P,\mathbb Z)$. It turns out that a T-dual exists if and only if the class of the gerbe $\mathcal{L} \rightarrow P$ in $H^3(P,\mathbb Z)$ lies in the $F_2$ stage of this filtration. However, as noted in \cite{Hull:2006qs}, a gerbe whose class lies in the $F_1$ stage of the filtration can be locally identified with a gerbe whose class lies in the $F_2$ stage. Therefore, it is to be expected that local T-duals may be defined which, while not gluing to a proper global T-dual, could glue into a different kind of geometric object. This object has been modelled in terms of noncommutative geometry \cite{Mathai:2004qc,Mathai:2004qq,Mathai:2005fd} and, more importantly for this paper, in terms of Lie 2-groups.

A Lie 2-group is a smooth, categorified group. It is the algebraic structure playing the role of the structure group of fibrations whose gluing data has itself a higher level of symmetries. In our setting, this is the case for $(P,\mathcal{L})$, whose gluing data includes $U(1)$-bundles over the torus $\mathbb R^n/\mathbb Z^n$. Nikolaus and Waldorf \cite{Nikolaus:2018qop} proved that there is a commutative diagram of Lie 2-groups of the form
    \begin{equation}\label{eq:tdualspan}
    \begin{tikzcd}
        T\mathbb B_n^{F_2} \ar[d] & T\mathbb D_n^{F_2} \ar[l] \ar[d] \ar[r] & T\mathbb B_n^{F_2} \\ 
        T\mathbb B_n^{F_1} & T\mathbb D_n^{F_1} \ar[l] & 
    \end{tikzcd},
    \end{equation}
    and with the following properties.
    \begin{enumerate}
        \item\label{it:intro1} The underlying category of $T\mathbb B_n^{F_1}$ is $\mathbb R^n/\mathbb Z^n \times L(\mathbb R^n/\mathbb Z^n)$, where $L(\mathbb R^n/\mathbb Z^n)$ is the category of $U(1)$-bundles over $\mathbb R^n/\mathbb Z^n$. The 2-group structure is a semidirect product of addition in $\mathbb R^n/\mathbb Z^n$, tensor product in $L(\mathbb R^n/\mathbb Z^n)$, and action of $\mathbb R^n/\mathbb Z^n$ on $L(\mathbb R^n/\mathbb Z^n)$ by pull-back along translations. This implies $T\mathbb B_n^{F_1}$-bundles over $X$ are in one-to-one correspondence with backgrounds $(P,\mathcal{L})$ over $X$ of type $F_1$.
        \item\label{it:intro1.5} $T\mathbb B_n^{F_2}$ is the full sub-2-group of $T\mathbb B_n^{F_1}$ with set of objects $\mathbb R^n/\mathbb Z^n \times \{ 1 \}$. This implies $T\mathbb B_n^{F_2}$-bundles are in one-to-one correspondence with backgrounds $(P,\mathcal{L})$ of type $F_2$.
        \item\label{it:intro2} Homomorphisms $T\mathbb B_n^{F_i} \leftarrow T\mathbb D_n^{F_i}$ are homotopy equivalences for $i=1, \,2$. In particular, every $T\mathbb B_n^{F_i}$-bundle can be lifted to a $T\mathbb D_n^{F_i}$-bundle.
        \item\label{it:intro3} Two backgrounds $(P,\mathcal{L})$ and $(\hat{P},\widehat{\mathcal{L}})$ are T-dual if and only if they arise from the same $T\mathbb D_n^{F_2}$-bundle through the upper horizontal arrows.
    \end{enumerate}
    Thus, given a background $(P,\mathcal{L})$ of type $F_1$, one can lift it by \ref{it:intro2} to a $T\mathbb D_n^{F_1}$-bundle, and interpret the result by \ref{it:intro3} as the geometric object containing the data of all its local T-duals. Later work has also managed to model the transformation of string-theoretic fields by considering connections on these bundles \cite{Kim:2022opr,Waldorf:2022tib}.

    \subsection{Main result}

    The main result of this article is a higher categorification of the results in \cite{Nikolaus:2018qop}, related to 2-gerbes over $\mathbb R^{n+1}/\mathbb Z^{n+1}$-fibrations $P \rightarrow X$. It requires the introduction of Lie 3-groups, in order to account for the higher level of symmetries present in 2-gerbes. To understand the notation, note that the Serre spectral sequence of $P$ also induces a filtration $\pi^{\ast}H^4(X,\mathbb Z) = F_4 \subset F_3 \subset  F_2 \subset F_1 \subset F_0 = H^4(P,\mathbb Z)$, which can be used to classify 2-gerbes over $P$.
    \begin{theorem}\label{th:mainintro}
        There exist Lie 3-groups $T_2\mathbb B_{n+1}^{F_1}$, $T_2\mathbb D_{n+1}^{F_1}$, $T_2\mathbb B_{n+1}^{F_2}$, $T_2\mathbb D_{n+1}^{F_2}$ fitting in a commutative diagram
        \begin{equation}
            \begin{tikzcd}
        && {T \mathbb B_n^{F_2}} & {T \mathbb D_n^{F_2}} \\
        {T_2\mathbb B_{n+1}^{F_2}} & {T_2\mathbb D_{n+1}^{F_2}} & {T \mathbb B_n^{F_1}} & {T \mathbb D_n^{F_1}} \\
        {T_2\mathbb B_{n+1}^{F_1}} & {T_2\mathbb D_{n+1}^{F_1}}
        \arrow[from=1-3, to=2-1]
        \arrow[from=1-3, to=2-3]
        \arrow[from=1-4, to=1-3]
        \arrow[from=1-4, to=2-2]
        \arrow[from=1-4, to=2-4]
        \arrow[from=2-1, to=3-1]
        \arrow[from=2-2, to=2-1]
        \arrow[from=2-2, to=3-2]
        \arrow[from=2-3, to=3-1]
        \arrow[from=2-4, to=2-3]
        \arrow[from=2-4, to=3-2]
        \arrow[from=3-2, to=3-1]
    \end{tikzcd}
        \end{equation}
        with the following properties.
        \begin{enumerate}
            \item\label{it:main1} The underlying 2-category of $T_2\mathbb B_{n+1}^{F_1}$ is equivalent to $\mathbb R^{n+1}/\mathbb Z^{n+1} \times \mathcal{G}(\mathbb R^{n+1}/\mathbb Z^{n+1})$, for $\mathcal{G}(\mathbb R^{n+1}/\mathbb Z^{n+1})$ the 2-category of gerbes over $\mathbb R^{n+1}/\mathbb Z^{n+1}$. The 3-group structure is a semidirect product of addition in $\mathbb R^{n+1}/\mathbb Z^{n+1}$, tensor product in $\mathcal{G}(\mathbb R^{n+1}/\mathbb Z^{n+1})$, and action of $\mathbb R^{n+1}/\mathbb Z^{n+1}$ on $\mathcal{G}(\mathbb R^{n+1}/\mathbb Z^{n+1})$ by pull-back along translations. 
            \item\label{it:main2} $T_2\mathbb B_{n+1}^{F_2}$ is the full sub-3-group of $T_2\mathbb B_{n+1}^{F_1}$ with set of objects $\mathbb R^{n+1}/\mathbb Z^{n+1} \times \{ 1 \}$.
            \item\label{it:main3} Homomorphisms $T_2\mathbb B_{n+1}^{F_i} \leftarrow T_2\mathbb D_{n+1}^{F_i}$ are homotopy equivalences for $i=1, \,2$.
            \item\label{it:main4} Homomorphisms $T \mathbb B_n^{F_i} \rightarrow T_2 \mathbb B_{n+1}^{F_i}$ are injective.
        \end{enumerate}
    \end{theorem}

Theorem \ref{th:mainintro} should be interpreted as follows. By \ref{it:main1} and \ref{it:main2}, $T_2\mathbb B_{n+1}^{F_i}$-bundles are to be thought of as 2-gerbes over $\mathbb R^{n+1}/\mathbb Z^{n+1}$-fibrations of type $F_i$, as these are the objects that can be constructed by gluing along gerbes over $\mathbb R^{n+1}/\mathbb Z^{n+1}$-fibers. By \ref{it:main3}, $T_2\mathbb B_{n+1}^{F_i}$-bundles can be lifted to $T_2\mathbb D_{n+1}^{F_i}$-bundles. By \ref{it:main4} and the results in \cite{Nikolaus:2018qop}, the corresponding $T_2\mathbb D_{n+1}^{F_i}$-bundles contain additional data describing the construction of dual backgrounds. However, we note that in this article we work purely at the algebraic level. Associated bundles, connections on them, and a physical interpretation of \cref{th:mainintro} in terms of a duality for higher-dimensional field theories lifting T-duality are studied in a separate paper with G. Gagliardo and C. Sämann \cite{Gagliardo:2026inprep}.
     
    \subsection{Organization of the paper}

    In \cref{sec:basic} we recall basic definitions from 2-category theory and gerbes which are needed to fix notation. In \cref{sec:lie3} we introduce the model for Lie 3-groups that is used in this paper, which is weaker than the model based on Lie 2-crossed modules, commonly found in the literature. In \cref{sec:results} we construct the Lie 3-groups of interest, and prove \cref{th:mainintro} by splitting it into \cref{prop:t2dnwhet2bn}, \cref{th:t2bnf1isgerbes} and \cref{th:dimredf1}.

    \section{Basic definitions}\label{sec:basic}
    In this section we collect some standard definitions to fix notation.

    \subsection{2-categories, pseudofunctors and pseudonatural transformations}\label{sec:2cat}

    A \emph{2-category} \cite{Leinster:1998aa} $C$ is a collection $C_0$ of \emph{objects}, a collection $C_1$ of \emph{arrows} and a collection $C_2$ of \emph{2-cells} with source and target maps of the form
    \begin{equation}
        \begin{tikzcd}[ampersand replacement = \& ]
                g \ar[r,bend left = 40, "{\gamma}"{name=F},pos=0.55] \ar[r,bend right = 40, "{\eta}"{name=G},swap,pos=0.5] \ar[Rightarrow,from=F,to=G,"{\psi}",swap] \& g'
            \end{tikzcd},
    \end{equation}
    along with horizontal and vertical composition rules
    \begin{equation}
        \begin{aligned}
        &\begin{tikzcd}[ampersand replacement = \&]
                g \ar[r,bend left = 40, "{\gamma}"{name=F},pos=0.55] \ar[r,bend right = 40, "{\eta}"{name=G},swap,pos=0.5] \ar[Rightarrow,from=F,to=G,"{\psi}",swap] \& g' \ar[r,bend left = 40, "{\gamma'}"{name=F2}] \ar[r,bend right = 40, "{\eta'}"{name=G2},swap] \ar[Rightarrow,from=F2,to=G2,"{\psi'}",swap] \& g''
            \end{tikzcd} =
            \begin{tikzcd}[ampersand replacement = \&,column sep = 9ex]
                g \ar[r,bend left = 40, "{\gamma' \circ \gamma}", ""{name=F,coordinate,pos=0.5},pos=0.55] \ar[r,bend right = 40, "{\eta' \circ \eta}"{swap}, ""{name=G,coordinate,pos=0.5},pos=0.55] \ar[Rightarrow,from=F,to=G,"{\psi' \circ \psi}",swap] \& g''
            \end{tikzcd} ,
        &\qquad
        \begin{tikzcd}[ampersand replacement = \& ,column sep = 12ex]
                g \ar[r,bend left = 70, "{\gamma}"{name=F}] \ar[r,bend left = 10, "{\eta}"{name=G},swap,pos=0.5]
                \ar[r,bend right = 70,"{\lambda}"{name=H},swap]
                \ar[Rightarrow,from=F,to=G,"{\psi}",swap,pos=0.2,shorten = 1ex]
                \ar[Rightarrow,from=G,to=H,"{\phi}",swap,pos=0.4]
                \& g'
            \end{tikzcd} = 
        \begin{tikzcd}[ampersand replacement = \& ,column sep = 12ex]
                g \ar[r,bend left = 70, "{\gamma}"{name=F}]
                \ar[r,bend right = 70,"{\lambda}"{name=H},swap]
                \ar[Rightarrow,from=F,to=H,"{\phi \circ \psi}",swap,pos=0.3]
                \& g'
            \end{tikzcd}.
        \end{aligned}
    \end{equation}
    which must be strictly associative and satisfy the interchange law
    \begin{equation}
        \begin{tikzcd}[ampersand replacement = \&, column sep = 12ex]
            g \ar[r,bend left = 70,"{\gamma' \circ \gamma}"{name=F}]
            \ar[r,"{\eta' \circ \eta}"{name=G},pos=0.55]
            \ar[r,bend right = 70,"{\lambda' \circ \lambda}"{name=H},swap]
            \ar[Rightarrow,from=F,to=G,"{\psi' \circ \psi}",swap]
            \ar[Rightarrow,from=G,to=H,"{\phi' \circ \phi}",swap]
            \& g''
        \end{tikzcd}
        =
        \begin{tikzcd}[ampersand replacement = \&, column sep = 10ex]
            g \ar[r,bend left = 55,"{\gamma}"{name=F}]
            \ar[r,bend right = 55,"{\lambda}"{name=H},swap]
            \ar[Rightarrow,from=F,to=H,"{\phi \circ \psi}",swap]
            \& g' \ar[r,bend left = 55,"{\gamma'}"{name=F2}]
            \ar[r,bend right = 55,"{\lambda'}"{name=H2},swap]
            \ar[Rightarrow,from=F2,to=H2,"{\phi' \circ \psi'}",swap]
            \& g''.
        \end{tikzcd}
    \end{equation}

    We also demand the existence of an identity arrow for each object, and an identity 2-cell for each arrow. If every arrow and every 2-cell is invertible, we say the 2-category is a \emph{2-groupoid}. If, additionally, $C_0$, $C_1$, and $C_2$ are all smooth manifolds, and all the structure maps are smooth, then we say $C$ is a \emph{Lie 2-groupoid}. A \emph{category} (resp. \emph{Lie groupoid}) is a 2-category (resp. Lie 2-groupoid) with only identity 2-cells.

    A \emph{pseudofunctor} $F: C \rightarrow D$ between 2-categories  $C$ and $D$ is a map sending 2-cells in $C$ to 2-cells in $D$, preserving their sources and targets, as well as the vertical composition rule, along with \emph{compositor cells} for each $g \stackrel{\gamma}{\rightarrow} g' \stackrel{\gamma'}{\rightarrow} g''$
    \begin{equation}
                \begin{tikzcd}[column sep = 5ex]
            F(g) \ar[rr,bend left=60,"{F(\gamma' \circ \gamma)}"{name=F}] 
              \ar[rr,bend right=10,"{F(\gamma') \circ F(\gamma)}"{name=G},swap]
             & & F(g'') \\ \ar[Rightarrow,from=F,to=G,"{F(\gamma,\gamma')}"{swap,pos=0.5},shorten >=1.5pt] &
                \end{tikzcd}
    \end{equation}
    which must be associative on arrows $g \stackrel{\gamma}{\rightarrow} g' \stackrel{\gamma'}{\rightarrow} g'' \stackrel{\gamma''}{\rightarrow} g'''$
    \begin{equation}
        \begin{tikzcd}[every label/.append style={font=\tiny},column sep=9ex]
            F(g) \ar[rr,bend left=75,"{F(\gamma'' \circ \gamma' \circ \gamma)}"{name=A}]
            \ar[rr,bend left=10,"{F(\gamma'' \circ \gamma') \circ F(\gamma)}"{name=B},swap]
            \ar[rr,bend right=80,"{(F(\gamma'') \circ F(\gamma')) \circ F(\gamma)}"{name=C},swap]
            & & F(g''')
            \ar[Rightarrow,from=A,to=B,"{F(\gamma,\gamma'' \circ \gamma')}"{swap},shorten >=1.5pt]
            \ar[Rightarrow,from=B,to=C,"{F(\gamma',\gamma'') \circ \mathrm{id}_{F(\gamma)}}"{swap},shorten >=1.5pt]
        \end{tikzcd}
        =
        \begin{tikzcd}[every label/.append style={font=\tiny},column sep=9ex]
            F(g) \ar[rr,bend left=75,"{F(\gamma'' \circ \gamma' \circ \gamma)}"{name=A}]
            \ar[rr,bend left=10,"{F(\gamma'') \circ F(\gamma' \circ \gamma)}"{name=B},swap]
            \ar[rr,bend right=80,"{F(\gamma'') \circ (F(\gamma') \circ F(\gamma))}"{name=C},swap]
            & & F(g''')
            \ar[Rightarrow,from=A,to=B,"{F(\gamma' \circ \gamma,\gamma'')}"{swap},shorten >=1.5pt,pos=0.7]
            \ar[Rightarrow,from=B,to=C,"{\mathrm{id}_{F(\gamma'')} \circ F(\gamma,\gamma')}"{swap},shorten >=1.5pt],
        \end{tikzcd}
    \end{equation}
    and natural on 2-cells $\psi: \gamma \Rightarrow \eta: g \rightarrow g'$ and $\psi': \gamma' \Rightarrow \eta':g' \rightarrow g''$
    \begin{equation}
        \begin{tikzcd}[every label/.append style={font=\tiny},column sep=9ex]
            F(g) \ar[rr,bend left=65,"{F(\gamma' \circ \gamma)}"{name=A}]
            \ar[rr,bend left=5,"{F(\gamma') \circ F(\gamma)}"{name=B},swap]
            \ar[rr,bend right=75,"{F(\eta') \circ F(\eta)}"{name=C},swap]
            & & F(g'')
            \ar[Rightarrow,from=A,to=B,"{F(\gamma,\gamma')}"{swap},shorten >=1.5pt]
            \ar[Rightarrow,from=B,to=C,"{F(\psi') \circ F(\psi)}"{swap},shorten >=1.5pt,pos=0.4]
        \end{tikzcd}
        =
        \begin{tikzcd}[every label/.append style={font=\tiny},column sep=9ex]
            F(g) \ar[rr,bend left=65,"{F(\gamma' \circ \gamma)}"{name=A}]
            \ar[rr,bend left=5,"{F(\eta' \circ \eta)}"{name=B},swap]
            \ar[rr,bend right=75,"{F(\eta') \circ F(\eta)}"{name=C},swap]
            & & F(g'')
            \ar[Rightarrow,from=A,to=B,"{F(\psi' \circ \psi)}"{swap},shorten >=1.5pt]
            \ar[Rightarrow,from=B,to=C,"{F(\eta,\eta')}"{swap},shorten >=1.5pt].
        \end{tikzcd}
    \end{equation}
    A \emph{functor} is a pseudofunctor in which the compositor cells are identities. When the 2-categories are Lie 2-groupoids, we say a pseudofunctor is \emph{smooth} if all structure maps are smooth.
    
    A \emph{pseudonatural transformation} $\alpha:F \Rightarrow G$ between pseudofunctors $F,G:C\rightarrow D$ is a collection of arrows and 2-cells in $D$
    \begin{equation}
        \alpha(g):F(g)\longrightarrow G(g), \quad g \in C_0 \qquad \begin{tikzcd}[column sep=12ex]
            F(g)  \ar[r, "{\alpha(g)}"{name=A}] \ar[d,"{F(\gamma)}",swap] & G(g) \ar[d,"{G(\gamma)}"] \\
            F(g') \ar[r,"{\alpha(g')}"{name=B},swap] & G(g')
            \ar[Rightarrow,from=1-2,to=2-1,"{\alpha(\gamma)}",swap]
        \end{tikzcd}, \quad \gamma:g\rightarrow g' \in C_1
    \end{equation}
    which must respect composition of arrows $g\stackrel{\gamma}{\rightarrow}g'\stackrel{\gamma'}{\rightarrow}g''$,
    \begin{equation}
        \begin{tikzcd}[column sep=12ex]
            F(g)  \ar[r, "{\alpha(g)}"{name=A}] \ar[d,"{F(\gamma)}",swap] & G(g) \ar[d,"{G(\gamma)}"] \\
            F(g') \ar[r,"{\alpha(g')}"{name=B},swap] \ar[d,"{F(\gamma')}",swap] & G(g') 
            \ar[Rightarrow,from=1-2,to=2-1,"{\alpha(\gamma)}",swap] \ar[d,"{G(\gamma')}"] \\
            F(g'') \ar[r,"{\alpha(g'')}"{name=B},swap] & G(g'')
            \ar[Rightarrow,from=2-2,to=3-1,"{\alpha(\gamma')}",swap,pos=0.65]
        \end{tikzcd} = 
        \begin{tikzcd}[column sep=12ex]
            F(g)  \ar[r, "{\alpha(g)}"{name=A}] \ar[d,"{F(\gamma)}",swap] & G(g) \ar[d,"{G(\gamma)}"] \\
            F(g'') \ar[r,"{\alpha(g'')}"{name=B},swap] & G(g'')
            \ar[Rightarrow,from=1-2,to=2-1,"{\alpha(\gamma' \circ \gamma)}",swap]
        \end{tikzcd}
    \end{equation}
    and be natural on 2-cells $\psi:\gamma\Rightarrow\eta:g\rightarrow g'$
    \begin{equation}
        \begin{tikzcd}[column sep=14ex]
             F(g) \ar[r,"{\alpha(g)}",{name=U}] \ar[d,"{}"{name=E}] \ar[d,bend right=100, "{F(\eta)}"{name=F},swap]  \ar[Rightarrow,from=E,to=F,"{F(\psi)}"]
             & G(g) \ar[d,"{G(\gamma)}"] \ar[Rightarrow, dl,"{\alpha(\gamma)}",swap]\\
             F(g') \ar[r,"{\alpha(g')}"{name=D}]
             & G(g')
            \end{tikzcd} = 
         \begin{tikzcd}[column sep=14ex]
             F(g)  \ar[d,"{F(\eta)}"{name=E},swap] \ar[r,"{\alpha(g)}",{name=U}]
             & F(g') \ar[d,"{}"{name=R}] \ar[Rightarrow, dl,"{\alpha(\eta)}",swap] \ar[d,bend left=100, "{G(\gamma)}"{name=G}] \ar[Rightarrow,from=G,to=R,"{G(\psi)}"]\\
             G(g) \ar[r,"{\alpha(g')}"{name=D},swap] 
             & G(g') 
            \end{tikzcd}.
    \end{equation}
    A \emph{natural transformation} is a pseudonatural transformation in which the defining 2-cells are all identities. When $C$ and $D$ are Lie 2-groupoids, and $F, \,G$ are smooth, a pseudonatural transformation is called \emph{smooth} if its object and arrow components depend smoothly on their arguments.

    The most basic example of a 2-category, denoted by $S \times BG \times B^2T$ for a set $S$, a group $G$ and an abelian group $T$, has 2-cells of the form
    \begin{equation}
        \begin{tikzcd}[ampersand replacement = \& ]
                s \ar[r,bend left = 40, "{g}"{name=F},pos=0.55] \ar[r,bend right = 40, "{h}"{name=G},swap,pos=0.5] \ar[Rightarrow,from=F,to=G,"{t}",swap] \& s'
            \end{tikzcd}, \quad \quad s, \,s' \in S, \: g, \,h \in G, \: t \in T,
    \end{equation}
    composed horizontally and vertically using the group products of $G$ and $T$. Another standard example is $\mathsf{Cat}$, the 2-category which has categories as objects, functors as arrows, and natural transformations as 2-cells, with their standard horizontal and vertical composition rules \cite{0387984038}.

    \subsection{Gerbes}

Given a submersion $\pi:Y \rightarrow X$, we write $Y^{[n]} := Y \times_{\pi} \cdots \times_{\pi} Y$ for the $n$-fold fibered product of $Y$ over $X$. We also write $p_{i_{1}...i_n}: Y^{[m]} \rightarrow Y^{[n]}$ with $1 \leq i_r \leq m$ for the map defined by the $i_1$-th, $i_2$-th, ..., $i_n$-th components of $Y^{[m]}$. For any object $L$ lying over $Y^{[n]}$ that can be pulled back along smooth maps (such as a line bundle, or an isomorphism of line bundles), we write $L_{i_{1}...i_n}$ for its pull-back along $p_{i_{1}...i_n}$.

    A \emph{bundle gerbe} $(Y,\pi,L,\alpha)$ over a manifold $X$ \cite{Murray:9407015} is a submersion $\pi: Y \rightarrow X$, a $U(1)$-bundle $\pi^L:L \rightarrow Y^{[2]}$ and an isomorphism $\alpha: L_{12} \otimes L_{23} \rightarrow L_{13}$ of $U(1)$-bundles over $Y^{[3]}$ such that $\alpha_{134} \circ \alpha_{123} = \alpha_{124} \circ \alpha_{234}: L_{12} \otimes L_{23} \otimes L_{34} \rightarrow L_{14}$ over $Y^{[4]}$. Note in particular $\alpha_{111}$ is equivalent to a section $\tau$ of $L_{11}$. A \emph{gerbe} $(\mathcal{L},\pi)$ over a manifold $X$ is a Lie groupoid $\mathcal{L}$ with a smooth \emph{projection} functor $\pi:\mathcal{L} \rightarrow X$ which can be constructed from a bundle gerbe $(Y,\pi,L,\alpha)$ by letting $\mathcal{L}_0 := Y$, $\mathcal{L}_1 := L$ with identities defined from $\tau$, $\pi:\mathcal{L} \rightarrow X$ by the submersion of the bundle gerbe, source and target maps
    \begin{equation}
        \begin{tikzcd}
            \pi^L_1(l) \ar[r,bend left = 20, "y"] & \pi^L_2(l)
        \end{tikzcd} , \quad \quad l \in L,
    \end{equation}
    and composition
    \begin{equation}
        \begin{tikzcd}[every label/.append style = {font = \footnotesize},column sep = 4ex]
            \pi^L_1(l_{12}) \ar[r,bend left = 20, "{l_{12}}"] & \pi^L_2(l_{12}) = \pi^L_1(l_{23})  \ar[r,bend left = 20, "{l_{23}}"] & \pi^L_2(l_{23})
        \end{tikzcd} := 
        \begin{tikzcd}[every label/.append style = {font = \footnotesize},column sep = 4ex]
             \pi^L_1(l_{12}) \ar[r,bend left = 20, "{\alpha(l_{12} \otimes l_{23})}"] & \pi^L_2(l_{23}).
        \end{tikzcd}
    \end{equation}

    If $(Y,\pi,L,\alpha)$, $(Y,\pi,L',\alpha')$ are bundle gerbes over $X$ defined over a common acyclic submersion $\pi:Y \rightarrow X$, then an \emph{isomorphism of bundle gerbes over the same acyclic submersion} $\beta : (Y,\pi,L,\alpha) \rightarrow (Y,\pi,L',\alpha')$ is an isomorphism of $U(1)$-bundles $\beta: L \rightarrow L'$ over $Y^{[2]}$ such that $\beta_{13} \circ \alpha = \alpha' \circ \beta_{23} \circ \beta_{12}$ over $Y^{[3]}$. An \emph{isomorphism of gerbes} is a smooth functor $\gamma:\mathcal{L} \rightarrow \mathcal{L}'$ which can be constructed from an isomorphism of bundle gerbes over the same acyclic submersion by $\gamma(y) = y$ on objects and $\gamma(l) = \beta(l)$ on arrows.

    If $\beta, \, \beta': (Y,\pi,L,\alpha) \rightarrow (Y,\pi',L',\alpha')$ are isomorphisms of bundle gerbes over the same acyclic submersion, then a \emph{2-isomorphism of bundle gerbes} $f: \beta \Rightarrow \beta'$ is a function $f: Y \rightarrow U(1)$ such that $f_1 \cdot \beta = f_2 \cdot \beta'$ over $Y^{[2]}$. A \emph{2-isomorphism of gerbes} is a smooth natural transformation $\psi:\gamma \Rightarrow \gamma':\mathcal{L} \rightarrow \mathcal{L}'$ constructed from a 2-isomorphism of bundle gerbes by $\psi(y) = f(y) \cdot \tau'(y)$.

    \emph{Pullback} of gerbes is defined by pulling back all their constituent data. \emph{Tensor product} of gerbes is defined by taking the tensor products of all $U(1)$-bundles and isomorphisms between them.

    Given a manifold $X$, fix an acyclic submersion $Y \rightarrow X$. Gerbes defined over this submersion, their isomorphisms and 2-isomorphisms form a 2-subcategory of $\mathsf{Cat}$ which does not depend on $Y$, as it is actually equivalent to 
    \begin{equation}\label{eq:gxcohomology}
    \mathcal{G}(X) \cong H^3(X,\mathbb Z) \times BH^2(X,\mathbb Z) \times B^2C^{\infty}(X,U(1)).
    \end{equation} 

    \begin{remark}
        The auxiliary choice of a fixed, acyclic submersion to define $\mathcal{G}(X)$ is made here to avoid the more technical notion of a bibundle between Lie groupoids, necessary to model isomorphisms of gerbes defined over different and possibly non-acyclic submersions \cite{Murray:9908135,Nikolaus:2011ag}.
    \end{remark}

    \section{Definitions regarding Lie 3-groups}\label{sec:lie3}

    An algebraic model for weak set-theoretic 3-groups can be found in \cite{Gurski:2006} (based on \cite{Gordon:1995tr}), where it is also proved that every such 3-group is equivalent to a \emph{Gray} 3-group or, equivalently, a 2-crossed module. Since this equivalence relies on the axiom of choice, it is no longer true in the case of weak \emph{Lie} 3-groups. 
    
    In order to present the most general definition of weak Lie 3-groups, one must work with higher stacks or simplicial models, as in \cite{Zhu:0801.2057}. Here we introduce a family of Lie 3-groups that are weaker than Gray Lie 3-groups (i.e. Lie 2-crossed modules) but still strict enough to be suitable for an algebraic description; this will be sufficient for the purposes of this paper. For notational convenience, we will refer to these simply as \emph{Lie 3-groups}, while acknowledging the restricted scope of our definition.
    \pagebreak 
    \subsection{Lie 3-groups}
    
    \begin{definition}
        A \emph{Lie 3-group} is
        \begin{enumerate}
            \item A Lie 2-groupoid $\mathcal{G}$.
            
            \item For each $g \in \mathcal{G}_0$, \emph{left} and \emph{right multiplication} smooth pseudofunctors $L_g: \mathcal{G} \rightarrow \mathcal{G}$ and $R_g: \mathcal{G} \rightarrow \mathcal{G}$ such that $L_{g_1}(g_2) = R_{g_2}(g_1) =: g_1g_2$. In particular, for $g_i \stackrel{\gamma_i}{\rightarrow} g_i' \stackrel{\gamma_i'}{\rightarrow} g_i''$, $i=1, \,2$, there are \emph{compositor} 2-cells
            \begin{equation}
                \begin{tikzcd}[column sep = 10ex]
            g_1g_2 \ar[rr,bend left=60,"{(\gamma_1' \circ \gamma_1)g_2}"{name=F}] 
              \ar[rr,bend right=10,"{(\gamma_1'g_2) \circ (\gamma_1g_2)}"{name=G},swap]
             & & g_1''g_2 \\ \ar[Rightarrow,from=F,to=G,"{R_{g_2}(\gamma_1,\gamma_1')}"{swap,pos=0.5},shorten >=1.5pt] &
                \end{tikzcd},
                \begin{tikzcd}[column sep = 10ex]
            g_1g_2 \ar[rr,bend left=60,"{g_1(\gamma_2' \circ \gamma_2)}"{name=F}] 
              \ar[rr,bend right=10,"{(g_1\gamma_2') \circ (g_1\gamma_2)}"{name=G},swap]
             & & g_1g_2'' \\ \ar[Rightarrow,from=F,to=G,"{L_{g_1}(\gamma_2,\gamma_2')}"{swap,pos=0.5},shorten >=1.5pt] &
                \end{tikzcd}.
            \end{equation}
            
            \item For $\gamma_i:g_i \rightarrow g_i' \in \mathcal{G}_1$, $i = 1,\,2$, \emph{interchange} 2-cells
            \begin{equation}
            \begin{tikzcd}[column sep=14ex]
             g_1g_2 \ar[r,"{\gamma_1g_2}",{name=U}] \ar[d,"{g_1\gamma_2}",swap] 
             & g_1'g_2 \ar[d,"{g_1'\gamma_2}"] \ar[Rightarrow, dl,"{I(\gamma_1,\gamma_2)}",swap]\\
             g_1g_2' \ar[r,"{\gamma_1g_2'}",{name=D},swap] 
             & g_1'g_2'
            \end{tikzcd}
            \end{equation}
            natural on 2-cells and respecting composition of arrows.

            \item For $g_i \in \mathcal{G}_0$, $i=1, \,2, \,3$, \emph{associator} $1$-arrows
            \begin{equation}
            \begin{tikzcd}[column sep=14ex]
                (g_1g_2)g_3 \ar[r, "{\alpha(g_1,g_2,g_3)}"] & g_1(g_2g_3)
            \end{tikzcd}
            \end{equation}
            and for $\gamma_i:g_i \rightarrow g_i' \in \mathcal{G}_1$, $i = 1,\,2,\,3$, \emph{associator} 2-cells
            \begin{equation}
            \begin{tikzcd}[column sep=14ex]
             (g_1g_2)g_3 \ar[r,"{\alpha(g_1,g_2,g_3)}",{name=U}] \ar[d,"{(\gamma_1g_2)g_3}",swap] 
             & g_1(g_2g_3) \ar[d,"{\gamma_1(g_2g_3)}"] \\
             (g_1'g_2)g_3 \ar[r,"{\alpha(g_1',g_2,g_3)}",{name=D},swap]  \ar[Rightarrow, ur,"{\alpha(\gamma_1,g_2,g_3)}"]
             & g_1'(g_2g_3)
            \end{tikzcd},
            \end{equation}
            \begin{equation}
            \begin{tikzcd}[column sep=14ex]
             (g_1g_2)g_3 \ar[r,"{\alpha(g_1,g_2,g_3)}",{name=U}] \ar[d,"{(g_1\gamma_2)g_3}",swap] 
             & g_1(g_2g_3) \ar[d,"{g_1(\gamma_2g_3)}"] \\
             (g_1g_2')g_3 \ar[r,"{\alpha(g_1,g_2',g_3)}",{name=D},swap]  \ar[Rightarrow, ur,"{\alpha(g_1,\gamma_2,g_3)}"]
             & g_1(g_2'g_3)
            \end{tikzcd},
            \end{equation}
            \begin{equation}
            \begin{tikzcd}[column sep=14ex]
             (g_1g_2)g_3 \ar[r,"{\alpha(g_1,g_2,g_3)}",{name=U}] \ar[d,"{(g_1g_2)\gamma_3}",swap] 
             & g_1(g_2g_3) \ar[d,"{g_1(g_2\gamma_3)}"] \\
             (g_1g_2)g_3' \ar[r,"{\alpha(g_1,g_2,g_3')}",{name=D},swap]  \ar[Rightarrow, ur,"{\alpha(g_1,g_2,\gamma_3)}"]
             & g_1(g_2g_3')
            \end{tikzcd}
            \end{equation}
            natural on 2-cells and respecting composition of arrows and interchange cells.
            
            \item For $g_i \in \mathcal{G}_0$, $i = 1, \,...,\,4$, \emph{pentagonator} 2-cells
                \begin{equation}
                \begin{tikzcd}
	                    & ((g_1g_2)g_3)g_4 \\
	            (g_1(g_2g_3))g_4 &        & (g_1g_2)(g_3g_4) \\
	            g_1((g_2g_3)g_4) &        & g_1(g_2(g_3g_4)) 
	           \arrow[from=1-2, to=2-1, "{\alpha(g_1,g_2,g_3)g_4}",swap]
	           \arrow[from=2-1, to=3-1, "{\alpha(g_1,g_2g_3,g_4)}",swap]
	           \arrow[from=3-1, to=3-3, "{g_1\alpha(g_2,g_3,g_4)}"{name=A},swap]
	           \arrow[from=1-2, to=2-3, "{\alpha(g_1g_2,g_3,g_4)}"]
	           \arrow[from=2-3, to=3-3, "{\alpha(g_1,g_2,g_3g_4)}"]
                \arrow[Rightarrow,from=3-1,to=2-3,"{\Sigma(g_1,g_2,g_3,g_4)}"]
            \end{tikzcd}
                \end{equation}
            which are natural on arrows and satisfy the \emph{associahedron identity}
            \begin{align}
            \begin{split}
            \begin{tikzcd}[ampersand replacement=\&, every label/.append style = {font = \footnotesize}, every matrix/.append style={nodes={font=\footnotesize}}, row sep = 6ex]
	{((g_1g_2)g_3)(g_4g_5)} \&\& {(g_1g_2)(g_3(g_4g_5))} \\
	{(((g_1g_2)g_3)g_4)g_5} \& {((g_1g_2)(g_3g_4))g_5} \& {(g_1g_2)((g_3g_4)g_5)} \& [2ex] {g_1(g_2(g_3(g_4g_5)))} \\
	{((g_1(g_2g_3))g_4)g_5} \&\& {g_1(g_2((g_3g_4)g_5))} \\
	{(g_1((g_2g_3)g_4))g_5} \& {(g_1(g_2(g_3g_4)))g_5} \& {g_1((g_2(g_3g_4))g_5)}
	\arrow["{(\alpha(g_1,g_2,g_3)g_4)g_5}"', from=2-1, to=3-1]
	\arrow["{\alpha(g_1,g_2g_3,g_4)g_5}"', from=3-1, to=4-1]
	\arrow["{(g_1\alpha(g_2,g_3,g_4))g_5}"', from=4-1, to=4-2]
	\arrow["{\alpha(g_1,g_2(g_3g_4),g_5)}"', from=4-2, to=4-3]
	\arrow["{\alpha(g_1g_2,g_3,g_4g_5)}"{name=a}, from=1-1, to=1-3]
	\arrow["{\alpha(g_1,g_2,g_3(g_4g_5))}"{name=p}, from=1-3, to=2-4]
	\arrow[bend right = 40, "{g_1(g_2(\alpha(g_3,g_4,g_5)))}"'{name=A}, from=3-3, to=2-4]
	\arrow["{g_1\alpha(g_2,g_3g_4,g_5)}"', from=4-3, to=3-3]
	\arrow["{\alpha((g_1g_2)g_3,g_4,g_5)}", from=2-1, to=1-1]
	\arrow[from=2-3, to=3-3]
	\arrow[from=2-3, to=1-3]
	\arrow[from=2-1, to=2-2]
	\arrow[from=2-2, to=4-2]
	\arrow[from=2-2, to=2-3]
    \arrow["{\Sigma(g_1g_2,g_3,g_4,g_5)}", Rightarrow, from=2-2, to=a,pos=0.6]
        \arrow["{\widetilde{\Sigma}(g_1,g_2,g_3,g_4)g_5}", Rightarrow, from=4-1, to=2-2,pos=0.85]
        \arrow["{\Sigma(g_1,g_2,g_3g_4,g_5)}", Rightarrow, from=4-2, to=2-3,pos=0.85]
        \arrow["{\alpha(g_1,g_2,\alpha(g_3,g_4,g_5))^{-1}}", Rightarrow, from=A, to=1-3,shorten=4ex,sloped,swap,pos=0.6]
\end{tikzcd} = \\
\begin{tikzcd}[column sep= 6ex, row sep = 7ex, ampersand replacement=\&, every label/.append style = {font = \footnotesize}, every matrix/.append style={nodes={font=\footnotesize}}]
	\&\& {(((g_1g_2)g_3)g_4)g_5} \\
	\& {((g_1(g_2g_3))g_4)g_5} \&\& {((g_1g_2)g_3)(g_4g_5)} \\
	\& {(g_1((g_2g_3)g_4))g_5} \& {(g_1(g_2g_3))(g_4g_5)} \& {(g_1g_2)(g_3(g_4g_5))} \\
	\& {g_1(((g_2g_3)g_4)g_5)} \& {g_1((g_2g_3)(g_4g_5))} \& {g_1(g_2(g_3(g_4g_5)))} \\
	{(g_1(g_2(g_3g_4)))g_5} \&\& {g_1((g_2(g_3g_4))g_5)} \& {g_1(g_2((g_3g_4)g_5))}
	\arrow["{(\alpha(g_1,g_2,g_3)g_4)g_5}"', from=1-3, to=2-2]
	\arrow["{\alpha(g_1,g_2g_3,g_4)g_5}"', from=2-2, to=3-2]
	\arrow[from=3-2, to=4-2]
	\arrow["{}", from=4-2, to=4-3]
	\arrow["{\alpha(g_1g_2,g_3,g_4g_5)}", from=2-4, to=3-4]
	\arrow["{\alpha(g_1,g_2,g_3(g_4g_5))}"{name=t}, from=3-4, to=4-4]
	\arrow["{\alpha((g_1g_2)g_3,g_4,g_5)}", from=1-3, to=2-4]
	\arrow[from=2-2, to=3-3]
	\arrow[from=2-4, to=3-3]
	\arrow[""{name=s},from=3-3, to=4-3]
	\arrow[from=4-2, to=5-3]
	\arrow["{g_1\alpha(g_2,g_3g_4,g_5)}"', from=5-3, to=5-4]
	\arrow[from=4-3, to=4-4]
	\arrow["{g_1(g_2\alpha(g_3,g_4,g_5))}"', from=5-4, to=4-4]
	\arrow["{(g_1\alpha(g_2,g_3,g_4))g_5}"', from=3-2, to=5-1]
	\arrow["{\alpha(g_1,g_2(g_3g_4),g_5)}"', from=5-1, to=5-3]
        \arrow[Rightarrow, "{\Sigma(g_1,g_2g_3,g_4,g_5)}", from = 4-2, to = s, pos= 0.7]
        \arrow[Rightarrow, "{\Sigma(g_1,g_2,g_3,g_4g_5)}", from = 4-3, to = t, pos= 0.7]
        \arrow[Rightarrow, "{g_1\widetilde{\Sigma}(g_2,g_3,g_4,g_5)}", from = 5-3, to = 4-4, pos= 0.5]
        \arrow[Rightarrow, "{\alpha(\alpha(g_1,g_2,g_3),g_4,g_5)}", from = 2-2, to = 2-4, pos= 0.5]
        \arrow[Rightarrow, "{\alpha(g_1,\alpha(g_2,g_3,g_4),g_5)}", from = 5-1, to = 4-2, pos= 0.5,swap]
\end{tikzcd},
            \end{split} 
            \end{align}
            where $\widetilde{\Sigma}(g_1,g_2,g_3,g_4)g_5$ and $g_1\widetilde{\Sigma}(g_2,g_3,g_4,g_5)$ stand respectively for

            \begin{align}
            \begin{split}
                \begin{tikzcd}[row sep = 7ex, column sep = 15ex, every label/.append style = {font = \tiny}, ampersand replacement = \&]
	                    \& (((g_1g_2)g_3)g_4)g_5 \\
	            ((g_1(g_2g_3))g_4)g_5 \&        \& ((g_1g_2)(g_3g_4))g_5 \\
                \\
	            (g_1((g_2g_3)g_4))g_5 \&        \& (g_1(g_2(g_3g_4)))g_5 
	           \arrow[from=1-2, to=2-1, "{(\alpha(g_1,g_2,g_3)g_4)g_5}",swap]
	           \arrow[from=2-1, to=4-1, bend right = 40, "{\alpha(g_1,g_2g_3,g_4)g_5}"{name=P},swap]
	           \arrow[from=4-1, to=4-3, "{(g_1\alpha(g_2,g_3,g_4))g_5}"{name=A},swap]
	           \arrow[from=1-2, to=2-3, bend left = 50, looseness = 0.75, "{\alpha(g_1g_2,g_3,g_4)g_5}"{name=C},pos=0.65]
	           \arrow[from=2-3, to=4-3, "{\alpha(g_1,g_2,g_3g_4)g_5}"]
                \arrow["{}"{name=A}, bend right = 40, looseness = 1.75, from =1-2, to = 4-3]
                \arrow["{}"{name=B}, bend right=20, from=1-2, to = 4-3, pos= 0.1]
                \arrow[Rightarrow,from=A,to=B,"{\Sigma(g_1,g_2,g_3,g_4)g_5}",swap,pos=0,shorten=1ex]
                \arrow[Rightarrow, "{R_{g_5}(\alpha(g_1g_2,g_3,g_4),\alpha(g_1,g_2,g_3g_4))}" ,from=B, to = C,shorten = 3ex,swap,pos=0]
                \arrow[Rightarrow, from=P, to=A, "{R_{g_5}(\alpha(g_1,g_2,g_3)g_4,\alpha(g_1,g_2g_3,g_4),g_1\alpha(g_2,g_3,g_4))^{-1}}",shorten=1ex,pos=0]
            \end{tikzcd}, \\
            \begin{tikzcd}[row sep = 7ex, column sep = 15ex, every label/.append style = {font = \tiny}, ampersand replacement = \&]
	                    \& g_1(((g_2g_3)g_4)g_5) \\
	            g_1((g_2(g_3g_4))g_5) \&        \& g_1((g_2g_3)(g_4g_5)) \\
                \\
	            g_1(g_2((g_3g_4)g_5)) \&        \& g_1(g_2(g_3(g_4g_5))) 
	           \arrow[from=1-2, to=2-1, "{g_1((\alpha(g_2,g_3,g_4)g_5))}",swap]
	           \arrow[from=2-1, to=4-1, bend right = 40, "{g_1\alpha(g_2,g_3g_4,g_5)}"{name=P},swap]
	           \arrow[from=4-1, to=4-3, "{g_1(g_2\alpha(g_3,g_4,g_5))}"{name=A},swap]
	           \arrow[from=1-2, to=2-3, bend left = 50, looseness = 0.75, "{g_1\alpha(g_2g_3,g_4,g_5)}"{name=C},pos=0.65]
	           \arrow[from=2-3, to=4-3, "{g_1\alpha(g_2,g_3,g_4g_5)}"]
                \arrow["{}"{name=A}, bend right = 40, looseness = 1.75, from =1-2, to = 4-3]
                \arrow["{}"{name=B}, bend right=20, from=1-2, to = 4-3, pos= 0.1]
                \arrow[Rightarrow,from=A,to=B,"{g_1\Sigma(g_2,g_3,g_4,g_5)}",swap,pos=0,shorten=1ex]
                \arrow[Rightarrow, "{L_{g_1}(\alpha(g_2g_3,g_4,g_5),\alpha(g_2,g_3,g_4g_5))}" ,from=B, to = C,shorten = 3ex,swap,pos=0]
                \arrow[Rightarrow, from=P, to=A, "{L_{g_1}(\alpha(g_2,g_3,g_4)g_5,\alpha(g_2,g_3g_4,g_5),g_2\alpha(g_3,g_4,g_5))^{-1}}",shorten=1ex,pos=0]
            \end{tikzcd}.
            \end{split}
            \end{align}
            
        \end{enumerate}
        A \emph{Gray Lie 3-group} is a Lie 3-group in which the  associator arrows and the compositor, associator and pentagonator 2-cells are identities. Note interchange 2-cells may not be identities. A \emph{Lie 2-group} is a Lie 3-group with only identity 2-cells.
    \end{definition}

    \begin{remark}
        In a Lie 2-group, the pseudofunctors $L_{g_1}$ and $R_{g_2}$ define a single smooth functor $m: \mathcal{G} \times \mathcal{G} \rightarrow \mathcal{G}$, and the associator arrows define a natural transformation $m \circ (m \times id) \Rightarrow m \circ (id \times m): \mathcal{G} \times \mathcal{G} \times \mathcal{G} \rightarrow \mathcal{G}$.
    \end{remark}
    
   We spell out in more detail the data that defines a Lie 3-group $\mathcal{G}$. The pseudofunctors $L_{g_1}, \,R_{g_2}: \mathcal{G} \rightarrow \mathcal{G}$ are given by smooth maps $L_{g_1}^i, R_{g_2}^i: \mathcal{G}_i \rightarrow \mathcal{G}_i$ preserving all source, target, identity and vertical composition maps. Their compositors must be natural on 2-cells; i.e., given $\psi_i:\gamma_i \Rightarrow \eta_i$, $\psi_i':\gamma_i' \Rightarrow \eta_i'$, we must have
    \begin{equation}
        \begin{tikzcd}[every label/.append style = {font = \tiny},column sep = 10ex]
            g_1g_2 \ar[rr,bend left=60,"{(\gamma_1' \circ \gamma_1)g_2}"{name=F}] 
              \ar[rr,bend left = 10,"{(\gamma_1'g_2) \circ (\gamma_1g_2)}"{name=G},swap]
              \ar[rr,bend right=80,"{(\eta_1'g_2) \circ (\eta_1g_2)}"{name=H},swap]
             & & g_1''g_2 \\ \ar[Rightarrow,from=F,to=G,"{R_{g_2}(\gamma_1,\gamma_1')}"{swap,pos=0.6},shorten >=1.5pt] \ar[Rightarrow,from=G,to=H,"{(\psi_1'g_2) \circ (\psi_1g_2)}"{swap,pos=0.25},shorten >=1.5pt] &
        \end{tikzcd} = 
        \begin{tikzcd}[every label/.append style = {font = \tiny},column sep = 10ex]
            g_1g_2 \ar[rr,bend left=60,"{(\gamma_1' \circ \gamma_1)g_2}"{name=F}] 
              \ar[rr,bend left = 10,"{(\eta_1' \circ \eta_1)g_2)}"{name=G},swap]
              \ar[rr,bend right=80,"{(\eta_1'g_2) \circ (\eta_1g_2)}"{name=H},swap]
             & & g_1''g_2 \\ \ar[Rightarrow,from=F,to=G,"{(\psi_1' \circ \psi_1)g_2}"{swap,pos=0.55},shorten >=1.5pt] \ar[Rightarrow,from=G,to=H,"{R_{g_2}(\eta_1,\eta_1')}"{swap,pos=0.35},shorten >=1.5pt] &
        \end{tikzcd},
    \end{equation}
    \begin{equation}
        \begin{tikzcd}[every label/.append style = {font = \tiny},column sep = 10ex]
            g_1g_2 \ar[rr,bend left=60,"{g_1(\gamma_2' \circ \gamma_2)}"{name=F}] 
              \ar[rr,bend left = 10,"{(g_1\gamma_2') \circ (g_1\gamma_2)}"{name=G},swap]
              \ar[rr,bend right=80,"{(g_1\eta_2') \circ (g_1\eta_2)}"{name=H},swap]
             & & g_1g_2'' \\ \ar[Rightarrow,from=F,to=G,"{L_{g_1}(\gamma_2,\gamma_2')}"{swap,pos=0.6},shorten >=1.5pt] \ar[Rightarrow,from=G,to=H,"{(g_1\psi_2') \circ (g_1\psi_2)}"{swap,pos=0.25},shorten >=1.5pt] &
        \end{tikzcd} = 
        \begin{tikzcd}[every label/.append style = {font = \tiny},column sep = 10ex]
            g_1g_2 \ar[rr,bend left=60,"{g_1(\gamma_2' \circ \gamma_2)}"{name=F}] 
              \ar[rr,bend left = 10,"{g_1(\eta_2' \circ \eta_2)}"{name=G},swap]
              \ar[rr,bend right=80,"{(g_1\eta_2') \circ (g_1\eta_2)}"{name=H},swap]
             & & g_1g_2'' \\ \ar[Rightarrow,from=F,to=G,"{g_1(\psi_2' \circ \psi_2)}"{swap,pos=0.5},shorten >=1.5pt] \ar[Rightarrow,from=G,to=H,"{L_{g_1}(\eta_2,\eta_2')}"{swap,pos=0.35},shorten >=1.5pt] &
        \end{tikzcd}
    \end{equation}
    and they must be associative in the sense that given $g_i \stackrel{\gamma_i}{\rightarrow} g_i' \stackrel{\gamma_i'}{\rightarrow} g_i'' \stackrel{\gamma_i''}{\rightarrow} g_i'''$ then
    \begin{equation}
        \begin{tikzcd}[every label/.append style = {font = \tiny},column sep = 10ex]
            g_1g_2 \ar[rr,bend left=75,"{(\gamma_1'' \circ \gamma_1' \circ \gamma_1)g_2}"{name=F},pos=0.5] 
              \ar[rr,bend left = 10,"{((\gamma_1''\circ \gamma_1') g_2) \circ (\gamma_1g_2)}"{name=G},swap,pos=0.60]
              \ar[rr,bend right=80,"{(\gamma_1''g_2 \circ \gamma_1'g_2) \circ (\gamma_1g_2)}"{name=H},swap]
             & & g_1'''g_2 \\ \ar[Rightarrow,from=F,to=G,"{R_{g_2}(\gamma_1'' \circ \gamma_1',\gamma_1)}"{swap,pos=0.55},shorten >=1.5pt] \ar[Rightarrow,from=G,to=H,"{R_{g_2}(\gamma_1'',\gamma_1') \circ id}"{swap,pos=0.5},shorten >=1.5pt] &
        \end{tikzcd} = 
        \begin{tikzcd}[every label/.append style = {font = \tiny},column sep = 10ex]
            g_1g_2 \ar[rr,bend left=75,"{(\gamma_1'' \circ \gamma_1' \circ \gamma_1)g_2}"{name=F}] 
              \ar[rr,bend left = 10,"{(\gamma_1'' g_2) \circ (\gamma_1' \circ \gamma_1)g_2}"{name=G},swap,pos=0.6]
              \ar[rr,bend right=80,"{(\gamma_1''g_2) \circ (\gamma_1'g_2 \circ \gamma_1g_2)}"{name=H},swap]
             & & g_1'''g_2 \\ \ar[Rightarrow,from=F,to=G,"{R_{g_2}(\gamma_1'', \gamma_1'\circ \gamma_1)}"{swap,pos=0.55},shorten >=1.5pt] \ar[Rightarrow,from=G,to=H,"{id \circ R_{g_2}(\gamma_1',\gamma_1)}"{swap,pos=0.5},shorten >=1.5pt] &
        \end{tikzcd} =: R_{g_2}(\gamma_1,\gamma_1',\gamma_1''),
    \end{equation}
    \begin{equation}
        \begin{tikzcd}[every label/.append style = {font = \tiny},column sep = 10ex]
            g_1g_2 \ar[rr,bend left=75,"{g_1(\gamma_2'' \circ \gamma_2' \circ \gamma_2)}"{name=F},pos=0.5] 
              \ar[rr,bend left = 10,"{(g_1(\gamma_2''\circ \gamma_2')) \circ (g_1\gamma_2)}"{name=G},swap,pos=0.60]
              \ar[rr,bend right=80,"{(g_1\gamma_2'' \circ g_1\gamma_2') \circ (g_1\gamma_2)}"{name=H},swap]
             & & g_1g_2''' \\ \ar[Rightarrow,from=F,to=G,"{L_{g_1}(\gamma_2'' \circ \gamma_2',\gamma_2)}"{swap,pos=0.55},shorten >=1.5pt] \ar[Rightarrow,from=G,to=H,"{L_{g_1}(\gamma_2'',\gamma_2') \circ id}"{swap,pos=0.5},shorten >=1.5pt] &
        \end{tikzcd} = 
        \begin{tikzcd}[every label/.append style = {font = \tiny},column sep = 10ex]
            g_1g_2 \ar[rr,bend left=75,"{g_1(\gamma_2'' \circ \gamma_2' \circ \gamma_2)}"{name=F}] 
              \ar[rr,bend left = 10,"{(g_1 \gamma_2') \circ g_1(\gamma_2' \circ \gamma_2)}"{name=G},swap,pos=0.6]
              \ar[rr,bend right=80,"{(g_1\gamma_2'') \circ (g_1\gamma_2' \circ g_1\gamma_2)}"{name=H},swap]
             & & g_1g_2''' \\ \ar[Rightarrow,from=F,to=G,"{L_{g_1}(\gamma_2'', \gamma_2'\circ \gamma_2)}"{swap,pos=0.55},shorten >=1.5pt] \ar[Rightarrow,from=G,to=H,"{id \circ L_{g_1}(\gamma_2',\gamma_2)}"{swap,pos=0.5},shorten >=1.5pt] &
        \end{tikzcd} =: \quad L_{g_1}(\gamma_2,\gamma_2',\gamma_2'').
    \end{equation}
    Naturality for the interchange 2-cells $I(\gamma_1,\gamma_2)$ means that, for given $\psi_i: \gamma_i \Rightarrow \eta_i: g_i \rightarrow g_i'$, $i = 1, \,2$, we have
    \begin{equation}
        \begin{tikzcd}[column sep=14ex]
             g_1g_2 \ar[r,"{\gamma_1g_2}",{name=U}] \ar[d,"{}"{name=E}] \ar[d,bend right=100, "{g_1\eta_2}"{name=F},swap]  \ar[Rightarrow,from=E,to=F,"{g_1\psi_2}"]
             & g_1'g_2 \ar[d,"{g_1'\gamma_2}"] \ar[Rightarrow, dl,"{I(\gamma_1,\gamma_2)}",swap]\\
             g_1g_2' \ar[r,"{}"{name=D}] \ar[r,bend right=30, "{\eta_1g_2'}"{name=G},swap] \ar[Rightarrow,from=D,to=G,"{\psi_1g_2'}"]
             & g_1'g_2' 
            \end{tikzcd} = 
         \begin{tikzcd}[column sep=14ex]
             g_1g_2 \ar[r,bend left=40, "{\gamma_1g_2}"{name=F}] \ar[r,"{}"{name=U}] \ar[d,"{g_1\eta_2}"{name=E},swap]   \ar[Rightarrow,from=F,to=U,"{\psi_1g_2}",pos=0.9]
             & g_1'g_2 \ar[d,"{}"{name=R}] \ar[Rightarrow, dl,"{I(\eta_1,\eta_2)}",swap]   \ar[d,bend left=100, "{g_1'\gamma_2}"{name=G}] \ar[Rightarrow,from=G,to=R,"{g_1'\psi_2}"]\\
             g_1g_2' \ar[r,"{\eta_1g_2'}"{name=D},swap] 
             & g_1'g_2' 
            \end{tikzcd}
    \end{equation}
    On the other hand, respecting composition of arrows means that given $g_i \stackrel{\gamma_i}{\rightarrow} g_i' \stackrel{\gamma_i'}{\rightarrow} g_i''$, $i=1, \,2$, we have
    \begin{equation}
        \begin{tikzcd}[column sep=14ex]
            & \ar[Rightarrow,d,"{R_{g_2}(\gamma_1',\gamma_1)}"] & \\
             g_1g_2 \ar[rr,bend left = 40,"{(\gamma_1' \circ \gamma_1)g_2}"{name=A}] \ar[r,"{\gamma_1g_2}",{name=U}] \ar[d,"{g_1\gamma_2}",swap] 
             & g_1'g_2  \ar[d,"{g_1'\gamma_2}",swap] \ar[Rightarrow, dl,"{I(\gamma_1,\gamma_2)}",swap] \ar[r,"{\gamma_1'g_2}",{name=U2}] & g_1''g_2  \ar[d,"{g_1''\gamma_2}"] \ar[Rightarrow, dl,"{I(\gamma_1',\gamma_2)}",swap] \\
             g_1g_2' \ar[r,"{\gamma_1g_2'}",{name=D},swap] 
             & g_1'g_2' \ar[r,"{\gamma_1'g_2'}",{name=D2},swap]  & g_1''g_2'
            \end{tikzcd} = 
        \begin{tikzcd}[column sep=14ex]
             g_1g_2 \ar[r,"{(\gamma_1' \circ \gamma_1)g_2}"{name=A}] \ar[d,"{g_1\gamma_2}",swap] 
             & g_1''g_2  \ar[d,"{g_1''\gamma_2}"] \ar[Rightarrow, dl,"{I(\gamma_1' \circ \gamma_1,\gamma_2)}",swap] \\
             g_1g_2' \ar[r,"{(\gamma_1' \circ \gamma_1)g_2'}",{name=D},swap,pos=0.3] \ar[r,bend right = 80, "{\gamma_1'g_2' \circ \gamma_1g_2'}"{name=T},swap,pos=0.51] \ar[Rightarrow,from=D,to=T,"{R_{g_2'}(\gamma_1',\gamma_1)}",pos=0.2]
             & g_1''g_2'
            \end{tikzcd} ,     
    \end{equation}
    \begin{equation}
        \begin{tikzcd}[every label/.append style = {font = \tiny},column sep=14ex]
             g_1g_2 \ar[r,"{\gamma_1g_2}",{name=U}] \ar[d,"{g_1\gamma_2}",swap] 
             & g_1'g_2  \ar[d,"{}"] \ar[Rightarrow, dl,"{I(\gamma_1,\gamma_2)}",swap] \ar[dd,bend left=80,looseness = 1.5,"{g_1'(\gamma_2'\circ \gamma_2)}",pos=0.5] &[-10ex] \\
             g_1g_2' \ar[r,"{\gamma_1g_2'}",{name=D}] \ar[d,"{g_1\gamma_2'}",swap]
             & g_1'g_2' \ar[d,"{}"] \ar[Rightarrow, dl,"{I(\gamma_1,\gamma_2')}",swap] & [-10ex] \ar[Rightarrow,l,"{L_{g_1'}(\gamma_2,\gamma_2')}",swap]\\
             g_1g_2'' \ar[r,"{\gamma_1g_2''}",{name=2},swap] 
             & g_1'g_2'' & [-10ex]
            \end{tikzcd} = 
        \begin{tikzcd}[every label/.append style = {font = \tiny},column sep=14ex,row sep = 10ex]
             g_1g_2 \ar[r,"{\gamma_1 g_2}"{name=A}] \ar[d,"{}"{name=X},swap] \ar[d,bend right=80,looseness=2,"{g_1\gamma_2' \circ g_1\gamma_2}"{name=Y},swap] \ar[Rightarrow,from=X,to=Y,"{L_{g_1}(\gamma_2,\gamma_2')}",swap]
             & g_1'g_2  \ar[d,"{g_1'(\gamma_2' \circ \gamma_2)}"] \ar[Rightarrow, dl,"{I(\gamma_1,\gamma_2' \circ \gamma_2)}",swap] \\
             g_1g_2'' \ar[r,"{\gamma_1g_2''}",{name=D},swap]
             & g_1'g_2''
            \end{tikzcd} .  
    \end{equation}
    For the associator 2-cells, naturality on 2-cells means that for given $\psi_i: \gamma_i \Rightarrow \eta_i: g_i \rightarrow g_i'$, $i = 1, \,2, \,3$ we have
    \begin{equation}
            \begin{tikzcd}[column sep=11ex]
             (g_1g_2)g_3 \ar[r,"{\alpha(g_1,g_2,g_3)}",{name=U}] \ar[d,"{}"{name=A}] \ar[d,bend right= 100,"{(\gamma_1g_2)g_3)}"{name=B},swap] \ar[Rightarrow,from=B,to=A,"{(\psi_1g_2)g_3}"]
             & g_1(g_2g_3) \ar[d,"{\eta_1(g_2g_3)}"]  \\
             (g_1'g_2)g_3 \ar[r,"{\alpha(g_1',g_2,g_3)}",{name=D},swap] \ar[Rightarrow, ur,"{\alpha(\eta_1,g_2,g_3)}"]
             & g_1'(g_2g_3)
            \end{tikzcd} =
            \begin{tikzcd}[column sep=11ex]
             (g_1g_2)g_3 \ar[r,"{\alpha(g_1,g_2,g_3)}",{name=U}] \ar[d,"{(\gamma_1g_2)g_3}",swap] 
             & g_1(g_2g_3) \ar[d,"{}"{name=B}] \ar[d,bend left= 100,"{\eta_1(g_2g_3)}"{name=A}] \ar[Rightarrow,from=B,to=A,"{\psi_1(g_2g_3)}",pos=0.4] \\
             (g_1'g_2)g_3 \ar[r,"{\alpha(g_1',g_2,g_3)}",{name=D},swap] \ar[Rightarrow, ur,"{\alpha(\gamma_1,g_2,g_3)}"] 
             & g_1'(g_2g_3)
            \end{tikzcd},
            \end{equation}
    \begin{equation}
            \begin{tikzcd}[column sep=11ex]
             (g_1g_2)g_3 \ar[r,"{\alpha(g_1,g_2,g_3)}",{name=U}] \ar[d,"{}"{name=A}] \ar[d,bend right= 100,"{(g_1\gamma_2)g_3}"{name=B},swap] \ar[Rightarrow,from=B,to=A,"{(g_1\psi_2)g_3}"]
             & g_1(g_2g_3) \ar[d,"{g_1(\eta_2g_3)}"]  \\
             (g_1g_2')g_3 \ar[r,"{\alpha(g_1,g_2',g_3)}",{name=D},swap]  \ar[Rightarrow, ur,"{\alpha(g_1,\eta_2,g_3)}"]
             & g_1(g_2'g_3)
            \end{tikzcd} =
            \begin{tikzcd}[column sep=11ex]
             (g_1g_2)g_3 \ar[r,"{\alpha(g_1,g_2,g_3)}",{name=U}] \ar[d,"{(g_1\gamma_2)g_3}",swap] 
             & g_1(g_2g_3) \ar[d,"{}"{name=B}]  \ar[d,bend left= 100,"{g_1(\eta_2g_3)}"{name=A}] \ar[Rightarrow,from=B,to=A,"{g_1(\psi_2g_3)}",pos=0.4] \\
             (g_1g_2')g_3 \ar[r,"{\alpha(g_1,g_2',g_3)}",{name=D},swap] \ar[Rightarrow, ur,"{\alpha(g_1,\gamma_2,g_3)}"]
             & g_1(g_2'g_3)
            \end{tikzcd},
            \end{equation}
    \begin{equation}
            \begin{tikzcd}[column sep=11ex]
             (g_1g_2)g_3 \ar[r,"{\alpha(g_1,g_2,g_3)}",{name=U}] \ar[d,"{}"{name=A}] \ar[d,bend right= 100,"{(g_1g_2)\gamma_3}"{name=B},swap] \ar[Rightarrow,from=B,to=A,"{(g_1g_2)\psi_3}"]
             & g_1(g_2g_3) \ar[d,"{g_1(g_2\eta_3)}"]  \\
             (g_1g_2)g_3' \ar[r,"{\alpha(g_1,g_2,g_3')}",{name=D},swap]  \ar[Rightarrow, ur,"{\alpha(g_1,g_2,\eta_3)}"]
             & g_1(g_2g_3')
            \end{tikzcd} =
            \begin{tikzcd}[column sep=11ex]
             (g_1g_2)g_3 \ar[r,"{\alpha(g_1,g_2,g_3)}",{name=U}] \ar[d,"{(g_1g_2)\gamma_3}",swap] 
             & g_1(g_2g_3) \ar[d,"{}"{name=B}]  \ar[d,bend left= 100,"{g_1(g_2\eta_3)}"{name=A}] \ar[Rightarrow,from=B,to=A,"{g_1(g_2\psi_3)}",pos=0.4] \\
             (g_1g_2)g_3' \ar[r,"{\alpha(g_1,g_2,g_3')}",{name=D},swap] \ar[Rightarrow, ur,"{\alpha(g_1,g_2,\gamma_3)}"]
             & g_1(g_2g_3')
            \end{tikzcd}.
            \end{equation}
    Respecting composition of arrows means that for given $g_i \stackrel{\gamma_i}{\rightarrow} g_i' \stackrel{\gamma_i'}{\rightarrow} g_i''$, $i=1, \,2, \,3$ we have
    \begin{align}
    \begin{split}
            \begin{tikzcd}[column sep=14ex,ampersand replacement = \&, nodes in empty cells=true]
             (g_1g_2)g_3 \ar[r,"{\alpha(g_1,g_2,g_3)}",{name=U}] \ar[dd,"{((\gamma_1' \circ \gamma_1)g_2)g_3}",swap]  
             \& g_1(g_2g_3) \ar[dd,"{(\gamma_1' \circ \gamma_1)(g_2g_3)}",swap,pos=0.8]  \ar[dd,bend left = 100,looseness=1.5,"{\gamma_1'(g_2g_3) \circ \gamma_1(g_2g_3)}"{name=A}] \ar[Rightarrow,from=2-2,to=A,"{R_{g_2g_3}(\gamma_1,\gamma_1')}",pos=0.5]\\
              \& \\
             (g_1''g_2)g_3 \ar[r,"{\alpha(g_1'',g_2,g_3)}",{name=A},swap] \ar[Rightarrow, uur,"{\alpha(\gamma_1' \circ \gamma_1,g_2,g_3)}"] 
             \& g_1''(g_2g_3)
            \end{tikzcd} = \\
            \begin{tikzcd}[column sep=14ex,row sep = 8ex, ampersand replacement = \&]
             (g_1g_2)g_3 \ar[r,"{\alpha(g_1,g_2,g_3)}",{name=U}] \ar[d,bend left = 50,"{(\gamma_1g_2)g_3}"{name=A},pos=0.3] \ar[dd,bend right = 60,looseness=0.75,"{}"{name=B},swap] \ar[dd,bend right = 90,looseness=2.2,"{((\gamma_1' \circ \gamma_1) g_2)g_3}"{name=C},swap] \ar[Rightarrow,from=B,to=A,"{R_{g_3}(\gamma_1g_2,\gamma_1'g_2)}",pos=0.5,shorten=2ex,sloped] \ar[Rightarrow,from=C,to=B,"{R_{g_2}(\gamma_1,\gamma_1')g_3}",pos=0.45]
             \& g_1(g_2g_3) \ar[d,"{\gamma_1(g_2g_3)}"] \\
             (g_1'g_2)g_3 \ar[d,bend left = 50,"{(\gamma_1'g_2)g_3}"{name=A},pos=0.3]  \ar[r] \ar[Rightarrow, ur,"{\alpha(\gamma_1,g_2,g_3)}",pos=0.4,swap] \& g_1'(g_2g_3) \ar[d,"{\gamma_1'(g_2g_3)}"]  \\
             (g_1''g_2)g_3 \ar[r,"{\alpha(g_1'',g_2,g_3)}",{name=D},swap]  \ar[Rightarrow, ur,"{\alpha(\gamma_1',g_2,g_3)}",pos=0.4,swap]
             \& g_1''(g_2g_3)
            \end{tikzcd},
    \end{split}
    \end{align}

    \begin{align}
    \begin{split}
            \begin{tikzcd}[column sep=18ex,ampersand replacement = \&, nodes in empty cells=true]
             (g_1g_2)g_3 \ar[r,"{\alpha(g_1,g_2,g_3)}",{name=U}] \ar[dd,"{(g_1(\gamma_2' \circ \gamma_2))g_3}",swap]  
             \& g_1(g_2g_3) \ar[dd,bend right = 50,"{g_1((\gamma_2' \circ \gamma_2)g_3)}"{name=x},swap,pos=0.7]  \ar[dd,bend left = 70,"{}"{name=A}] \ar[dd,bend left = 100, looseness = 3.5,"{g_1(\gamma_2'g_3) \circ g_1(\gamma_2g_3)}"{name=B}] \ar[Rightarrow,from=x,to=A,"{g_1R_{g_3}(\gamma_2,\gamma_2')}",pos=0.9] \ar[Rightarrow,from=A,to=B,"{L_{g_1}(\gamma_2g_3,\gamma_2'g_3)}",pos=0.45]\\
              \& \\
             (g_1g_2'')g_3 \ar[r,"{\alpha(g_1,g_2'',g_3)}",{name=A},swap] \ar[Rightarrow, uur,"{\alpha(g_1,\gamma_2' \circ \gamma_2,g_3)}"] 
             \& g_1(g_2''g_3)
            \end{tikzcd} = \\
            \begin{tikzcd}[column sep=14ex,row sep = 8ex, ampersand replacement = \&]
             (g_1g_2)g_3 \ar[r,"{\alpha(g_1,g_2,g_3)}",{name=U}] \ar[d,bend left = 50,"{(g_1\gamma_2)g_3}"{name=A},pos=0.3] \ar[dd,bend right = 60,looseness=0.75,"{}"{name=B},swap] \ar[dd,bend right = 90,looseness=2.2,"{(g_1(\gamma_2' \circ \gamma_2))g_3}"{name=C},swap] \ar[Rightarrow,from=B,to=A,"{R_{g_3}(g_1\gamma_2,g_1\gamma_2')}",pos=0.5,shorten=2ex,sloped] \ar[Rightarrow,from=C,to=B,"{L_{g_1}(\gamma_2,\gamma_2')g_3}",pos=0.45]
             \& g_1(g_2g_3) \ar[d,"{g_1(\gamma_2g_3)}"] \\
             (g_1g_2')g_3 \ar[d,bend left = 50,"{(g_1\gamma_2')g_3}"{name=A},pos=0.3]  \ar[r] \ar[Rightarrow, ur,"{\alpha(g_1,\gamma_2,g_3)}",pos=0.4,swap] \& g_1(g_2'g_3) \ar[d,"{g_1(\gamma_2'g_3)}"]  \\
             (g_1g_2'')g_3 \ar[r,"{\alpha(g_1,g_2'',g_3)}",{name=D},swap]  \ar[Rightarrow, ur,"{\alpha(g_1,\gamma_2',g_3)}",pos=0.4,swap]
             \& g_1(g_2''g_3)
            \end{tikzcd},
    \end{split}
    \end{align}

    \begin{align}
    \begin{split}
            \begin{tikzcd}[column sep=18ex,ampersand replacement = \&, nodes in empty cells=true]
             (g_1g_2)g_3 \ar[r,"{\alpha(g_1,g_2,g_3)}",{name=U}] \ar[dd,"{(g_1g_2)(\gamma_3' \circ \gamma_3)}",swap]  
             \& g_1(g_2g_3) \ar[dd,bend right = 50,"{g_1(g_2(\gamma_3' \circ \gamma_3))}"{name=x},swap,pos=0.7]  \ar[dd,bend left = 70,"{}"{name=A}] \ar[dd,bend left = 100, looseness = 3.5,"{g_1(g_2\gamma_3') \circ g_1(g_2\gamma_3)}"{name=B}] \ar[Rightarrow,from=x,to=A,"{g_1L_{g_2}(\gamma_3,\gamma_3')}",pos=0.9] \ar[Rightarrow,from=A,to=B,"{L_{g_1}(g_2\gamma_3,g_2\gamma_3')}",pos=0.45]\\
              \& \\
             (g_1g_2)g_3'' \ar[r,"{\alpha(g_1,g_2,g_3'')}",{name=A},swap] \ar[Rightarrow, uur,"{\alpha(g_1,g_2,\gamma_3' \circ \gamma_3)}"] 
             \& g_1(g_2g_3'')
            \end{tikzcd} = \\
            \begin{tikzcd}[column sep=10ex,row sep = 8ex, ampersand replacement = \&]
             (g_1g_2)g_3 \ar[r,"{\alpha(g_1,g_2,g_3)}",{name=U}] \ar[d,"{(g_1g_2)\gamma_3}"{name=A},pos=0.3,swap] \ar[dd,bend right = 80,looseness=1.5,"{(g_1g_2)(\gamma_3' \circ \gamma_3)}"{name=C},swap] \ar[Rightarrow,from=C,to=2-1,"{L_{g_1g_2}(\gamma_3,\gamma_3')}",pos=0.6]
             \& g_1(g_2g_3) \ar[d,"{g_1(g_2\gamma_3)}"] \\
             (g_1g_2)g_3' \ar[d,"{(g_1g_2)\gamma_3'}"{name=A},pos=0.3,swap]  \ar[r] \ar[Rightarrow, ur,"{\alpha(g_1,g_2,\gamma_3)}",pos=0.5] \& g_1(g_2g_3') \ar[d,"{g_1(g_2\gamma_3')}"]  \\
             (g_1g_2)g_3'' \ar[r,"{\alpha(g_1,g_2,g_3'')}",{name=D},swap]  \ar[Rightarrow, ur,"{\alpha(g_1,g_2,\gamma_3')}",pos=0.5]
             \& g_1(g_2g_3'')
            \end{tikzcd},
    \end{split}
    \end{align}
    
    and respecting the interchange cells means that given $\gamma_i:g_i \rightarrow g_i'$, $i=1, \,2,\, 3$ we have
    \begin{align}
    \begin{split}
    \begin{tikzcd}[column sep=10ex,row sep = 5ex, ampersand replacement = \&]
             (g_1g_2)g_3 \ar[r,"{\alpha(g_1,g_2,g_3)}",{name=U}] \ar[d,"{(\gamma_1g_2)g_3}"{name=A},pos=0.3,swap] \ar[dd,bend right = 100,looseness=1.75,"{(g_1'\gamma_2 \circ \gamma_1g_2)g_3}"{name=C},swap] \ar[Rightarrow,from=C,to=2-1,"{R_{g_3}(\gamma_1g_2,g_1'\gamma_2)}",pos=0.65]
             \& g_1(g_2g_3) \ar[d,"{\gamma_1(g_2g_3)}"] \ar[dd,bend left = 100,looseness=1.75,"{\gamma_1(g_2'g_3) \circ g_1(\gamma_2g_3)}"{name=D}] \ar[Rightarrow,from=2-2,to=D,"{I(\gamma_1,\gamma_2g_3)}",pos=0.5] \\
             (g_1'g_2)g_3 \ar[d,"{(g_1'\gamma_2)g_3}"{name=A},pos=0.3,swap]  \ar[r] \ar[Rightarrow, ur,"{\alpha(\gamma_1,g_2,g_3)}",pos=0.6] \& g_1'(g_2g_3) \ar[d,"{g_1'(\gamma_2g_3)}"]  \\
             (g_1'g_2')g_3 \ar[r,"{\alpha(g_1',g_2',g_3)}",{name=D},swap]  \ar[Rightarrow, ur,"{\alpha(g_1',\gamma_2,g_3)}",pos=0.6]
             \& g_1'(g_2'g_3)
            \end{tikzcd} = \\
    \begin{tikzcd}[column sep=14ex,row sep = 6ex, ampersand replacement = \&]
             (g_1g_2)g_3 \ar[r,"{\alpha(g_1,g_2,g_3)}",{name=U}] \ar[d,bend left = 50,"{(g_1\gamma_2)g_3}"{name=A},pos=0.3] \ar[dd,bend right = 80,looseness=1,"{}"{name=B},swap] \ar[dd,bend right = 110,looseness=2,"{(g_1'\gamma_2 \circ \gamma_1g_2)g_3}"{name=C},swap] \ar[Rightarrow,from=B,to=A,"{R_{g_3}(g_1\gamma_2,\gamma_1g_2')}",pos=0.5,shorten=2ex,sloped] \ar[Rightarrow,from=C,to=B,"{I(\gamma_1,\gamma_2)g_3}",pos=0.6]
             \& g_1(g_2g_3) \ar[d,"{g_1(\gamma_2g_3)}"] \\
             (g_1g_2')g_3 \ar[d,bend left = 50,"{(\gamma_1g_2')g_3}"{name=A},pos=0.3]  \ar[r] \ar[Rightarrow, ur,"{\alpha(g_1,\gamma_2,g_3)}",pos=0.4,swap] \& g_1(g_2'g_3) \ar[d,"{\gamma_1(g_2'g_3)}"]  \\
             (g_1'g_2')g_3 \ar[r,"{\alpha(g_1',g_2',g_3)}",{name=D},swap]  \ar[Rightarrow, ur,"{\alpha(\gamma_1,g_2',g_3)}",pos=0.4,swap]
             \& g_1'(g_2'g_3)
            \end{tikzcd} ,    
    \end{split}
    \end{align}

    \begin{align}
    \begin{split}
    \begin{tikzcd}[column sep=10ex,row sep = 5ex, ampersand replacement = \&]
             (g_1g_2)g_3 \ar[r,"{\alpha(g_1,g_2,g_3)}",{name=U}] \ar[d,"{(g_1g_2)\gamma_3}"{name=A},pos=0.3,swap] \ar[dd,bend right = 80,looseness=1.5,"{(g_1g_2')\gamma_3 \circ (g_1\gamma_2)g_3}"{name=C},swap] \ar[Rightarrow,from=C,to=2-1,"{I(g_1\gamma_2,\gamma_3)}",pos=0.65]
             \& g_1(g_2g_3) \ar[d,"{g_1(g_2\gamma_3)}"] \ar[dd,bend left = 100,looseness=2.35,"{g_1(\gamma_2g_3' \circ g_2\gamma_3)}"{name=D}] \ar[Rightarrow,from=2-2,to=D,"{L_{g_1}(g_2\gamma_3,\gamma_2g_3')^{-1}}",pos=0.45] \\
             (g_1g_2)g_3' \ar[d,"{(g_1\gamma_2)g_3'}"{name=A},pos=0.3,swap]  \ar[r] \ar[Rightarrow, ur,"{\alpha(g_1,g_2,\gamma_3)}",pos=0.6] \& g_1(g_2g_3') \ar[d,"{g_1(\gamma_2g_3')}"]  \\
             (g_1g_2')g_3' \ar[r,"{\alpha(g_1,g_2',g_3')}",{name=D},swap]  \ar[Rightarrow, ur,"{\alpha(g_1,\gamma_2,g_3')}",pos=0.6]
             \& g_1(g_2'g_3')
            \end{tikzcd} = \\
    \begin{tikzcd}[column sep=14ex,ampersand replacement = \&]
             (g_1g_2)g_3 \ar[r,"{\alpha(g_1,g_2,g_3)}",{name=U}] \ar[d,"{(g_1\gamma_2)g_3}",swap]  
             \& g_1(g_2g_3) \ar[d,"{g_1(\gamma_2g_3)}",swap]  \ar[dd,bend left = 90,looseness=3,"{}"{name=A}] \ar[dd,bend left = 100,looseness=5,"{g_1(\gamma_2g_3' \circ g_2\gamma_3)}"{name=B}]  \ar[Rightarrow,from=2-2,to=A,"{L_{g_1}(\gamma_2g_3,g_2'\gamma_3)^{-1}}",pos=0.4] \ar[Rightarrow,from=A,to=B,"{g_1I(\gamma_2,\gamma_3)}",pos=0.4]\\
             (g_1g_2')g_3 \ar[r,"{}",{name=D},swap] \ar[d,"{(g_1g_2')\gamma_3}",swap] \ar[Rightarrow, ur,"{\alpha(g_1,\gamma_2,g_3)}"]
             \& g_1(g_2'g_3) \ar[d,"{g_1(g_2'\gamma_3)}",swap] \\
             (g_1g_2')g_3' \ar[r,"{\alpha(g_1,g_2',g_3')}",{name=A},swap]  \ar[Rightarrow, ur,"{\alpha(g_1,g_2',\gamma_3)}"]
             \& g_1(g_2'g_3')
            \end{tikzcd},    
    \end{split}
    \end{align}

    \begin{align}
    \begin{split}
    \begin{tikzcd}[column sep=10ex,row sep = 8ex, ampersand replacement = \&]
             (g_1g_2)g_3 \ar[r,"{\alpha(g_1,g_2,g_3)}",{name=U}] \ar[d,"{(g_1g_2)\gamma_3}"{name=A},pos=0.3,swap] \ar[dd,bend right = 80,looseness=1.5,"{(g_1'g_2)\gamma_3 \circ (\gamma_1g_2)g_3}"{name=C},swap] \ar[Rightarrow,from=C,to=2-1,"{I(\gamma_1g_2,\gamma_3)}",pos=0.6]
             \& g_1(g_2g_3) \ar[d,"{g_1(g_2\gamma_3)}"] \\
             (g_1g_2)g_3' \ar[d,"{(\gamma_1g_2)g_3'}"{name=A},pos=0.3,swap]  \ar[r] \ar[Rightarrow, ur,"{\alpha(g_1,g_2,\gamma_3)}",pos=0.5] \& g_1(g_2g_3') \ar[d,"{\gamma_1(g_2g_3')}"]  \\
             (g_1'g_2)g_3' \ar[r,"{\alpha(g_1',g_2,g_3')}",{name=D},swap]  \ar[Rightarrow, ur,"{\alpha(\gamma_1,g_2,g_3')}",pos=0.5]
             \& g_1'(g_2g_3')
            \end{tikzcd} = \\
    \begin{tikzcd}[column sep=14ex,ampersand replacement = \&]
             (g_1g_2)g_3 \ar[r,"{\alpha(g_1,g_2,g_3)}",{name=U}] \ar[d,"{(\gamma_1g_2)g_3}",swap]  
             \& g_1(g_2g_3) \ar[d,"{\gamma_1(g_2g_3)}",swap]  \ar[dd,bend left = 90,looseness=1.5,"{\gamma_1(g_2g_3') \circ g_1(g_2\gamma_3)}"{name=A}] \ar[Rightarrow,from=2-2,to=A,"{I(\gamma_1,g_2\gamma_3)}",pos=0.4]\\
             (g_1'g_2)g_3 \ar[r,"{}",{name=D},swap] \ar[d,"{(g_1'g_2)\gamma_3}",swap] \ar[Rightarrow, ur,"{\alpha(\gamma_1,g_2,g_3)}"]
             \& g_1'(g_2g_3) \ar[d,"{g_1'(g_2\gamma_3)}",swap] \\
             (g_1'g_2)g_3' \ar[r,"{\alpha(g_1',g_2,g_3')}",{name=A},swap]  \ar[Rightarrow, ur,"{\alpha(g_1',g_2,\gamma_3)}"]
             \& g_1'(g_2g_3')
            \end{tikzcd}.  
    \end{split}
    \end{align}
    
    Finally, naturality on arrows for the pentagonator 2-cells means that for $\gamma_i:g_i \rightarrow g_i'$, $i=1, \,...,\, 4$ we have
    \begin{align}
    \begin{split}
    \begin{tikzcd}[ampersand replacement = \&]
	\& {((g_1'g_2)g_3)g_4} \& {((g_1g_2)g_3)g_4} \\
	{(g_1'(g_2g_3))g_4} \&\& {(g_1'g_2)(g_3g_4)} \& [3ex] {(g_1g_2)(g_3g_4)} \\
	{g_1'((g_2g_3)g_4)} \&\& {g_1'(g_2(g_3g_4))} \& [3ex] {g_1(g_2(g_3g_4))}
	\arrow[from=1-2, to=2-3]
	\arrow[from=2-3, to=3-3]
	\arrow["{((\gamma_1g_2)g_3)g_4}"', from=1-3, to=1-2]
	\arrow[from=2-4, to=2-3]
	\arrow["{\gamma_1(g_2(g_3g_4))}", from=3-4, to=3-3]
	\arrow["{\alpha(g_1g_2,g_3,g_4)}", from=1-3, to=2-4]
	\arrow["{\alpha(g_1',g_2,g_3)g_4}"', from=1-2, to=2-1]
	\arrow["{\alpha(g_1',g_2g_3,g_4)}"', from=2-1, to=3-1]
	\arrow["{g_1'\alpha(g_2,g_3,g_4)}"', from=3-1, to=3-3]
	\arrow["{\alpha(g_1,g_2,g_3g_4)}", from=2-4, to=3-4]
        \arrow[Rightarrow, "{\Sigma(g_1',g_2,g_3,g_4)}", from=3-1, to=2-3]
        \arrow[Rightarrow, "{\alpha(\gamma_1g_2,g_3,g_4)}", from=1-3, to=2-3,pos=0.6]
        \arrow[Rightarrow, "{\alpha(\gamma_1,g_2,g_3g_4)}"', from=2-4, to=3-3,pos=0.4]
    \end{tikzcd} = \\
    \begin{tikzcd}[ampersand replacement = \&]
	\& {((g_1g_2)g_3)g_4} \& [3ex] {((g_1g_2)g_3)g_4} \\
	{(g_1'(g_2g_3))g_4} \& [3ex] {(g_1(g_2g_3))g_4} \&\& {(g_1g_2)(g_3g_4)} \\
	{g_1'((g_2g_3)g_4)} \& [3ex] {g_1((g_2g_3)g_4)} \&\& {g_1(g_2(g_3g_4))} \\
        \& g_1'(g_2(g_3g_4))
	\arrow["{\alpha(g_1g_2,g_3,g_4)}", from=1-3, to=2-4]
	\arrow["{\alpha(g_1,g_2,g_3g_4)}", from=2-4, to=3-4]
	\arrow[from=1-3, to=2-2]
	\arrow[from=2-2, to=3-2]
	\arrow[from=3-2, to=3-4]
	\arrow["{\alpha(g_1',g_2,g_3)g_4}"', from=1-2, to=2-1]
	\arrow["{\alpha(g_1',g_2g_3,g_4)}"', from=2-1, to=3-1]
	\arrow[from=2-2, to=2-1]
	\arrow[from=3-2, to=3-1]
	\arrow["{((\gamma_1g_2)g_3)g_4}"', from=1-3, to=1-2]
        \arrow["{g_1'\alpha(g_2,g_3,g_4)}"', from=3-1, to = 4-2]
        \arrow["{\gamma_1(g_2(g_3g_4))}",from=3-4, to=4-2]
        \arrow[Rightarrow, "{\Sigma(g_1,g_2,g_3,g_4)}", from=3-2, to=2-4]
        \arrow[Rightarrow, "{\alpha(\gamma_1,g_2,g_3)g_4}", from=1-2, to=2-2,pos=0.4]
        \arrow[Rightarrow, "{\alpha(\gamma_1,g_2g_3,g_4)}"', from=2-2, to=3-1,pos=0.4]
        \arrow[Rightarrow, "{I(\gamma_1,\alpha(g_2,g_3,g_4))}"',from=4-2,to=3-2]
    \end{tikzcd},
    \end{split}
    \end{align}

    \begin{align}
    \begin{split}
    \begin{tikzcd}[ampersand replacement = \&]
	\& {((g_1g_2')g_3)g_4} \& {((g_1g_2)g_3)g_4} \\
	{(g_1(g_2'g_3))g_4} \&\& {(g_1g_2')(g_3g_4)} \& [3ex] {(g_1g_2)(g_3g_4)} \\
	{g_1((g_2'g_3)g_4)} \&\& {g_1(g_2'(g_3g_4))} \& [3ex] {g_1(g_2(g_3g_4))}
	\arrow[from=1-2, to=2-3]
	\arrow[from=2-3, to=3-3]
	\arrow["{((g_1\gamma_2)g_3)g_4}"', from=1-3, to=1-2]
	\arrow[from=2-4, to=2-3]
	\arrow["{g_1(\gamma_2(g_3g_4))}", from=3-4, to=3-3]
	\arrow["{\alpha(g_1g_2,g_3,g_4)}", from=1-3, to=2-4]
	\arrow["{\alpha(g_1,g_2',g_3)g_4}"', from=1-2, to=2-1]
	\arrow["{\alpha(g_1,g_2'g_3,g_4)}"', from=2-1, to=3-1]
	\arrow["{g_1\alpha(g_2',g_3,g_4)}"', from=3-1, to=3-3]
	\arrow["{\alpha(g_1,g_2,g_3g_4)}", from=2-4, to=3-4]
        \arrow[Rightarrow, "{\Sigma(g_1,g_2',g_3,g_4)}", from=3-1, to=2-3]
        \arrow[Rightarrow, "{\alpha(g_1\gamma_2,g_3,g_4)}", from=1-3, to=2-3,pos=0.6]
        \arrow[Rightarrow, "{\alpha(g_1,\gamma_2,g_3g_4)}"', from=2-4, to=3-3,pos=0.4]
    \end{tikzcd} = \\
    \begin{tikzcd}[ampersand replacement = \&]
	\& {((g_1g_2)g_3)g_4} \& [3ex] {((g_1g_2)g_3)g_4} \\
	{(g_1(g_2'g_3))g_4} \& [3ex] {(g_1(g_2g_3))g_4} \&\& {(g_1g_2)(g_3g_4)} \\
	{g_1((g_2'g_3)g_4)} \& [3ex] {g_1((g_2g_3)g_4)} \&\& {g_1(g_2(g_3g_4))} \\
        \& g_1(g_2'(g_3g_4))
	\arrow["{\alpha(g_1g_2,g_3,g_4)}", from=1-3, to=2-4]
	\arrow["{\alpha(g_1,g_2,g_3g_4)}", from=2-4, to=3-4]
	\arrow[from=1-3, to=2-2]
	\arrow[from=2-2, to=3-2]
	\arrow[from=3-2, to=3-4]
	\arrow["{\alpha(g_1,g_2',g_3)g_4}"', from=1-2, to=2-1]
	\arrow["{\alpha(g_1,g_2'g_3,g_4)}"', from=2-1, to=3-1]
	\arrow[from=2-2, to=2-1]
	\arrow[from=3-2, to=3-1]
	\arrow["{((g_1\gamma_2)g_3)g_4}"', from=1-3, to=1-2]
        \arrow["{g_1\alpha(g_2',g_3,g_4)}"', from=3-1, to = 4-2]
        \arrow["{g_1(\gamma_2(g_3g_4))}",from=3-4, to=4-2]
        \arrow[Rightarrow, "{\Sigma(g_1,g_2,g_3,g_4)}", from=3-2, to=2-4]
        \arrow[Rightarrow, "{\alpha(g_1,\gamma_2,g_3)g_4}", from=1-2, to=2-2,pos=0.4]
        \arrow[Rightarrow, "{\alpha(g_1,\gamma_2g_3,g_4)}"', from=2-2, to=3-1,pos=0.4]
        \arrow[Rightarrow, "{g_1\alpha(\gamma_2,g_3,g_4)}"',from=4-2,to=3-2]
    \end{tikzcd},
    \end{split}
    \end{align}

    \begin{align}
    \begin{split}
    \begin{tikzcd}[ampersand replacement = \&]
	\& {((g_1g_2)g_3')g_4} \& {((g_1g_2)g_3)g_4} \\
	{(g_1(g_2g_3'))g_4} \&\& {(g_1g_2)(g_3'g_4)} \& [3ex] {(g_1g_2)(g_3'g_4)} \\
	{g_1((g_2g_3')g_4)} \&\& {g_1(g_2(g_3'g_4))} \& [3ex] {g_1(g_2(g_3g_4))}
	\arrow[from=1-2, to=2-3]
	\arrow[from=2-3, to=3-3]
	\arrow["{((g_1g_2)\gamma_3)g_4}"', from=1-3, to=1-2]
	\arrow[from=2-4, to=2-3]
	\arrow["{g_1(g_2(\gamma_3g_4))}", from=3-4, to=3-3]
	\arrow["{\alpha(g_1g_2,g_3,g_4)}", from=1-3, to=2-4]
	\arrow["{\alpha(g_1,g_2,g_3')g_4}"', from=1-2, to=2-1]
	\arrow["{\alpha(g_1,g_2g_3',g_4)}"', from=2-1, to=3-1]
	\arrow["{g_1\alpha(g_2,g_3',g_4)}"', from=3-1, to=3-3]
	\arrow["{\alpha(g_1,g_2,g_3g_4)}", from=2-4, to=3-4]
        \arrow[Rightarrow, "{\Sigma(g_1,g_2,g_3',g_4)}", from=3-1, to=2-3]
        \arrow[Rightarrow, "{\alpha(g_1g_2,\gamma_3,g_4)}", from=1-3, to=2-3,pos=0.6]
        \arrow[Rightarrow, "{\alpha(g_1,g_2,\gamma_3g_4)}"', from=2-4, to=3-3,pos=0.4]
    \end{tikzcd} = \\
    \begin{tikzcd}[ampersand replacement = \&]
	\& {((g_1g_2)g_3)g_4} \& [3ex] {((g_1g_2)g_3)g_4} \\
	{(g_1(g_2g_3'))g_4} \& [3ex] {(g_1(g_2g_3))g_4} \&\& {(g_1g_2)(g_3g_4)} \\
	{g_1((g_2g_3')g_4)} \& [3ex] {g_1((g_2g_3)g_4)} \&\& {g_1(g_2(g_3g_4))} \\
        \& g_1(g_2(g_3'g_4))
	\arrow["{\alpha(g_1g_2,g_3,g_4)}", from=1-3, to=2-4]
	\arrow["{\alpha(g_1,g_2,g_3g_4)}", from=2-4, to=3-4]
	\arrow[from=1-3, to=2-2]
	\arrow[from=2-2, to=3-2]
	\arrow[from=3-2, to=3-4]
	\arrow["{\alpha(g_1,g_2,g_3')g_4}"', from=1-2, to=2-1]
	\arrow["{\alpha(g_1,g_2g_3',g_4)}"', from=2-1, to=3-1]
	\arrow[from=2-2, to=2-1]
	\arrow[from=3-2, to=3-1]
	\arrow["{((g_1g_2)\gamma_3)g_4}"', from=1-3, to=1-2]
        \arrow["{g_1\alpha(g_2,g_3',g_4)}"', from=3-1, to = 4-2]
        \arrow["{g_1(g_2(\gamma_3g_4))}",from=3-4, to=4-2]
        \arrow[Rightarrow, "{\Sigma(g_1,g_2,g_3,g_4)}", from=3-2, to=2-4]
        \arrow[Rightarrow, "{\alpha(g_1,g_2,\gamma_3)g_4}", from=1-2, to=2-2,pos=0.4]
        \arrow[Rightarrow, "{\alpha(g_1,g_2\gamma_3,g_4)}"', from=2-2, to=3-1,pos=0.4]
        \arrow[Rightarrow, "{g_1\alpha(g_2,\gamma_3,g_4)}"',from=4-2,to=3-2]
    \end{tikzcd},
    \end{split}
    \end{align}

    \begin{align}
    \begin{split}
    \begin{tikzcd}[ampersand replacement = \&]
	\& {((g_1g_2)g_3)g_4'} \& {((g_1g_2)g_3)g_4} \\
	{(g_1(g_2g_3))g_4'} \&\& {(g_1g_2)(g_3g_4')} \& [3ex] {(g_1g_2)(g_3g_4')} \\
	{g_1((g_2g_3)g_4')} \&\& {g_1(g_2(g_3g_4'))} \& [3ex] {g_1(g_2(g_3g_4))}
	\arrow[from=1-2, to=2-3]
	\arrow[from=2-3, to=3-3]
	\arrow["{((g_1g_2)g_3)\gamma_4}"', from=1-3, to=1-2]
	\arrow[from=2-4, to=2-3]
	\arrow["{g_1(g_2(g_3\gamma_4))}", from=3-4, to=3-3]
	\arrow["{\alpha(g_1g_2,g_3,g_4)}", from=1-3, to=2-4]
	\arrow["{\alpha(g_1,g_2,g_3)g_4'}"', from=1-2, to=2-1]
	\arrow["{\alpha(g_1,g_2g_3,g_4')}"', from=2-1, to=3-1]
	\arrow["{g_1\alpha(g_2,g_3,g_4')}"', from=3-1, to=3-3]
	\arrow["{\alpha(g_1,g_2,g_3g_4)}", from=2-4, to=3-4]
        \arrow[Rightarrow, "{\Sigma(g_1,g_2,g_3,g_4')}", from=3-1, to=2-3]
        \arrow[Rightarrow, "{\alpha(g_1g_2,g_3,\gamma_4)}", from=1-3, to=2-3,pos=0.6]
        \arrow[Rightarrow, "{\alpha(g_1,g_2,g_3\gamma_4)}"', from=2-4, to=3-3,pos=0.4]
    \end{tikzcd} = \\
    \begin{tikzcd}[ampersand replacement = \&]
	\& {((g_1g_2)g_3)g_4} \& [3ex] {((g_1g_2)g_3)g_4} \\
	{(g_1(g_2g_3))g_4'} \& [3ex] {(g_1(g_2g_3))g_4} \&\& {(g_1g_2)(g_3g_4)} \\
	{g_1((g_2g_3)g_4')} \& [3ex] {g_1((g_2g_3)g_4)} \&\& {g_1(g_2(g_3g_4))} \\
        \& g_1(g_2(g_3g_4'))
	\arrow["{\alpha(g_1g_2,g_3,g_4)}", from=1-3, to=2-4]
	\arrow["{\alpha(g_1,g_2,g_3g_4)}", from=2-4, to=3-4]
	\arrow[bend left= 30, from=1-3, to=2-2]
	\arrow[from=2-2, to=3-2]
	\arrow[from=3-2, to=3-4]
	\arrow["{\alpha(g_1,g_2,g_3)g_4'}"', from=1-2, to=2-1]
	\arrow["{\alpha(g_1,g_2g_3,g_4')}"', from=2-1, to=3-1]
	\arrow[from=2-2, to=2-1]
	\arrow[from=3-2, to=3-1]
	\arrow["{((g_1g_2)g_3)\gamma_4}"', from=1-3, to=1-2]
        \arrow["{g_1\alpha(g_2,g_3,g_4')}"', from=3-1, to = 4-2]
        \arrow["{g_1(g_2(g_3\gamma_4))}",from=3-4, to=4-2]
        \arrow[Rightarrow, "{\Sigma(g_1,g_2,g_3,g_4)}", from=3-2, to=2-4]
        \arrow[Rightarrow, "{I(\alpha(g_1,g_2,g_3),\gamma_4)^{-1}}", from=1-2, to=2-2,pos=0.4]
        \arrow[Rightarrow, "{\alpha(g_1,g_2g_3,\gamma_4)}"', from=2-2, to=3-1,pos=0.4]
        \arrow[Rightarrow, "{g_1\alpha(g_2,g_3,\gamma_4)}"',from=4-2,to=3-2]
    \end{tikzcd},
    \end{split}
    \end{align}

    \pagebreak

    \subsection{Homomorphisms of Lie 3-groups}\label{sec:hom3groups}
     
    For the purposes of this article it is not necessary to give the most general definition of a homomorphism between Lie 3-groups, so we just present a definition that covers all the homomorphisms that appear here.
    
    \begin{definition}\label{def:hom3groups}
        A \emph{homomorphism of Lie 3-groups} $\phi:\mathcal{G}^1 \rightarrow \mathcal{G}^2$ is given by 
        \begin{enumerate}
            \item A smooth pseudofunctor $\phi:\mathcal{G}^1 \rightarrow \mathcal{G}^2$. In particular, given $g \stackrel{\gamma}{\rightarrow} g' \stackrel{\gamma'}{\rightarrow} g''$ in $\mathcal{G}^1$ we have a compositor in $\mathcal{G}^2$ denoted by
                \begin{equation}
                \begin{tikzcd}[column sep = 10ex]
            \phi(g) \ar[rr,bend left=60,"{\phi(\gamma' \circ \gamma)}"{name=F}] 
              \ar[rr,bend right=10,"{\phi(\gamma') \circ \phi(\gamma)}"{name=G},swap]
             & & \phi(g'') \\ \ar[Rightarrow,from=F,to=G,"{\phi(\gamma,\gamma')}"{swap,pos=0.5},shorten >=1.5pt] &
                \end{tikzcd}
            \end{equation}
            and required to be natural on 2-cells of $\mathcal{G}^1$ and associative on arrows of $\mathcal{G}^1$.
            
            \item For $g_i \in \mathcal{G}^1_0$, $i=1, \,2$, multiplicator arrows
            \begin{equation}
                \phi(g_1g_2) \stackrel{\phi(g_1,g_2)}{\rightarrow} \phi(g_1)\phi(g_2),
            \end{equation}
            and for $\gamma_i:g_i \rightarrow g_i' \in \mathcal{G}_1^1$, $i=1, \,2$, multiplicator 2-cells 
            \begin{equation}
            \begin{tikzcd}[column sep=14ex]
             \phi(g_1g_2) \ar[r,"{\phi(g_1,g_2)}",{name=U}] \ar[d,"{\phi(\gamma_1g_2)}",swap] 
             & \phi(g_1)\phi(g_2) \ar[d,"{\phi(\gamma_1)\phi(g_2)}"] \\
             \phi(g_1'g_2) \ar[r,"{\phi(g_1',g_2)}",{name=D},swap] \ar[Rightarrow, ur,"{\phi(\gamma_1,g_2)}"] 
             & \phi(g_1')\phi(g_2)
            \end{tikzcd},
            \begin{tikzcd}[column sep=14ex]
             \phi(g_1g_2) \ar[r,"{\phi(g_1,g_2)}",{name=U}] \ar[d,"{\phi(g_1\gamma_2)}",swap] 
             & \phi(g_1)\phi(g_2) \ar[d,"{\phi(g_1)\phi(\gamma_2)}"] \\
             \phi(g_1g_2') \ar[r,"{\phi(g_1,g_2')}",{name=D},swap] \ar[Rightarrow, ur,"{\phi(g_1,\gamma_2)}"] 
             & \phi(g_1)\phi(g_2')
            \end{tikzcd}
            \end{equation}
        These must be natural on 2-cells of $\mathcal{G}^1$, respect composition of arrows of $\mathcal{G}^1$ and intertwine the interchange cells of $\mathcal{G}^1$ and $\mathcal{G}^2$. 
        \item For $g_i \in  \mathcal{G}_0^1$, $i=1, \,2, \,3$, associator 2-cells 
        \begin{equation}
            \begin{tikzcd}[column sep=18ex]
             \phi((g_1g_2)g_3) \ar[r,"{\phi(\alpha(g_1,g_2,g_3))}",{name=U}] \ar[d,"{\phi(g_1g_2,g_3)}",swap] 
             & \phi(g_1(g_2g_3)) \ar[d,"{\phi(g_1,g_2g_3)}"] \ar[Rightarrow, ddl,"{\phi(g_1,g_2,g_3)}",swap] \\
             \phi(g_1g_2)\phi(g_3)  \ar[d,"{\phi(g_1,g_2)\phi(g_3)}",swap]
             & \phi(g_1)\phi(g_2g_3) \ar[d,"{\phi(g_1)\phi(g_2,g_3)}"] \\
             (\phi(g_1)\phi(g_2))\phi(g_3) \ar[r,"{\alpha(\phi(g_1),\phi(g_2),\phi(g_3))}",swap] & \phi(g_1)(\phi(g_2)\phi(g_3))
            \end{tikzcd}
        \end{equation}
        that are required to be natural on arrows of $\mathcal{G}^1$ and to intertwine the pentagonators of $\mathcal{G}^1$ and $\mathcal{G}^2$.
        \end{enumerate} 
    \end{definition}
        \pagebreak 
    We spell out in detail the conditions in \cref{def:hom3groups}. First, the smooth pseudofunctor $\phi: \mathcal{G}^1 \rightarrow \mathcal{G}^2$ is a map of objects, arrows and 2-cells preserving all source, target and identity maps. Naturality on 2-cells for its compositor means that for $\psi:\gamma \Rightarrow \eta:g \rightarrow g' \in \mathcal{G}^1_2$, $\psi':\gamma' \Rightarrow \eta':g' \rightarrow g'' \in \mathcal{G}^1_2$ we must have
    \begin{equation}
        \begin{tikzcd}[every label/.append style = {font = \tiny},column sep = 10ex]
            \phi(g) \ar[rr,bend left=60,"{\phi(\gamma' \circ \gamma)}"{name=F}] 
              \ar[rr,bend left = 10,"{\phi(\gamma') \circ \phi(\gamma)}"{name=G},swap]
              \ar[rr,bend right=80,"{\phi(\eta') \circ \phi(\eta)}"{name=H},swap]
             & & \phi(g'') \\ \ar[Rightarrow,from=F,to=G,"{\phi(\gamma,\gamma')}"{swap,pos=0.5},shorten >=1.5pt] \ar[Rightarrow,from=G,to=H,"{\phi(\psi') \circ \phi(\psi)}"{swap,pos=0.35},shorten >=1.5pt] &
        \end{tikzcd} = 
        \begin{tikzcd}[every label/.append style = {font = \tiny},column sep = 10ex]
            \phi(g) \ar[rr,bend left=60,"{\phi(\gamma' \circ \gamma)}"{name=F}] 
              \ar[rr,bend left = 10,"{\phi(\eta' \circ \eta)}"{name=G},swap]
              \ar[rr,bend right=80,"{\phi(\eta') \circ \phi(\eta)}"{name=H},swap]
             & & \phi(g'') \\ \ar[Rightarrow,from=F,to=G,"{\phi(\psi' \circ \psi)}"{swap,pos=0.5},shorten >=1.5pt] \ar[Rightarrow,from=G,to=H,"{\phi(\eta,\eta')}"{swap,pos=0.35},shorten >=1.5pt] &
        \end{tikzcd},
    \end{equation}
    
    and associativity on arrows of $\mathcal{G}^1$ means that for $g \stackrel{\gamma}{\rightarrow} g' \stackrel{\gamma'}{\rightarrow} g'' \stackrel{\gamma''}{\rightarrow} g'''$ in $\mathcal{G}^1$ we must have

    \begin{equation}
        \begin{tikzcd}[every label/.append style = {font = \tiny},column sep = 10ex]
            \phi(g) \ar[rr,bend left=75,"{\phi(\gamma'' \circ \gamma' \circ \gamma)}"{name=F},pos=0.5] 
              \ar[rr,bend left = 10,"{\phi(\gamma''\circ \gamma')  \circ \phi(\gamma)}"{name=G},swap,pos=0.60]
              \ar[rr,bend right=80,"{(\phi(\gamma'') \circ \phi(\gamma')) \circ \phi(\gamma)}"{name=H},swap]
             & & \phi(g''') \\ \ar[Rightarrow,from=F,to=G,"{\phi(\gamma'' \circ \gamma',\gamma)}"{swap,pos=0.55},shorten >=1.5pt] \ar[Rightarrow,from=G,to=H,"{\phi(\gamma'',\gamma') \circ id}"{swap,pos=0.5},shorten >=1.5pt] &
        \end{tikzcd} = 
        \begin{tikzcd}[every label/.append style = {font = \tiny},column sep = 10ex]
            \phi(g) \ar[rr,bend left=75,"{\phi(\gamma'' \circ \gamma' \circ \gamma)}"{name=F}] 
              \ar[rr,bend left = 10,"{\phi(\gamma'') \circ \phi(\gamma' \circ \gamma)}"{name=G},swap,pos=0.6]
              \ar[rr,bend right=80,"{\phi(\gamma'') \circ (\phi(\gamma') \circ \phi(\gamma))}"{name=H},swap]
             & &\phi(g''') \\ \ar[Rightarrow,from=F,to=G,"{\phi(\gamma'', \gamma'\circ \gamma)}"{swap,pos=0.55},shorten >=1.5pt] \ar[Rightarrow,from=G,to=H,"{id \circ \phi(\gamma',\gamma)}"{swap,pos=0.5},shorten >=1.5pt] &
        \end{tikzcd} =: \phi(\gamma,\gamma',\gamma''),
    \end{equation}
    
    Naturality for the multiplicator 2-cells means that for $\psi_i: \gamma_i \Rightarrow \eta_i \in \mathcal{G}_2^1$, $i=1, \,2$,

    \begin{equation}
            \begin{tikzcd}[column sep=11ex]
             \phi(g_1g_2) \ar[r,"{\phi(g_1,g_2)}",{name=U}] \ar[d,"{}"{name=A}] \ar[d,bend right= 100,"{\phi(\gamma_1g_2)}"{name=B},swap] \ar[Rightarrow,from=B,to=A,"{\phi(\psi_1g_2)}"]
             & \phi(g_1)\phi(g_2) \ar[d,"{\phi(\eta_1)\phi(g_2)}"]  \\
             \phi(g_1'g_2) \ar[r,"{\phi(g_1',g_2)}",{name=D},swap] \ar[Rightarrow, ur,"{\phi(\eta_1,g_2)}"]
             & \phi(g_1')\phi(g_2)
            \end{tikzcd} =
            \begin{tikzcd}[column sep=11ex]
             \phi(g_1g_2) \ar[r,"{\phi(g_1,g_2)}",{name=U}] \ar[d,"{\phi(\gamma_1g_2)}",swap] 
             & \phi(g_1)\phi(g_2) \ar[d,"{}"{name=B}] \ar[d,bend left= 100,"{\phi(\eta_1)\phi(g_2)}"{name=A}] \ar[Rightarrow,from=B,to=A,"{\phi(\psi_1)\phi(g_2)}",pos=0.4] \\
             \phi(g_1'g_2) \ar[r,"{\phi(g_1',g_2)}",{name=D},swap] \ar[Rightarrow, ur,"{\phi(\gamma_1,g_2)}"] 
             & \phi(g_1')\phi(g_2)
            \end{tikzcd},
            \end{equation}
    
     \begin{equation}
            \begin{tikzcd}[column sep=11ex]
             \phi(g_1g_2) \ar[r,"{\phi(g_1,g_2)}",{name=U}] \ar[d,"{}"{name=A}] \ar[d,bend right= 100,"{\phi(g_1\gamma_2)}"{name=B},swap] \ar[Rightarrow,from=B,to=A,"{\phi(g_1\psi_2)}"]
             & \phi(g_1)\phi(g_2) \ar[d,"{\phi(g_1)\phi(\eta_2)}"]  \\
             \phi(g_1g_2') \ar[r,"{\phi(g_1,g_2')}",{name=D},swap] \ar[Rightarrow, ur,"{\phi(g_1,\eta_2)}"]
             & \phi(g_1)\phi(g_2')
            \end{tikzcd} =
            \begin{tikzcd}[column sep=11ex]
             \phi(g_1g_2) \ar[r,"{\phi(g_1,g_2)}",{name=U}] \ar[d,"{\phi(g_1\gamma_2)}",swap] 
             & \phi(g_1)\phi(g_2) \ar[d,"{}"{name=B}] \ar[d,bend left= 100,"{\phi(g_1)\phi(\eta_2)}"{name=A}] \ar[Rightarrow,from=B,to=A,"{\phi(g_1)\phi(\psi_2)}",pos=0.4] \\
             \phi(g_1g_2') \ar[r,"{\phi(g_1,g_2')}",{name=D},swap] \ar[Rightarrow, ur,"{\phi(g_1,\gamma_2)}"] 
             & \phi(g_1)\phi(g_2')
            \end{tikzcd},
            \end{equation}
    
    The requirement to respect composition for the multiplicator 2-cells means that for $g_i \stackrel{\gamma_i}{\rightarrow} g_i' \stackrel{\gamma_i'}{\rightarrow} g_i''$ in $\mathcal{G}^1$, $i=1,2$,
    \begin{align}
    \begin{split}
        \begin{tikzcd}[every label/.append style = {font = \tiny},column sep=14ex, ampersand replacement = \&]
             \phi(g_1g_2) \ar[r,"{\phi(g_1,g_2)}",{name=U}] \ar[d,"{}",swap] \ar[dd,bend right=80,looseness = 2,"{}"{name=A},pos=0.5,swap] \ar[dd,bend right=100,looseness = 4.5,"{\phi((\gamma_1'\circ \gamma_1)g_2)}"{name=B},pos=0.5,swap]
             \& \phi(g_1)\phi(g_2)  \ar[d,"{\phi(\gamma_1)\phi(g_2)}"]   \&[-10ex] \\
             \phi(g_1'g_2) \ar[Rightarrow, ur,"{\phi(\gamma_1,g_2)}"] \ar[r,"{\phi(g_1',g_2)}",{name=D}] \ar[d,"{}",swap] \ar[Rightarrow,from=A,to=2-1,"{\phi(\gamma_1g_2,\gamma_1'g_2)}",pos=0.6]
             \ar[Rightarrow,from=B,to=A,"{\phi(R_{g_2}(\gamma_1,\gamma_1'))}"]
             \& \phi(g_1')\phi(g_2) \ar[d,"{\phi(\gamma_1')\phi(g_2)}"]  \& [-10ex] \\
             \phi(g_1''g_2) \ar[Rightarrow, ur,"{\phi(\gamma_1',g_2)}"] \ar[r,"{\phi(g_1'',g_2)}",{name=2},swap] 
             \& \phi(g_1'')\phi(g_2) \& [-10ex]
            \end{tikzcd} = \\
        \begin{tikzcd}[every label/.append style = {font = \tiny},column sep=14ex,row sep = 10ex, ampersand replacement = \&]
             \phi(g_1g_2) \ar[r,"{\phi(g_1,g_2)}"{name=A}] \ar[d,"{\phi((\gamma_1' \circ \gamma_1)g_2)}"{name=X},swap]
             \& \phi(g_1)\phi(g_2)  \ar[d,"{}"{name=A},swap] \ar[d,bend left = 80, looseness = 1.5, "{}"{name=B}] \ar[d,bend left = 100, looseness = 4.5, "{\phi(\gamma_1')\phi(g_2) \circ \phi(\gamma_1)\phi(g_2)}"{name=C}] \ar[Rightarrow, from=A,to=B,"{\phi(\gamma_1,\gamma_1')\phi(g_2)}"] \ar[Rightarrow, from=B,to=C,"{R_{\phi(g_2)}(\phi(\gamma_1),\phi(\gamma_1'))}",pos=0.45]\\
             \phi(g_1''g_2) \ar[r,"{\phi(g_1'',g_2)}",{name=D},swap] \ar[Rightarrow, ur,"{\phi(\gamma_1'\circ \gamma_1,g_2)}"]
             \& \phi(g_1'')\phi(g_2).
            \end{tikzcd} ,
    \end{split}  
    \end{align}
    \begin{align}
    \begin{split}
        \begin{tikzcd}[every label/.append style = {font = \tiny},column sep=14ex, ampersand replacement = \&]
             \phi(g_1g_2) \ar[r,"{\phi(g_1,g_2)}",{name=U}] \ar[d,"{}",swap] \ar[dd,bend right=80,looseness = 2,"{}"{name=A},pos=0.5,swap] \ar[dd,bend right=100,looseness = 4.5,"{\phi(g_1(\gamma_2'\circ \gamma_2))}"{name=B},pos=0.5,swap]
             \& \phi(g_1)\phi(g_2)  \ar[d,"{\phi(g_1)\phi(\gamma_2)}"]   \&[-10ex] \\
             \phi(g_1g_2') \ar[Rightarrow, ur,"{\phi(g_1,\gamma_2)}"] \ar[r,"{\phi(g_1,g_2')}",{name=D}] \ar[d,"{}",swap] \ar[Rightarrow,from=A,to=2-1,"{\phi(g_1\gamma_2,g_1\gamma_2')}",pos=0.6]
             \ar[Rightarrow,from=B,to=A,"{\phi(L_{g_1}(\gamma_2,\gamma_2'))}"]
             \& \phi(g_1)\phi(g_2') \ar[d,"{\phi(g_1)\phi(\gamma_2')}"]  \& [-10ex] \\
             \phi(g_1g_2'') \ar[Rightarrow, ur,"{\phi(g_1,\gamma_2')}"] \ar[r,"{\phi(g_1,g_2'')}",{name=2},swap] 
             \& \phi(g_1)\phi(g_2'') \& [-10ex]
            \end{tikzcd} = \\
        \begin{tikzcd}[every label/.append style = {font = \tiny},column sep=14ex,row sep = 10ex, ampersand replacement = \&]
             \phi(g_1g_2) \ar[r,"{\phi(g_1,g_2)}"{name=A}] \ar[d,"{\phi(g_1(\gamma_2' \circ \gamma_2))}"{name=X},swap]
             \& \phi(g_1)\phi(g_2)  \ar[d,"{}"{name=A},swap] \ar[d,bend left = 80, looseness = 1.5, "{}"{name=B}] \ar[d,bend left = 100, looseness = 4.5, "{\phi(g_1)\phi(\gamma_2') \circ \phi(g_1)\phi(\gamma_2)}"{name=C}] \ar[Rightarrow, from=A,to=B,"{\phi(g_1)\phi(\gamma_2,\gamma_2')}"] \ar[Rightarrow, from=B,to=C,"{L_{\phi(g_1)}(\phi(\gamma_2),\phi(\gamma_2'))}",pos=0.45]\\
             \phi(g_1g_2'') \ar[r,"{\phi(g_1,g_2'')}",{name=D},swap] \ar[Rightarrow, ur,"{\phi(g_1,\gamma_2'\circ \gamma_2)}"]
             \& \phi(g_1)\phi(g_2'')
            \end{tikzcd} .
    \end{split}  
    \end{align}
    and respecting interchange cells means that for $\gamma_i:g_i \rightarrow g_i' \in \mathcal{G}^1_1$, $i=1,2$
    \begin{align}
    \begin{split}
        \begin{tikzcd}[ampersand replacement=\&, column sep = 15ex, row sep = 14ex]
	\& {\phi(g_1g_2)} \& {\phi(g_1'g_2)} \\
	{\phi(g_1)\phi(g_2)} \& {\phi(g_1g_2')} \& {\phi(g_1'g_2')} \\
	{\phi(g_1)\phi(g_2')} \& {\phi(g_1')\phi(g_2')}
	\arrow[from=1-2, to=2-2]
	\arrow[from=2-2, to=2-3]
	\arrow["{\phi(\gamma_1g_2)}", from=1-2, to=1-3]
	\arrow["{\phi(g_1'\gamma_2)}", from=1-3, to=2-3]
	\arrow[from=2-2, to=3-1]
	\arrow["{\phi(\gamma_1)\phi(g_2')}"', from=3-1, to=3-2]
	\arrow["{\phi(g_1',g_2')}", from=2-3, to=3-2]
	\arrow["{\phi(g_1,g_2)}"', from=1-2, to=2-1]
	\arrow["{\phi(g_1)\phi(\gamma_2)}"', from=2-1, to=3-1]
        \arrow[Rightarrow, from=1-3, to=2-2, "{\substack{\phi(g_1\gamma_2,\gamma_1g_2') \circ \\ \phi(I(\gamma_1,\gamma_2)) \circ  \\ \phi(\gamma_1g_2,g_1'\gamma_2)^{-1}}}"']
        \arrow[Rightarrow, from=2-2, to=2-1, "{\phi(g_1,\gamma_2)}"']
        \arrow[Rightarrow, from=2-3, to=3-1, "{\phi(\gamma_1,g_2')}"']
\end{tikzcd} = \\
        \begin{tikzcd}[ampersand replacement=\&, column sep = 15ex, row sep = 14ex]
	\& {\phi(g_1g_2)} \& {\phi(g_1'g_2)} \\
	{\phi(g_1)\phi(g_2)} \& {\phi(g_1')\phi(g_2)} \& {\phi(g_1'g_2')} \\
	{\phi(g_1)\phi(g_2')} \& {\phi(g_1')\phi(g_2')}
	\arrow["{\phi(\gamma_1g_2)}", from=1-2, to=1-3]
	\arrow["{\phi(g_1'\gamma_2)}", from=1-3, to=2-3]
	\arrow["{\phi(\gamma_1)\phi(g_2')}"', from=3-1, to=3-2]
	\arrow["{\phi(g_1',g_2')}", from=2-3, to=3-2]
	\arrow["{\phi(g_1,g_2)}"', from=1-2, to=2-1]
	\arrow["{\phi(g_1)\phi(\gamma_2)}"', from=2-1, to=3-1]
	\arrow[from=1-3, to=2-2]
	\arrow[from=2-1, to=2-2]
	\arrow[from=2-2, to=3-2]
        \arrow[Rightarrow,from=1-3,to=2-1,"{\phi(\gamma_1,g_2)}"']
        \arrow[Rightarrow,from=2-2,to=3-1,"{I(\phi(\gamma_1),\phi(\gamma_2))}"']
        \arrow[Rightarrow,from=2-3,to=2-2,"{\phi(g_1',\gamma_2)}"']
\end{tikzcd}
\end{split}
    \end{align}
        \pagebreak[4]

    Finally, the associator cells being natural on arrows of $\mathcal{G}^1$ means that for $\gamma_i:g_i \rightarrow g_i' \in \mathcal{G}_1^1$, $i=1,\,2,\,3$,
    \begin{align}
    \begin{split}
            \begin{tikzcd}[column sep=16ex, ampersand replacement = \&, row sep = 12ex]
             \& \phi((g_1'g_2)g_3) \ar[dr,"{\phi(\alpha(g_1',g_2,g_3))}"{name=S}] \& \\
             \phi((g_1g_2)g_3) \ar[ur,"{\phi((\gamma_1g_2)g_3)}"] \ar[r,"{\phi(\alpha(g_1,g_2,g_3))}"'{name=Q}] \ar[d,"{\phi(g_1g_2,g_3)}",swap] \ar[Rightarrow,from=S,to=Q,"{\substack{\varphi(\alpha(g_1,g_2,g_3),\gamma_1(g_2g_3)) \circ \\  \phi(\alpha(\gamma_1,g_2,g_3)) \circ \\ \varphi((\gamma_1g_2)g_3,\alpha(g_1',g_2,g_3))^{-1} }}"'{font=\footnotesize}]
             \& \phi(g_1(g_2g_3)) \ar[d,"{}"] \ar[Rightarrow, ddl,"{\phi(g_1,g_2,g_3)}",swap]  \ar[r,"{\phi(\gamma_1(g_2g_3))}"'] \& \phi(g_1'(g_2g_3)) \ar[d,"{\phi(g_1',g_2g_3)}"] \ar[Rightarrow, dl, "{\phi(\gamma_1,g_2g_3)}"]\\
             \phi(g_1g_2)\phi(g_3)  \ar[d,"{\phi(g_1,g_2)\phi(g_3)}",swap]
             \& \phi(g_1)\phi(g_2g_3) \ar[d,"{}"] \ar[r] \& \phi(g_1')\phi(g_2g_3) \ar[d,"{\phi(g_1')\phi(g_2,g_3)}"] \ar[Rightarrow, dl, "{I(\phi(\gamma_1),\phi(g_2,g_3))}"] \\
             (\phi(g_1)\phi(g_2))\phi(g_3) \ar[r,"{\alpha(\phi(g_1),\phi(g_2),\phi(g_3))}"'] \& \phi(g_1)(\phi(g_2)\phi(g_3)) \ar[r,"{\phi(\gamma_1)(\phi(g_2)\phi(g_3))}"'] \& \phi(g_1')(\phi(g_2)\phi(g_3))
            \end{tikzcd} = \\[8ex]
            \begin{tikzcd}[column sep=16ex, ampersand replacement = \&, row sep = 7ex]
             \phi((g_1g_2)g_3) \ar[r,"{\phi((\gamma_1g_2)g_3)}"] \ar[d,"{\phi(g_1g_2,g_3)}"'] \& \phi((g_1'g_2)g_3) \ar[r,"{\phi(\alpha(g_1',g_2,g_3))}",{name=U}] \ar[d,"{}",swap] \ar[Rightarrow,dl,"{\phi(\gamma_1g_2,g_3)}"']
             \& \phi(g_1'(g_2g_3)) \ar[d,"{\phi(g_1',g_2g_3)}"] \ar[Rightarrow, ddl,"{\phi(g_1',g_2,g_3)}",swap] \\
             \phi(g_1g_2)\phi(g_3) \ar[r] \ar[d,"{\phi(g_1,g_2)\phi(g_3)}"'] \& \phi(g_1'g_2)\phi(g_3)  \ar[d,"{}",swap] \ar[Rightarrow,dl,"{\phi(\gamma_1,g_2)\phi(g_3)}"']
             \& \phi(g_1')\phi(g_2g_3) \ar[d,"{\phi(g_1')\phi(g_2,g_3)}"] \\
             (\phi(g_1)\phi(g_2))\phi(g_3) \ar[r,"{(\phi(\gamma_1)\phi(g_2))\phi(g_3)}"'] \ar[dr,"{\alpha(\phi(g_1),\phi(g_2),\phi(g_3))}"'] \& (\phi(g_1')\phi(g_2))\phi(g_3) \ar[r,"{\alpha(\phi(g_1'),\phi(g_2),\phi(g_3))}"] \ar[Rightarrow,d,"{\alpha(\phi(\gamma_1),\phi(g_2),\phi(g_3))}"',pos=0.4] \& \phi(g_1')(\phi(g_2)\phi(g_3)) \\
              \& \phi(g_1)((\phi(g_2)\phi(g_3)) \ar[ur,"{\phi(\gamma_1)(\phi(g_2)\phi(g_3))}"'] \&
            \end{tikzcd},
    \end{split}
    \end{align} \pagebreak[4]

    \begin{align}
    \begin{split}
            \begin{tikzcd}[column sep=16ex, ampersand replacement = \&, row sep = 12ex]
             \& \phi((g_1g_2')g_3) \ar[dr,"{\phi(\alpha(g_1,g_2',g_3))}"{name=S}] \& \\
             \phi((g_1g_2)g_3) \ar[ur,"{\phi((g_1\gamma_2)g_3)}"] \ar[r,"{\phi(\alpha(g_1,g_2,g_3))}"'{name=Q}] \ar[d,"{\phi(g_1g_2,g_3)}",swap] \ar[Rightarrow,from=S,to=Q,"{\substack{\varphi(\alpha(g_1,g_2,g_3),g_1(\gamma_2g_3)) \circ \\  \phi(\alpha(g_1,\gamma_2,g_3)) \circ \\ \varphi((g_1\gamma_2)g_3,\alpha(g_1,g_2',g_3))^{-1} }}"'{font=\footnotesize}]
             \& \phi(g_1(g_2g_3)) \ar[d,"{}"] \ar[Rightarrow, ddl,"{\phi(g_1,g_2,g_3)}",swap]  \ar[r,"{\phi(g_1(\gamma_2g_3))}"'] \& \phi(g_1(g_2'g_3)) \ar[d,"{\phi(g_1,g_2'g_3)}"] \ar[Rightarrow, dl, "{\phi(g_1,\gamma_2g_3)}"]\\
             \phi(g_1g_2)\phi(g_3)  \ar[d,"{\phi(g_1,g_2)\phi(g_3)}",swap]
             \& \phi(g_1)\phi(g_2g_3) \ar[d,"{}"] \ar[r] \& \phi(g_1)\phi(g_2'g_3) \ar[d,"{\phi(g_1)\phi(g_2',g_3)}"] \ar[Rightarrow, dl, "{\phi(g_1)\phi(\gamma_2,g_3)}"] \\
             (\phi(g_1)\phi(g_2))\phi(g_3) \ar[r,"{\alpha(\phi(g_1),\phi(g_2),\phi(g_3))}"'] \& \phi(g_1)(\phi(g_2)\phi(g_3)) \ar[r,"{\phi(g_1)(\phi(\gamma_2)\phi(g_3))}"'] \& \phi(g_1)(\phi(g_2')\phi(g_3))
            \end{tikzcd} = \\[8ex]
            \begin{tikzcd}[column sep=16ex, ampersand replacement = \&, row sep = 7ex]
             \phi((g_1g_2)g_3) \ar[r,"{\phi((g_1\gamma_2)g_3)}"] \ar[d,"{\phi(g_1g_2,g_3)}"'] \& \phi((g_1g_2')g_3) \ar[r,"{\phi(\alpha(g_1,g_2',g_3))}",{name=U}] \ar[d,"{}",swap] \ar[Rightarrow,dl,"{\phi(g_1\gamma_2,g_3)}"']
             \& \phi(g_1(g_2'g_3)) \ar[d,"{\phi(g_1,g_2'g_3)}"] \ar[Rightarrow, ddl,"{\phi(g_1,g_2',g_3)}",swap] \\
             \phi(g_1g_2)\phi(g_3) \ar[r] \ar[d,"{\phi(g_1,g_2)\phi(g_3)}"'] \& \phi(g_1g_2')\phi(g_3)  \ar[d,"{}",swap] \ar[Rightarrow,dl,"{\phi(g_1,\gamma_2)\phi(g_3)}"']
             \& \phi(g_1)\phi(g_2'g_3) \ar[d,"{\phi(g_1)\phi(g_2',g_3)}"] \\
             (\phi(g_1)\phi(g_2))\phi(g_3) \ar[r,"{(\phi(g_1)\phi(\gamma_2))\phi(g_3)}"'] \ar[dr,"{\alpha(\phi(g_1),\phi(g_2),\phi(g_3))}"'] \& (\phi(g_1)\phi(g_2'))\phi(g_3) \ar[r,"{\alpha(\phi(g_1),\phi(g_2'),\phi(g_3))}"] \ar[Rightarrow,d,"{\alpha(\phi(g_1),\phi(\gamma_2),\phi(g_3))}"',pos=0.4] \& \phi(g_1)(\phi(g_2')\phi(g_3)) \\
              \& \phi(g_1)((\phi(g_2)\phi(g_3)) \ar[ur,"{\phi(g_1)(\phi(\gamma_2)\phi(g_3))}"'] \&
            \end{tikzcd},
    \end{split}
    \end{align} \pagebreak[4]

    \begin{align}
    \begin{split}
            \begin{tikzcd}[column sep=16ex, ampersand replacement = \&, row sep = 12ex]
             \& \phi((g_1g_2)g_3') \ar[dr,"{\phi(\alpha(g_1,g_2,g_3'))}"{name=S}] \& \\
             \phi((g_1g_2)g_3) \ar[ur,"{\phi((g_1g_2)\gamma_3)}"] \ar[r,"{\phi(\alpha(g_1,g_2,g_3))}"'{name=Q}] \ar[d,"{\phi(g_1g_2,g_3)}",swap] \ar[Rightarrow,from=S,to=Q,"{\substack{\varphi(\alpha(g_1,g_2,g_3),g_1(g_2\gamma_3)) \circ \\  \phi(\alpha(g_1,g_2,\gamma_3)) \circ \\ \varphi((g_1g_2)\gamma_3,\alpha(g_1,g_2,g_3'))^{-1} }}"'{font=\footnotesize}]
             \& \phi(g_1(g_2g_3)) \ar[d,"{}"] \ar[Rightarrow, ddl,"{\phi(g_1,g_2,g_3)}",swap]  \ar[r,"{\phi(g_1(g_2\gamma_3))}"'] \& \phi(g_1(g_2g_3')) \ar[d,"{\phi(g_1,g_2g_3')}"] \ar[Rightarrow, dl, "{\phi(g_1,g_2\gamma_3)}"]\\
             \phi(g_1g_2)\phi(g_3)  \ar[d,"{\phi(g_1,g_2)\phi(g_3)}",swap]
             \& \phi(g_1)\phi(g_2g_3) \ar[d,"{}"] \ar[r] \& \phi(g_1)\phi(g_2g_3') \ar[d,"{\phi(g_1)\phi(g_2,g_3')}"] \ar[Rightarrow, dl, "{\phi(g_1)\phi(g_2,\gamma_3)}"] \\
             (\phi(g_1)\phi(g_2))\phi(g_3) \ar[r,"{\alpha(\phi(g_1),\phi(g_2),\phi(g_3))}"'] \& \phi(g_1)(\phi(g_2)\phi(g_3)) \ar[r,"{\phi(g_1)(\phi(g_2)\phi(\gamma_3))}"'] \& \phi(g_1)(\phi(g_2)\phi(g_3'))
            \end{tikzcd} = \\[8ex]
            \begin{tikzcd}[column sep=16ex, ampersand replacement = \&, row sep = 7ex]
             \phi((g_1g_2)g_3) \ar[r,"{\phi((g_1g_2)\gamma_3)}"] \ar[d,"{\phi(g_1g_2,g_3)}"'] \& \phi((g_1g_2)g_3') \ar[r,"{\phi(\alpha(g_1,g_2,g_3'))}",{name=U}] \ar[d,"{}",swap] \ar[Rightarrow,dl,"{\phi(g_1g_2,\gamma_3)}"']
             \& \phi(g_1(g_2g_3')) \ar[d,"{\phi(g_1,g_2g_3')}"] \ar[Rightarrow, ddl,"{\phi(g_1,g_2,g_3')}",swap] \\
             \phi(g_1g_2)\phi(g_3) \ar[r] \ar[d,"{\phi(g_1,g_2)\phi(g_3)}"'] \& \phi(g_1g_2)\phi(g_3')  \ar[d,"{}",swap] \ar[Rightarrow,dl,"{I(\phi(g_1,g_2),\phi(\gamma_3))^{-1}}"']
             \& \phi(g_1)\phi(g_2g_3') \ar[d,"{\phi(g_1)\phi(g_2,g_3')}"] \\
             (\phi(g_1)\phi(g_2))\phi(g_3) \ar[r,"{(\phi(g_1)\phi(g_2))\phi(\gamma_3)}"'] \ar[dr,"{\alpha(\phi(g_1),\phi(g_2),\phi(g_3))}"'] \& (\phi(g_1)\phi(g_2))\phi(g_3') \ar[r,"{\alpha(\phi(g_1),\phi(g_2),\phi(g_3'))}"] \ar[Rightarrow,d,"{\alpha(\phi(g_1),\phi(g_2),\phi(\gamma_3))}"',pos=0.4] \& \phi(g_1)(\phi(g_2)\phi(g_3')) \\
              \& \phi(g_1)((\phi(g_2)\phi(g_3)) \ar[ur,"{\phi(g_1)(\phi(g_2)\phi(\gamma_3))}"'] \&
            \end{tikzcd}.
    \end{split}
    \end{align} \pagebreak[4]

    and for them to intertwine the pentagonators means that for $g_i \in \mathcal{G}^1_0$, $i=1, \,...,\,4$,

    \begin{align}
    \begin{split}
    \begin{tikzcd}[ampersand replacement=\&, scale cd=0.83]
	{\phi(((g_1g_2)g_3)g_4)} \& {\phi((g_1g_2)g_3)\phi(g_4)} \& {(\phi(g_1g_2)\phi(g_3))\phi(g_4)} \& {((\phi(g_1)\phi(g_2))\phi(g_3))\phi(g_4)} \\
	{\phi((g_1(g_2g_3))g_4)} \& {\phi((g_1g_2)(g_3g_4))} \& {\phi(g_1g_2)(\phi(g_3)\phi(g_4))} \& {(\phi(g_1)\phi(g_2))(\phi(g_3)\phi(g_4))} \\
	{\phi(g_1((g_2g_3)g_4))} \& {\phi(g_1(g_2(g_3g_4)))} \& {\phi(g_1g_2)\phi(g_3g_4)} \\
	\&\& {(\phi(g_1)\phi(g_2))\phi(g_3g_4)} \& {\phi(g_1)(\phi(g_2)(\phi(g_3)\phi(g_4)))} \\
	{\phi(g_1)\phi((g_2g_3)g_4)} \& {\phi(g_1)\phi(g_2(g_3g_4))} \&\& {\phi(g_1)(\phi(g_2)\phi(g_3g_4))}
	\arrow["{\alpha(\phi(g_1)\phi(g_2),\phi(g_3),\phi(g_4))}"{name=D}, from=1-4, to=2-4]
	\arrow["{\phi(\alpha(g_1,g_2g_3,g_4))}"', from=2-1, to=3-1]
	\arrow["{\alpha(\phi(g_1),\phi(g_2),\phi(g_3)\phi(g_4))}"{name=E}, from=2-4, to=4-4,pos=0.8]
	\arrow["{\phi((g_1g_2)g_3,g_4)}", from=1-1, to=1-2]
	\arrow["{\phi(\alpha(g_1,g_2,g_3)g_4)}"', from=1-1, to=2-1]
	\arrow["{\phi(g_1g_2,g_3)\phi(g_4)}", from=1-2, to=1-3]
	\arrow["{(\phi(g_1,g_2)\phi(g_3))\phi(g_4)}", from=1-3, to=1-4]
	\arrow["{\phi(g_1)\phi(\alpha(g_2,g_3,g_4))}"', from=5-1, to=5-2]
	\arrow["{\phi(g_1)\phi(g_2,g_3g_4)}"', from=5-2, to=5-4]
	\arrow["{\phi(g_1)(\phi(g_2)\phi(g_3,g_4))}"', from=5-4, to=4-4]
	\arrow["{\phi(g_1,(g_2g_3)g_4)}"', from=3-1, to=5-1]
	\arrow[from=3-1, to=3-2]
	\arrow[from=1-1, to=2-2]
	\arrow[from=2-2, to=3-2, "{}"{name=A}]
	\arrow[from=1-3, to=2-3,"{}"{name=C}]
	\arrow[from=2-2, to=3-3]
	\arrow[from=3-3, to=2-3]
	\arrow[from=2-3, to=2-4,"{}"{name=G},pos=1]
	\arrow[from=3-3, to=4-3]
	\arrow[from=4-3, to=5-4]
	\arrow[from=4-3, to=2-4, "{}"{name=F},pos=0.65] 
	\arrow[from=3-2, to=5-2, "{}"{name=B}]
        \arrow[Rightarrow, from =3-1 , to = A, "{\tilde{\phi}(\Sigma(g_1,g_2,g_3,g_4))}"]
        \arrow[Rightarrow, from =5-1 , to = B, "{\phi(g_1,\alpha(g_2,g_3,g_4))^{-1}}",pos=0.7]
        \arrow[Rightarrow, from =2-2 , to = C, "{\phi(g_1g_2,g_3,g_4)}"]
        \arrow[Rightarrow, from =5-2 , to = 4-3, "{\phi(g_1,g_2,g_3g_4)}"]
        \arrow[Rightarrow, from =2-3 , to = D, "{\alpha(\phi(g_1,g_2),\phi(g_3),\phi(g_4))^{-1}}",pos=0.65]
        \arrow[Rightarrow, from =4-3 , to = E, "{\alpha(\phi(g_1),\phi(g_2),\phi(g_3,g_4))^{-1}}",pos=1,shorten=3ex]
        \arrow[Rightarrow, from =F , to = G, "{I(\phi(g_1,g_2),\phi(g_3,g_4))}"]
        \end{tikzcd} = \\[8ex]
        \begin{tikzcd}[ampersand replacement=\&, scale cd=0.85]
	{\phi(((g_1g_2)g_3)g_4)} \& {\phi((g_1g_2)g_3)\phi(g_4)} \& {(\phi(g_1g_2)\phi(g_3))\phi(g_4)} \\
	\& {\phi(g_1(g_2g_3))\phi(g_4)} \& {((\phi(g_1)\phi(g_2))\phi(g_3))\phi(g_4)} \& {(\phi(g_1)\phi(g_2))(\phi(g_3)\phi(g_4))} \\
	\&\& {(\phi(g_1)(\phi(g_2)\phi(g_3)))\phi(g_4)} \\
	{\phi((g_1(g_2g_3))g_4)} \& {(\phi(g_1)\phi(g_2g_3))\phi(g_4)} \& {\phi(g_1)((\phi(g_2)\phi(g_3))\phi(g_4))} \& {\phi(g_1)(\phi(g_2)(\phi(g_3)\phi(g_4)))} \\
	\&\& {\phi(g_1)(\phi(g_2g_3)\phi(g_4))} \& {\phi(g_1)(\phi(g_2)\phi(g_3g_4))} \\
	\& {\phi(g_1((g_2g_3)g_4))} \& {\phi(g_1)\phi((g_2g_3)g_4)} \& {\phi(g_1)\phi(g_2(g_3g_4))}
	\arrow["{\alpha(\phi(g_1)\phi(g_2),\phi(g_3),\phi(g_4))}", from=2-3, to=2-4]
	\arrow[from=2-3, to=3-3,"{}"{name=B}]
	\arrow["{\phi(\alpha(g_1,g_2g_3,g_4))}"{name=D}, from=4-1, to=6-2,swap]
	\arrow[from=3-3, to=4-3]
	\arrow[from=4-3, to=4-4]
	\arrow["{\alpha(\phi(g_1),\phi(g_2),\phi(g_3)\phi(g_4))}"{name=C}, from=2-4, to=4-4]
	\arrow["{\phi((g_1g_2)g_3,g_4)}", from=1-1, to=1-2]
	\arrow["{\phi(\alpha(g_1,g_2,g_3)g_4)}"{name=P}, from=1-1, to=4-1,swap]
	\arrow[from=1-2, to=2-2,"{}"{name=Q}]
	\arrow[from=4-1, to=2-2]
	\arrow[from=2-2, to=4-2,"{}"{name=A}]
	\arrow[from=4-2, to=3-3,"{}"{name=R}]
	\arrow["{\phi(g_1g_2,g_3)\phi(g_4)}", from=1-2, to=1-3]
	\arrow["{(\phi(g_1,g_2)\phi(g_3))\phi(g_4)}", from=1-3, to=2-3]
	\arrow[from=4-2, to=5-3,"{}"{name=E}]
	\arrow[from=5-3, to=4-3]
	\arrow["{\phi(g_1)\phi(\alpha(g_2,g_3,g_4))}"', from=6-3, to=6-4]
	\arrow["{\phi(g_1)\phi(g_2,g_3g_4)}"', from=6-4, to=5-4]
	\arrow["{\phi(g_1)(\phi(g_2)\phi(g_3,g_4))}"', from=5-4, to=4-4]
	\arrow[from=6-3, to=5-3]
	\arrow["{\phi(g_1,(g_2g_3)g_4)}"', from=6-2, to=6-3]
        \arrow[Rightarrow,"{\Sigma(\phi(g_1),\phi(g_2),\phi(g_3),\phi(g_4))}", from=4-3, to=C,shorten=4ex,pos=0.45]
        \arrow[Rightarrow,"{\tilde{\phi}(g_1,g_2,g_3)\phi(g_4)}", from=A, to=B]
        \arrow[Rightarrow,"{\phi(g_1,g_2g_3,g_4)}", from=D, to=E]
        \arrow[Rightarrow,"{\phi(g_1)\tilde{\phi}(g_2,g_3,g_4)}"', from=6-4, to=5-3]
        \arrow[Rightarrow,"{\phi(\alpha(g_1,g_2,g_3),g_4)}", from=P, to=Q,shorten=3ex,pos=0.7]
        \arrow[Rightarrow,"{\alpha(\phi(g_1),\phi(g_2,g_3),\phi(g_4))^{-1}}"{font=\tiny,pos=0.9}, from=E, to=R,pos=0.9,swap]
\end{tikzcd}
        \end{split}
        \end{align}
    where we are using the notation
    \begin{align}
        \begin{split}
            \tilde{\phi}(g_1,g_2,g_3)\phi(g_4) &:= R_{\phi(g_4)}(\phi(g_1g_2,g_3),\phi(g_1,g_2)\phi(g_3),\alpha(\phi(g_1),\phi(g_2),\phi(g_3))) \circ \phi(g_1,g_2,g_3)\phi(g_4) \\
            &\circ R_{\phi(g_4)}(\phi(\alpha(g_1,g_2,g_3)),\phi(g_1,g_2g_3),\phi(g_1)\phi(g_2g_3))^{-1},\\ 
            \phi(g_1)\tilde{\phi}(g_2,g_3,g_4) &:= L_{\phi(g_1)}(\phi(g_2g_3,g_4),\phi(g_2,g_3)\phi(g_4),\alpha(\phi(g_2),\phi(g_3),\phi(g_4))) \circ \phi(g_1)\phi(g_2,g_3,g_4) \\
            &\circ L_{\phi(g_1)}(\phi(\alpha(g_2,g_3,g_4)),\phi(g_2,g_3g_4),\phi(g_2)\phi(g_3g_4))^{-1},\\
            \tilde{\phi}(\Sigma(g_1,g_2,g_3,g_4)) &:= \phi(\alpha(g_1g_2,g_3,g_4),\alpha(g_1,g_2,g_3g_4)) \circ \phi(\Sigma(g_1,g_2,g_3,g_4)) \\
            &\circ \phi(\alpha(g_1,g_2,g_3)g_4,\alpha(g_1,g_2g_3,g_4),g_1\alpha(g_2,g_3,g_4))^{-1}.
        \end{split}
    \end{align}

\pagebreak[4]
        
\section{Constructions and results}\label{sec:results}

In this section we present the constructions of the Lie 3-groups of interest and the proof of \cref{th:mainintro}. For this, we need to introduce the following notation.

\begin{notation}\label{not:low}
    We identify elements in $\Lambda^3 (\mathbb Z^n)^{\ast}$ and $\Lambda^2 (\mathbb Z^n)^{\ast}$ with skew-symmetric trilinear $H: \mathbb Z^n \times \mathbb Z^n \times \mathbb Z^n \rightarrow \mathbb Z$ and bilinear $B: \mathbb Z^n \times \mathbb Z^n \rightarrow \mathbb Z$ maps, respectively. These extend by linearity to give trilinear and bilinear maps on $\mathbb R^n$ which we denote by the same name. For $H \in \Lambda^3 (\mathbb Z^n)^{\ast}$, $B \in \Lambda^2 (\mathbb Z^n)^{\ast}$, we write 
    \begin{equation}
    H^{low}: \mathbb Z^n \times \mathbb Z^n \times \mathbb Z^n \rightarrow \mathbb Z, \quad B^{low}: \mathbb Z^n \times \mathbb Z^n \rightarrow \mathbb Z
    \end{equation}
    for the trilinear and bilinear maps defined on the standard basis $\{e_1,...,e_n\}$ of $\mathbb Z^n$ by
    \begin{equation}
        H^{low}(e_i,e_j,e_k) := \begin{cases}
        H(e_i,e_j,e_k)      & \quad \text{if } \quad i<j<k  \\
        0 & \quad \text{otherwise }
        \end{cases}, \quad  
        B^{low}(e_i,e_j) := \begin{cases}
        B(e_i,e_j)      & \quad \text{if } \quad i<j  \\
        0 & \quad \text{otherwise }.
        \end{cases} 
    \end{equation}
    These satisfy the following relations:
    \begin{align}
        H(k_1,k_2,k_3) &= \sum_{\sigma \in S_3}\operatorname{sgn}(\sigma)H^{low}(\sigma(k_1),\sigma(k_2),\sigma(k_3)),\\
        B(k_1,k_2) &= B^{low}(k_1,k_2) - B^{low}(k_2,k_1),\\
        (\iota_{m}H)^{low}(k_1,k_2) &= H^{low}(m,k_1,k_2)-H^{low}(k_1,m,k_2)+H^{low}(k_1,k_2,m),
    \end{align}
    where for $m \in \mathbb Z^n$ and $H \in \Lambda^3 (\mathbb Z^n)^{\ast}$ we write $\iota_mH \in \Lambda^2 (\mathbb Z^n)^{\ast}$ for the contraction of $H$ with $m$ in the first component.
    \end{notation}

    \subsection{The Lie 3-groups $T_2 \mathbb D_n^{F_1}$ and $T_2 \mathbb D_n^{F_2}$}\label{sec:t2dnf1}
    
    We define the Lie 3-group $T_2\mathbb D_n^{F_1}$ as follows. Its underlying Lie 2-groupoid has 2-cells of the form
    \begin{equation}
        \begin{tikzcd}[ampersand replacement = \& ]
                (v,B,H_k) \ar[r,bend left = 40, "{(k,B_k,A)}"{name=F}] \ar[r,bend right = 40, "{(k,B_k,A+A_k)}"{name=G},swap] \ar[Rightarrow,from=F,to=G,"{(A_k,z)}",swap] \& (v+k,B+B_k,H_k)
            \end{tikzcd},
    \end{equation}
    where $(v,B,H_k) \in \mathbb R^n \times \Lambda^2 (\mathbb R^n)^{\ast} \times  \Lambda^3 (\mathbb Z^n)^{\ast}$, $(k,B_k,A) \in \mathbb Z^n \times \Lambda^2 (\mathbb Z^n)^{\ast} \times (\mathbb R^n)^{\ast}$ and $(A_k,z) \in (\mathbb Z^n)^{\ast} \times U(1)$. These are composed as follows:
    \begin{align*}
        &\begin{tikzcd}[ampersand replacement = \& ]
                (v,B,H_k) \ar[r,bend left = 40, "{(k,B_k,A)}"{name=F}] \ar[r,bend right = 40, "{(k,B_k,A+A_k)}"{name=G},swap] \ar[Rightarrow,from=F,to=G,"{(A_k,z)}",swap,pos=0.2] \& (v+k,B+B_k,H_k) \ar[r,bend left = 40, "{(k',B_k',A')}"{name=F2}] \ar[r,bend right = 40, "{(k',B_k',A'+A_k')}"{name=G2},swap] \ar[Rightarrow,from=F2,to=G2,"{(A_k',z')}",swap,pos=0.2] \& (v+k+k',B+B_k+B_k',H_k)
            \end{tikzcd} \\
            =&\begin{tikzcd}[ampersand replacement = \&,column sep=15ex ]
                (v,B,H_k) \ar[r,bend left = 40, "{(k+k',B_k+B_k',A+A')}"{name=F}] \ar[r,bend right = 40, "{(k+k',B_k+B_k',A+A'+A_k+A_k')}"{name=G},swap] \ar[Rightarrow,from=F,to=G,"{(A_k+A_k',zz')}",swap,pos=0.35] \& (v+k+k',B+B_k+B_k',H_k),
            \end{tikzcd}  
    \end{align*}
    \begin{equation}
        \begin{tikzcd}[ampersand replacement = \& ,column sep = 12ex]
                (v,B,H_k) \ar[r,bend left = 70, "{(k,B_k,A)}"{name=F}] \ar[r,bend left = 10, "{(k,B_k,A+A_k)}"{name=G},swap,pos=0.05]
                \ar[r,bend right = 70,"{(k,B_k,A+A_k+A_k')}"{name=H},swap]
                \ar[Rightarrow,from=F,to=G,"{(A_k,z)}",swap,pos=0.5,shorten = 1ex]
                \ar[Rightarrow,from=G,to=H,"{(A_k',z')}",swap,pos=0.4]
                \& (v+k,B+B_k,H_k) 
            \end{tikzcd} = 
        \begin{tikzcd}[ampersand replacement = \& ,column sep = 12ex]
                (v,B,H_k) \ar[r,bend left = 70, "{(k,B_k,A)}"{name=F}]
                \ar[r,bend right = 70,"{(k,B_k,A+A_k+A_k')}"{name=H},swap]
                \ar[Rightarrow,from=F,to=H,"{(A_k+A_k',zz')}",swap,pos=0.3]
                \& (v+k,B+B_k,H_k) 
            \end{tikzcd}.
    \end{equation}
    The multiplication functors are 
    \begin{align}
        \begin{split}
            L_{(v_1,B_1,H_{k_1})}: T_2\mathbb D_n^{F_1} &\rightarrow T_2\mathbb D_n^{F_1} \\
            \begin{tikzcd}[ampersand replacement = \&,column sep=2.7ex ]
                {\scriptstyle (v_2,B_2,H_{k_2})} \ar[r,bend left = 60, "{(k_2,B_{k_2},A_2)}"{name=F}] \ar[r,bend right = 40, "{(k_2,B_{k_2},A_2+A_{k_2})}"{name=G},swap] \ar[Rightarrow,from=F,to=G,"{(A_{k_2},z_2)}"{swap,pos=0.3,font=\tiny}] \& {\scriptstyle (v_2+k_2,B_2+B_{k_2},H_{k_2}) }
            \end{tikzcd} &\mapsto
            \begin{tikzcd}[ampersand replacement = \&,column sep = -0.5ex ]
                {\scriptstyle (v_1+v_2,B_1+B_2+\iota_{v_2}H_{k_1},H_{k_1}+H_{k_2}) }\ar[r,bend left = 40, "{(k_2,B_{k_2}+\iota_{k_2}H_{k_1},A_2+\delta(H_{k_1},v_2,k_2))}"{name=F},pos=0.48] \ar[r,bend right = 40, "{(k_2,B_{k_2}+\iota_{k_2}H_{k_1},A_2+\delta(H_{k_1},v_2,k_2)+A_{k_2})}"{name=G},swap,pos=0.48] \ar[Rightarrow,from=F,to=G,"{(A_{k_2},z_2)}",swap,pos=0.3] \& {\scriptstyle (v_1+v_2+k_2,B_1+B_2+\iota_{v_2+k_2}H_{k_1}+B_{k_2},H_{k_1}+H_{k_2}) }
            \end{tikzcd} 
        \end{split},
    \end{align}
    where
    \begin{equation}
        \delta(H_{k_1},v_2,k_2) := -H_{k_1}^{low}(\cdot,k_2,v_2) + H_{k_1}^{low}(k_2,\cdot,v_2) + H_{k_1}^{low}(v_2,k_2,\cdot),
    \end{equation}
    \begin{align}
        \begin{split}
            R_{(v_2,B_2,H_{k_2})}: T_2\mathbb D_n^{F_1} &\rightarrow T_2\mathbb D_n^{F_1} \\
            \begin{tikzcd}[ampersand replacement = \& ]
                {\scriptstyle (v_1,B_1,H_{k_1})} \ar[r,bend left = 60, "{(k_1,B_{k_1},A_1)}"{name=F}] \ar[r,bend right = 40, "{(k_1,B_{k_1},A_1+A_{k_1})}"{name=G},swap] \ar[Rightarrow,from=F,to=G,"{(A_{k_1},z_1)}"{swap,pos=0.3,font=\tiny}] \& {\scriptstyle (v_1+k_1,B_1+B_{k_1},H_{k_1})}
            \end{tikzcd} &\mapsto
            \begin{tikzcd}[ampersand replacement = \&,column sep = -0.5ex ]
                {\scriptstyle (v_1+v_2,B_1+B_2+\iota_{v_2}H_{k_1},H_{k_1}+H_{k_2}) }\ar[r,bend left = 40, "{(k_1,B_{k_1},A_1-\iota_{v_2}B_{k_1})}"{name=F},pos=0.48] \ar[r,bend right = 40, "{(k_1,B_{k_1},A_1-\iota_{v_2}B_{k_1}+A_{k_1})}"{name=G},swap,pos=0.48] \ar[phantom,from=F,to=G,"{(A_{k_1},z_1exp(-\iota_{v_2}A_{k_1}))}"{swap,pos=0.3,font=\footnotesize}] \& {\scriptstyle (v_1+k_1+v_2,B_1+B_{k_1}+B_2 + \iota_{v_2}H_{k_1},H_{k_1}+H_{k_2}) }
            \end{tikzcd}
        \end{split}.
    \end{align}
    In the last diagram, the double arrow has been omitted to avoid overloading the diagram. This is also done in other diagrams throughout the text, assuming that double arrows always go from top to bottom. We equip $R_{(v_2,B_2,H_{k_2})}$ with the trivial compositor, while the compositor for $L_{(v_1,B_1,H_{k_1})}$ acting on 
    $$(v_2,B_2,H_{k_2}) \stackrel{(k_2,B_{k_2},A_2)}{\rightarrow} (v_2+k_2,B_2+B_{k_2},H_{k_2}) \stackrel{(k_2',B_{k_2}',A_2')}{\rightarrow} (v_2+k_2+k_2',B_2+B_{k_2}+B_{k_2}',H_{k_2})$$
    is
    \begin{equation}
    \begin{tikzcd}[column sep = 10ex]
            {\scriptstyle (v_1+v_2,B_1+B_2+\iota_{v_2}H_{k_1},H_{k_1}+H_{k_2})} \ar[rr,bend left=60,"{(k_2+k_2',B_{k_2}+B_{k_2}'+\iota_{k_2+k_2'}H_{k_1},A_2+A_2' + \delta(H_1,v_2,k_2+k_2'))}"{name=F},pos=0.53] 
              \ar[rr,bend right=10,"{(k_2+k_2',B_{k_2}+\iota_{k_2}H_{k_1}+B_{k_2}'+\iota_{k_2'}H_{k_1},A_2+A_2'+\delta(H_1,v_2,k_2) + \delta(H_1,v_2+k_2,k_2'))}"{name=G},swap,pos=0.52]
             & & {\scriptstyle (v_1+v_2+k_2+k_2',B_1+B_2+\iota_{v_2+k_2+k_2'}H_{k_1}+B_{k_2}+B_{k_2}',H_{k_1}+H_{k_2})} \\ \ar[phantom,from=F,to=G,"{L_{(v_1,B_1,H_{k_1})}(v_2,k_2,B_{k_2},A_2,k_2',B_{k_2}',A_2')}",shorten >=1.5pt] &
        \end{tikzcd},
    \end{equation}
    where
    \begin{align}
        \begin{split}
            L_{(v_1,B_1,H_{k_1})}(v_2,&k_2,B_{k_2},A_2,k_2',B_{k_2}',A_2') \\
        &= \left(- H_{k_1}^{low}(\cdot,k_2',k_2) + H_{k_1}^{low}(k_2',\cdot,k_2) + H_{k_1}^{low}(k_2,k_2',\cdot), exp(H_{k_1}^{low}(v_2,k_2',k_2)) \right).
        \end{split}
    \end{align}
    The interchange 2-cells are defined for $(v_i,B_i,H_{k_i}) \stackrel{(k_i,B_{k_i},A_i)}{\rightarrow} (v_i+k_i,B_i+B_{k_i},H_{k_i})$, $i=1, \,2$ by
    \begin{equation}
            \begin{tikzcd}[column sep=15ex,row sep = 10ex]
             {\scriptstyle (v_1+v_2,B_1+B_2+\iota_{v_2}H_{k_1},H_{k_1}+H_{k_2})} \ar[r,"{(k_1,B_{k_1},A_1-\iota_{v_2}B_{k_1})}",{name=U}] \ar[d,"{\substack{(k_2,B_{k_2}+\iota_{k_2}H_{k_1},\\ A_2-\delta(H_{k_1},v_2,k_2))}}",swap] 
             & {\scriptstyle (v_1+k_1+v_2,B_1+B_{k_1}+B_2+\iota_{v_2}H_{k_1},H_{k_1}+H_{k_2}) } \ar[d,"{\substack{(k_2,B_{k_2}+\iota_{k_2}H_{k_1},\\ A_2-\delta(H_{k_1},v_2,k_2))}}"] \ar[Rightarrow, dl,"{(-\iota_{k_2}B_{k_1},exp(-B_{k_1}^{low}(v_2,k_2)))}",swap,pos=0.5]\\
             {\scriptstyle (v_1+v_2+k_2,B_1+B_2+\iota_{v_2+k_2}H_{k_1}+B_{k_2},H_{k_1}+H_{k_2}) }\ar[r,"{(k_1,B_{k_1},A_1-\iota_{v_2+k_2}B_{k_1})}",{name=D},swap] 
             & {\scriptstyle (v_1+k_1+v_2+k_2,B_1+B_{k_1}+B_2+\iota_{v_2+k_2}H_{k_1}+B_{k_2},H_{k_1}+H_{k_2})}.
            \end{tikzcd}
            \end{equation}
    For $(v_i,B_i,H_{k_i})$, $i=1, \,2,\,3$, the associator arrows are
    \begin{equation}
    \begin{tikzcd}
        (v_1+v_2+v_3,B_1+B_2+\iota_{v_2}H_{k_1}+B_3 + \iota_{v_3}(H_{k_1}+H_{k_2}),H_{k_1}+H_{k_2}+H_{k_3}) \ar[d,"{(0,0,-H_{k_1}^{low}(\cdot,v_3,v_2) + H_{k_1}^{low}(v_3,\cdot,v_2) + H_{k_1}^{low}(v_2,v_3,\cdot))}"] \\
        (v_1+v_2+v_3,B_1+B_2+B_3+\iota_{v_3}H_{k_2} + \iota_{v_2+v_3}H_{k_1},H_{k_1}+H_{k_2}+H_{k_3}),
    \end{tikzcd}
    \end{equation}
    and the associator 2-cells are defined on $(v_i,B_i,H_{k_i}) \stackrel{(k_i,B_{k_i},A_i)}{\rightarrow} (v_i+k_i,B_i+B_{k_i},H_{k_i})$, $i=1, \,2,\, 3$ by
    \begin{equation}
            \begin{tikzcd}[column sep = 3ex]
            {\substack{(v_1+v_2+v_3,\\ B_1+B_2+B_3+\iota_{v_2+v_3}H_{k_1}+\iota_{v_3}H_{k_2},H_{k_1}+H_{k_2}+H_{k_3})}} \ar[rr,bend left=40,"{(k_1,B_{k_1},A_1-\iota_{v_2}B_{k_1}-\iota_{v_3}B_{k_1} - H_{k_1}^{low}(\cdot,v_3,v_2) + H_{k_1}^{low}(v_3,\cdot,v_2) + H_{k_1}^{low}(v_2,v_3,\cdot))}"{name=F}] 
              \ar[rr,bend right=10,"{(k_1,B_{k_1}, - H_{k_1}^{low}(\cdot,v_2,v_3) + H_{k_1}^{low}(v_3,\cdot,v_2) + H_{k_1}^{low}(v_2,v_3,\cdot) + A_1-\iota_{v_2+v_3}B_{k_1})}"{name=G},swap]
             & & {\substack{(v_1+k_1+v_2+v_3,\\ B_1+B_{k_1}+B_2+B_3+\iota_{v_2+v_3}H_{k_1}+\iota_{v_3}H_{k_2},H_{k_1}+H_{k_2}+H_{k_3})}}\\ \ar[Rightarrow,from=F,to=G,"{(0,exp(-B_{k_1}^{low}(v_2,v_3)))}"{swap,pos=0.5},shorten >=1.5pt] &
            \end{tikzcd},
    \end{equation}

    \begin{equation}
            \begin{tikzcd}[column sep = 1ex]
            {\scriptstyle (v_1+v_2+v_3,B_1+B_2+B_3+\iota_{v_2+v_3}H_{k_1}+\iota_{v_3}H_{k_2},H_{k_1}+H_{k_2}+H_{k_3})} \ar[rr,bend left=70,"{(k_2,B_{k_2}+\iota_{k_2}H_{k_1},A_2 + \delta(H_{k_1},v_2,k_2)-\iota_{v_3}(B_{k_2}+\iota_{k_2}H_{k_1}) - H_{k_1}^{low}(\cdot,v_3,v_2+k_2) + H_{k_1}^{low}(v_3,\cdot,v_2+k_2) + H_{k_1}^{low}(v_2+k_2,v_3,\cdot) )}"{name=F},pos=0.5] 
              \ar[rr,bend right=10,"{(k_2,B_{k_2}+\iota_{k_2}H_1, - H_{k_1}^{low}(\cdot,v_2,v_3) + H_{k_1}^{low}(v_3,\cdot,v_2) + H_{k_1}^{low}(v_2,v_3,\cdot) + A_2-\iota_{v_3}B_{k_2} + \delta(H_{k_1},v_2+v_3,k_2))}"{name=G},swap]
             & & {\scriptstyle (v_1+v_2+v_3,B_1+B_2+B_3+\iota_{v_2+v_3}H_{k_1}+\iota_{v_3}H_{k_2},H_{k_1}+H_{k_2}+H_{k_3})} \\ \ar[phantom,from=F,to=G,"{(0,exp(-H_1^{low}(v_2+k_2,v_3,k_2) +H_1^{low}(v_2,k_2,v_3) ))}"{swap,pos=0.5,font=\footnotesize},shorten >=1.5pt] &
            \end{tikzcd},
    \end{equation}
        
    \begin{equation}
            \begin{tikzcd}[column sep = 1.5ex]
            {\scriptstyle (v_1+v_2+v_3,B_1+B_2+B_3+\iota_{v_2+v_3}H_{k_1}+\iota_{v_3}H_{k_2},H_{k_1}+H_{k_2}+H_{k_3})} \ar[rr,bend left=70,"{(k_3,B_{k_3}+\iota_{k_3}(H_{k_1}+H_{k_2}),A_3 + \delta(H_{k_1}+H_{k_2},v_3,k_3) - H_{k_1}^{low}(\cdot,v_3+k_3,v_2) + H_{k_1}^{low}(v_3+k_3,\cdot,v_2) + H_{k_1}^{low}(v_2,v_3+k_3,\cdot) )}"{name=F},pos=0.5] 
              \ar[rr,bend right=10,"{(k_3,B_{k_3}+\iota_{k_3}H_{k_2}+\iota_{k_3}H_{k_1}, - H_{k_1}^{low}(\cdot,v_2,v_3) + H_{k_1}^{low}(v_3,\cdot,v_2) + H_{k_1}^{low}(v_2,v_3,\cdot) + A_3+ \delta(H_{k_2},v_3,k_3)  +\delta(H_{k_1},v_2+v_3,k_3)}"{name=G},swap]
             & & {\scriptstyle (v_1+v_2+v_3,B_1+B_2+B_3+\iota_{v_2+v_3}H_{k_1}+\iota_{v_3}H_{k_2},H_{k_1}+H_{k_2}+H_{k_3})} \\ \ar[Rightarrow,from=F,to=G,"{(0,exp(-H_1^{low}(v_2,v_3+k_3,k_3)))}"{swap,pos=0.5,font=\footnotesize},shorten = 3ex] &
            \end{tikzcd}.
    \end{equation}
Finally, the pentagonator of $(v_i,B_i,H_{k_i})$, $i= 1, \,...,\,4$ is
    \begin{equation}
            \begin{tikzcd}[column sep = 5ex]
            {\substack{(v_1+v_2+v_3+v_4, \\
            B_1+B_2+B_3+B_4+\iota_{v_2+v_3+v_4}H_{k_1}+\iota_{v_3+v_4}H_{k_2}+\iota_{v_4}H_{k_3}, \\
            H_{k_1}+H_{k_2}+H_{k_3}+H_{k_4})}} \ar[rr,bend left=60,"{\substack{(0,0,- H_{k_1}(\cdot,v_3,v_2) + H_{k_1}(v_3,\cdot,v_2) + H_{k_1}(v_2,v_3,\cdot)  - H_{k_1}(\cdot,v_4,v_2+v_3) + H_{k_1}(v_4,\cdot,v_2+v_3) + H_{k_1}(v_2+v_3,v_4,\cdot)  \\ - H_{k_2}(\cdot,v_4,v_3) + H_{k_2}(v_4,\cdot,v_3) + H_{k_2}(v_3,v_4,\cdot) )}}"{name=F}] 
              \ar[rr,bend right=10,"{(0,0,- (H_{k_1}+H_{k_2})(\cdot,v_4,v_3) + (H_{k_1}+H_{k_2}(v_4,\cdot,v_3) + (H_{k_1}+H_{k_2})(v_3,v_4,\cdot) - H_{k_1}(\cdot,v_3+v_4,v_2) + H_{k_1}(v_3+v_4,\cdot,v_2) + H_{k_1}(v_2,v_3+v_4,\cdot) )}"{name=G},swap]
             & & {\substack{(v_1+v_2+v_3+v_4, \\
             B_1+B_2+B_3+B_4+\iota_{v_2+v_3+v_4}H_{k_1}+\iota_{v_3+v_4}H_{k_2}+\iota_{v_4}H_{k_3}, \\
             H_{k_1}+H_{k_2}+H_{k_3}+H_{k_4})}} \\ \ar[Rightarrow,from=F,to=G,"{(0,exp(-H_{k_1}^{low}(v_2,v_3,v_4)))}"{swap,pos=0.4},shorten >=1.5pt] &
            \end{tikzcd}.
            \end{equation}

    Checking that all axioms are satisfied amounts to tedious but straightforward computations that we include in \cref{sec:welldefined}. The Lie 3-group $T_2 \mathbb D_n^{F_2}$ is defined to be the Lie sub-3-group of $T_2 \mathbb B_n^{F_1}$ obtained by setting $H_k = 0$.

    \subsection{The Lie 3-groups $T_2 \mathbb B_n^{F_1}$ and $T_2 \mathbb B_n^{F_2}$}\label{sec:t2bnf1}
    
    We define the Lie 3-group $T_2\mathbb B_n^{F_1}$ as follows. Its underlying Lie 2-groupoid is 
    \begin{equation}
    T_2\mathbb B_n^{F_1} := \mathbb R^n /\!/ \mathbb Z^n \times \Lambda^3 (\mathbb Z^n)^{\ast} \times B \Lambda^2 (\mathbb Z^n)^{\ast} \times B^2C^{\infty}(\mathbb R^n/\mathbb Z^n,U(1)),
    \end{equation}
    where $\mathbb R^n /\!/ \mathbb Z^n$ denotes the quotient groupoid for the standard action of $\mathbb Z^n$ on $\mathbb R^n$. The multiplication functors are
    \begin{align}
        \begin{split}
            L_{(v_1,H_{k_1})}: T_2\mathbb B_n^{F_1} &\rightarrow T_2\mathbb B_n^{F_1} \\
            \begin{tikzcd}[ampersand replacement = \& ]
                (v_2,H_{k_2}) \ar[r,bend left = 40, "{(k_2,B_{k_2})}"{name=F}] \ar[r,bend right = 40, "{(k_2,B_{k_2})}"{name=G},swap] \ar[Rightarrow,from=F,to=G,"{f_2}",swap] \& (v_2+k_2,H_{k_2})
            \end{tikzcd} &\mapsto
            \begin{tikzcd}[ampersand replacement = \& ]
                (v_1+v_2,H_{k_1}+H_{k_2}) \ar[r,bend left = 40, "{(k_2,B_{k_2}+\iota_{k_2}H_{k_1})}"{name=F}] \ar[r,bend right = 40, "{(k_2,B_{k_2}+\iota_{k_2}H_{k_1})}"{name=G},swap] \ar[Rightarrow,from=F,to=G,"{f_2}",swap] \& (v_1+v_2+k_2,H_{k_1}+H_{k_2})
            \end{tikzcd} 
        \end{split},
    \end{align}
    \begin{align}
        \begin{split}
            R_{(v_2,H_{k_2})}: T_2\mathbb B_n^{F_1} &\rightarrow T_2\mathbb B_n^{F_1} \\
            \begin{tikzcd}[ampersand replacement = \& ]
                (v_1,H_{k_1}) \ar[r,bend left = 40, "{(k_1,B_{k_1})}"{name=F}] \ar[r,bend right = 40, "{(k_1,B_{k_1})}"{name=G},swap] \ar[Rightarrow,from=F,to=G,"{f_1}",swap] \& (v_1+k_1,H_{k_1})
            \end{tikzcd} &\mapsto
            \begin{tikzcd}[ampersand replacement = \& ]
                (v_1+v_2,H_{k_1}+H_{k_2}) \ar[r,bend left = 40, "{(k_1,B_{k_1})}"{name=F}] \ar[r,bend right = 40, "{(k_1,B_{k_1})}"{name=G},swap] \ar[Rightarrow,from=F,to=G,"{v_2^{\ast}f_1}",swap] \& (v_1+k_1+v_2,H_{k_1}+H_{k_2})
            \end{tikzcd} 
        \end{split},
    \end{align}

    where for $v_2 \in \mathbb R^n$ and $f_1 \in C^{\infty}(\mathbb R^n/\mathbb Z^n,U(1))$ we write $v_2^{\ast}f_1 \in C^{\infty}(\mathbb R^n/\mathbb Z^n,U(1))$ for $v_2^{\ast}f_1(\cdot) := f_1(\cdot - v_2)$. We equip $R_{(v_2,H_{k_2})}$ with a trivial compositor, while the compositor for $L_{(v_1,H_{k_1})}$ acting on 
    \begin{equation}
    (v_2,H_{k_2}) \stackrel{(k_2,B_{k_2})}{\rightarrow} (v_2+k_2,H_{k_2}) \stackrel{(k_2',B_{k_2}')}{\rightarrow} (v_2+k_2+k_2',H_{k_2})
    \end{equation}
    is
    \begin{equation}
    \begin{tikzcd}[column sep = 11ex]
            (v_1+v_2,H_{k_1}+H_{k_2}) \ar[rr,bend left=60,"{(k_2+k_2',B_{k_2}+B_{k_2}'+\iota_{k_2+k_2'}H_{k_1})}"{name=F},pos=0.53] 
              \ar[rr,bend right=10,"{(k_2+k_2',B_{k_2}+\iota_{k_2}H_{k_1}+B_{k_2}'+\iota_{k_2'}H_{k_1}))}"{name=G},swap]
             & & (v_1+v_2+k_2+k_2',H_{k_1}+H_{k_2}) \\ \ar[Rightarrow,from=F,to=G,"{L_{(v_1,H_{k_1})}(v_2,k_2,B_{k_2},k_2',B_{k_2}')}"{swap,pos=0.7},shorten >=1.5pt] &
        \end{tikzcd},
    \end{equation}
    where
    \begin{equation}
        L_{(v_1,H_{k_1})}(v_2,k_2,B_{k_2},k_2',B_{k_2}') = exp\left(H_{k_1}^{low}(v_2,k_2',k_2) - H_{k_1}^{low}(\cdot,k_2',k_2) + H_{k_1}^{low}(k_2',\cdot,k_2) + H_{k_1}^{low}(k_2,k_2',\cdot)\right).
    \end{equation}
    The interchange 2-cells are defined for $(v_i,H_{k_i}) \stackrel{(k_i,B_{k_i})}{\rightarrow} (v_i+k_i,H_{k_i})$, $i=1, \,2$ by
    \begin{equation}
            \begin{tikzcd}[column sep=14ex]
             (v_1+v_2,H_{k_1}+H_{k_2}) \ar[r,"{(k_1,B_{k_1})}",{name=U}] \ar[d,"{(k_2,B_{k_2}+\iota_{k_2}H_{k_1})}",swap] 
             & (v_1+k_1+v_2,H_{k_1}+H_{k_2}) \ar[d,"{(k_2,B_{k_2}+\iota_{k_2}H_{k_1})}"] \ar[Rightarrow, dl,"{exp(-\iota_{k_2}B_{k_1}-B_{k_1}^{low}(v_2,k_2))}",sloped,pos=0.65]\\
             (v_1+v_2+k_2,H_{k_1}+H_{k_2}) \ar[r,"{(k_1,B_{k_1})}",{name=D},swap] 
             & (v_1+v_2+k_1+k_2,H_{k_1}+H_{k_2}).
            \end{tikzcd}
            \end{equation}
    The associator arrows are all identities, while the associator 2-cells are defined on $(v_i,H_{k_i}) \stackrel{(k_i,B_{k_i})}{\rightarrow} (v_i+k_i,H_{k_i})$, $i=1, \,2,\, 3$ by
    \begin{equation}
            \begin{tikzcd}[column sep = 6ex]
            (v_1+v_2+v_3,H_{k_1}+H_{k_2}+H_{k_3}) \ar[rr,bend left=40,"{(k_1,B_{k_1})}"{name=F}] 
              \ar[rr,bend right=10,"{(k_1,B_{k_1})}"{name=G},swap]
             & & (v_1+k_1+v_2+v_3,H_{k_1}+H_{k_2}+H_{k_3}) \\ \ar[Rightarrow,from=F,to=G,"{exp(-B_{k_1}^{low}(v_2,v_3))}"{swap,pos=0.5},shorten >=1.5pt] &
            \end{tikzcd},
    \end{equation}

    \begin{equation}
            \begin{tikzcd}[column sep = 6ex]
            (v_1+v_2+v_3,H_{k_1}+H_{k_2}+H_{k_3}) \ar[rr,bend left=70,"{(k_2,B_{k_2}+\iota_{k_2}H_{k_1})}"{name=F},pos=0.5] 
              \ar[rr,bend right=10,"{(k_2,B_{k_2}+\iota_{k_2}H_{k_1})}"{name=G},swap]
             & & (v_1+v_2+k_2+v_3,H_{k_1}+H_{k_2}+H_{k_3}) \\ \ar[phantom,from=F,to=G,"{exp(-H_{k_1}^{low}(v_2+k_2,v_3,k_2) +H_{k_1}^{low}(v_2,k_2,v_3) )}"{swap,pos=0.5,font=\footnotesize},shorten >=1.5pt] &
            \end{tikzcd},
    \end{equation}

    \begin{equation}
            \begin{tikzcd}[column sep = 6ex]
            (v_1+v_2+v_3,H_{k_1}+H_{k_2}+H_{k_3}) \ar[rr,bend left=70,"{(k_3,B_{k_3}+\iota_{k_3}(H_{k_1}+H_{k_2}))}"{name=F},pos=0.55] 
              \ar[rr,bend right=10,"{(k_3,B_{k_3}+\iota_{k_3}H_{k_2}+\iota_{k_3}H_{k_1})}"{name=G},swap]
             & & (v_1+v_2+v_3+k_3,H_{k_1}+H_{k_2}+H_{k_3}) \\ \ar[Rightarrow,from=F,to=G,"{exp(-H_{k_1}^{low}(v_2,v_3+k_3,k_3))}"{swap,pos=0.5,font=\footnotesize},shorten = 3ex] &
            \end{tikzcd}.
    \end{equation}
    Finally, the pentagonator of $(v_i,H_{k_i})$, $i= 1, \,...,\,4$ is
    \begin{equation}
            \begin{tikzcd}[column sep = 5ex]
            (v_1+v_2+v_3+v_4,H_{k_1}+H_{k_2}+H_{k_3}+H_{k_4}) \ar[rr,bend left=30,"{id}"{name=F}] 
              \ar[rr,bend right=10,"{id}"{name=G},swap]
             & & (v_1+v_2+v_3+v_4,H_{k_1}+H_{k_2}+H_{k_3}+H_{k_4}) \\ \ar[Rightarrow,from=F,to=G,"{exp(-H_{k_1}^{low}(v_2,v_3,v_4))}"{swap,pos=0.4},shorten >=1.5pt] &
            \end{tikzcd}.
            \end{equation}    
    Checking that all axioms are satisfied amounts to tedious but straightforward computations that we include in \cref{sec:welldefined}. The Lie 3-group $T_2 \mathbb B_n^{F_2}$ is defined to be the Lie sub-3-group of $T_2 \mathbb B_n^{F_1}$ obtained by setting $H_k = 0$.

    \begin{proposition}\label{prop:t2dnwhet2bn}
    There is a homomorphism of Lie 3-groups $T_2\mathbb B_n^{F_1} \leftarrow T_2\mathbb D_n^{F_1}$ restricting to $T_2\mathbb B_n^{F_2} \leftarrow T_2\mathbb D_n^{F_2}$ and inducing homotopy equivalences in both cases.
    \end{proposition}
    \begin{proof}

    By inspection of the formulas that define both 3-groups, it is clear that there are homomorphisms of Lie 3-groups $T_2\mathbb D_n^{F_1} \rightarrow T_2\mathbb B_n^{F_1}$ obtained simply by forgetting the components $B \in \Lambda^2 (\mathbb R^n)^{\ast}$ and $A \in (\mathbb R^n)^{\ast}$, while embedding $(\mathbb Z^n)^{\ast} \times U(1)\rightarrow C^{\infty}(\mathbb R^n/\mathbb Z^n,U(1))$, $(A_k,z) \mapsto z \cdot \exp(A_k(\cdot))$. This embedding is a homotopy equivalence \cite{Nikolaus:2018qop}, and the spaces $\Lambda^2 (\mathbb R^n)^{\ast}$ and $(\mathbb R^n)^{\ast}$ are contractible. Therefore, $T_2\mathbb D_n^{F_1} \rightarrow T_2\mathbb B_n^{F_1}$ is a homotopy equivalence. The fact that it restricts to $T_2\mathbb D_n^{F_2} \rightarrow T_2\mathbb B_n^{F_2}$ and is also a homotopy equivalence there is straightforward.
    \end{proof}

    \subsection{The 2-category of gerbes over $\mathbb R^n/\mathbb Z^n$}\label{sec:appgerbes}
  
    Consider the 2-category $\mathcal{G}(\mathbb R^n/\mathbb Z^n)$ of gerbes over $\mathbb R^n/\mathbb Z^n$. Note that gerbes can be multiplied with each other by tensor product, and that they can be pulled back by diffeomorphisms of the form $[v]:\mathbb R^n/ \mathbb Z^n \rightarrow \mathbb R^n/\mathbb Z^n$, $[x] \mapsto [x-v]$ for $[v] \in \mathbb R^n/\mathbb Z^n$. This defines a 3-group structure on the 2-groupoid $\mathbb R^n/ \mathbb Z^n \times \mathcal{G}(\mathbb R^n/\mathbb Z^n)$, described by the multiplication pseudofunctors
    \begin{align}
        \begin{split}
            L_{([v_1],\mathcal{L}_1)}: \mathbb R^n/ \mathbb Z^n \ltimes \mathcal{G}(\mathbb R^n/\mathbb Z^n) &\rightarrow \mathbb R^n/ \mathbb Z^n \ltimes \mathcal{G}(\mathbb R^n/\mathbb Z^n) \\
            \begin{tikzcd}[ampersand replacement = \& ]
                ([v_2],\mathcal{L}_2) \ar[r,bend left = 40, "{\gamma_2}"{name=F}] \ar[r,bend right = 40, "{\gamma_2'}"{name=G},swap] \ar[Rightarrow,from=F,to=G,"{\psi_2}",swap] \& ([v_2],\mathcal{L}_2')
            \end{tikzcd} &\mapsto
            \begin{tikzcd}[ampersand replacement = \& ]
                ([v_1+v_2],[v_2]^{\ast}\mathcal{L}_1 \otimes \mathcal{L}_2) \ar[r,bend left = 40, "{id \otimes \gamma_2}"{name=F}] \ar[r,bend right = 40, "{id \otimes \gamma_2'}"{name=G},swap] \ar[Rightarrow,from=F,to=G,"{id \otimes \psi_2}",swap,pos=0.2] \& ([v_1+v_2],[v_2]^{\ast}\mathcal{L}_1 \otimes \mathcal{L}_2')
            \end{tikzcd} 
        \end{split},
    \end{align}
    \begin{align}
        \begin{split}
            R_{([v_2],\mathcal{L}_2)}:  \mathbb R^n/ \mathbb Z^n \ltimes \mathcal{G}(\mathbb R^n/\mathbb Z^n) &\rightarrow \mathbb R^n/ \mathbb Z^n \ltimes \mathcal{G}(\mathbb R^n/\mathbb Z^n) \\
            \begin{tikzcd}[ampersand replacement = \& ]
                ([v_1],\mathcal{L}_1) \ar[r,bend left = 40, "{\gamma_1}"{name=F}] \ar[r,bend right = 40, "{\gamma_1'}"{name=G},swap] \ar[Rightarrow,from=F,to=G,"{\psi_1}",swap] \& ([v_1],\mathcal{L}_1')
            \end{tikzcd} &\mapsto
            \begin{tikzcd}[ampersand replacement = \& ]
                ([v_1+v_2],[v_2]^{\ast}\mathcal{L}_1 \otimes \mathcal{L}_2) \ar[r,bend left = 40, "{[v_2]^{\ast}\gamma_1 \otimes id}"{name=F}] \ar[r,bend right = 40, "{[v_2]^{\ast}\gamma_1' \otimes id}"{name=G},swap] \ar[Rightarrow,from=F,to=G,"{[v_2]^{\ast}\psi_1 \otimes id}",swap,pos=0.2] \& ([v_1+v_2],[v_2]^{\ast}\mathcal{L}_1' \otimes \mathcal{L}_2)
            \end{tikzcd} 
        \end{split}.
    \end{align}
    The compositor, interchange, associator, and pentagonator 2-cells, as well as the associator arrows, are all gerbe isomorphisms or 2-isomorphisms that exist because fibered products of manifolds are canonically isomorphic. For notational simplicity, we regard these as identities, as is usual in the literature.

    \begin{theorem}\label{th:t2bnf1isgerbes}
    The Lie 3-group $T_2\mathbb B_n^{F_1}$ is isomorphic to the 3-group $\mathbb R^n/\mathbb Z^n \ltimes \mathcal{G}(\mathbb R^n/\mathbb Z^n)$.
    \end{theorem}
    \begin{proof}
    Recall first that \eqref{eq:gxcohomology} describes the 2-category of gerbes over a space $X$. In particular, for $X = \mathbb R^n/\mathbb Z^n$ we obtain 
    \begin{equation}\label{eq:gerb2cat}
        \mathcal{G}(\mathbb R^n/ \mathbb Z^n) \cong \Lambda^3 (\mathbb Z^n)^{\ast} \times B \Lambda^2 (\mathbb Z^n)^{\ast} \times B^2C^{\infty}(\mathbb R^n/\mathbb Z^n,U(1)) =: \mathcal{G}_{red}(\mathbb R^n/ \mathbb Z^n).
    \end{equation}
    Note that $T_2\mathbb B_n^{F_1}$ is precisely a 3-group structure on the 2-category $\mathbb R^n/\mathbb Z^n \times \mathcal{G}_{red}(\mathbb R^n/\mathbb Z^n)$. We prove the by constructing an isomorphism between this 3-group structure and that on $\mathbb R^n/\mathbb Z^n \ltimes \mathcal{G}(\mathbb R^n/\mathbb Z^n)$.

    First, we construct an explicit functor \eqref{eq:gerb2cat} or, equivalently, a functor $\phi: T_2\mathbb B_n^{F_1} \rightarrow \mathbb R^n/\mathbb Z^n \ltimes \mathcal{G}(\mathbb R^n/\mathbb Z^n)$. For $H \in \Lambda^3 (\mathbb Z^n)^{\ast}$ we let $\mathcal{L}^H$ be the gerbe with $\mathcal{L}^H_0 = \mathbb R^n$, $\mathcal{L}^H_1 = \mathbb R^n \times \mathbb Z^n \times U(1)$, composition
    \begin{equation}
        \begin{tikzcd}
            x \ar[r,bend left=40,"{(k_1,z_1)}"] & x+ k_1 \ar[r,bend left=40,"{(k_2,z_2)}"] &x+ k_1+k_2
        \end{tikzcd} =
        \begin{tikzcd}[column sep = 20ex]
            x \ar[r,bend left=40,"{(k_1+k_2,z_1z_2exp(H^{low}(x,k_1,k_2))}"] & x+ k_1+k_2,
        \end{tikzcd}
    \end{equation}
    and projection $\mathcal{L}^H \rightarrow \mathbb R^n/\mathbb Z^n$, $x \mapsto [x]$. For $B \in \Lambda^2 (\mathbb Z^n)^{\ast}$ we define a functor
    \begin{align}
    \begin{split}
        L^B:\mathcal{L}^H &\rightarrow \mathcal{L}^H \\
        \begin{tikzcd}[ampersand replacement = \&]
            x \ar[r,bend left=40,"{(k,z)}"] \& x+ k
        \end{tikzcd} &\mapsto \begin{tikzcd}[ampersand replacement = \&]
            x \ar[r,bend left=40,"{(k,zexp(B^{low}(x,k))}",pos=0.55] \& x+ k.
        \end{tikzcd}
    \end{split}
    \end{align}
    For $f \in C^{\infty}(\mathbb R^n/\mathbb Z^n,U(1))$ we define the natural transformation 
    \begin{align}
    \begin{split}
        f:L^B &\Rightarrow L^B \\
        x &\mapsto \begin{tikzcd}[ampersand replacement = \&] x \ar[r,bend left=40,"{(0,f(x))}"] \& x \end{tikzcd}.
    \end{split}
    \end{align}
    This concludes the description of $\phi$, which preserves vertical and horizontal composition strictly.

    We proceed now to describe the action of $\mathbb R^n/\mathbb Z^n$ on $\mathcal{G}_{red}(\mathbb R^n/ \mathbb Z^n)$ in terms of $\phi$, which will let us define the multiplicator arrows and 2-cells of $\phi$. For $[v] \in \mathbb R^n/\mathbb Z^n$ we note that the gerbe $[v]^{\ast}\mathcal{L}^H$ is defined by the same formulas as $\mathcal{L}^H$ except for the projection to $\mathbb R^n/\mathbb Z^n$, which changes to $x \mapsto [x-v]$. For an explicit representative $v \in \mathbb R^n$ we define the isomorphism
    \begin{align}
    \begin{split}
        \Phi_{v,H}: \mathcal{L}^H &\rightarrow [v]^{\ast}\mathcal{L}^H\\
        \begin{tikzcd}[ampersand replacement = \&]
            x \ar[r,bend left=40,"{(k,z)}"] \& x+ k
        \end{tikzcd} &\mapsto \begin{tikzcd}[ampersand replacement = \&]
            x-v \ar[r,bend left=40,"{(k,zexp(H^{low}(v,x,k))}",pos=0.55] \& x-v+ k
        \end{tikzcd}.
    \end{split}
    \end{align}
    This indeed preserves composition because
    \begin{equation}
        H^{low}(v,x,k_1) + H^{low}(v,x+k_1,k_2) + H^{low}(x-v,k_1,k_2) = H^{low}(x,k_1,k_2) + H^{low}(v,x,k_1+k_2).
    \end{equation}
    Thus, we may define the multiplicator arrows of $\phi$ by
    \begin{equation}
        \phi((v_1,H_1),(v_2,H_2)): \mathcal{L}^{H_1+H_2} \stackrel{\Phi_{v_2,H_1} \otimes id}{\rightarrow} [v_2]^{\ast}\mathcal{L}^{H_1} \otimes \mathcal{L}^{H_2}.
    \end{equation}
    For different representatives $v$, $v+k$, we define a natural transformation 
    \begin{align}
    \begin{split}
        \Psi_{v,k,H}:\Phi_{v+k,H} &\Rightarrow \Phi_{v,H} \circ L^{\iota_kH} \\
        x &\mapsto \begin{tikzcd}[ampersand replacement = \&] x-v-k \ar[r,bend left=40,"{(k,exp(\Psi_{v,k,H}(x)))}"] \& x-v \end{tikzcd},
    \end{split}
    \end{align}
    where $\Psi_{v,k,H}(x) := H^{low}(v+k,x,k) - H^{low}(v,k,x)$. This is indeed natural because
    \begin{align}
    \begin{split}
        H^{low}(v+k,x,k') + &\Psi_{v,k,H}(x+k') + H^{low}(x-v-k,k',k) \\
        &=  \Psi_{v,k,H}(x) + (\iota_kH)^{low}(x,k') + H^{low}(v,x,k') + H^{low}(x-v-k,k,k').
    \end{split}
    \end{align}
    We note then that
    \begin{align}
        \Phi_{v+k+k',H} \stackrel{\Psi_{v+k,k',H}}{\Rightarrow} \Phi_{v+k,H} \circ L^{\iota_{k'}H} \stackrel{\Psi_{v,k,H} \circ id_{L^{\iota_{k'}H}}}{\Rightarrow} \Phi_{v,H} \circ L^{\iota_{k+k'}H}
    \end{align}
    equals
    \begin{equation}\label{eq:relation1}
        \Psi_{v,k+k',H} \cdot exp(-H^{low}(v,k',k) + H^{low}(\cdot,k',k) - H^{low}(k',\cdot,k) - H^{low}(k,k',\cdot)).
    \end{equation}
    Next we define natural transformations
    \begin{equation}
            \begin{tikzcd}[column sep=14ex]
             \mathcal{L}^H \ar[r,"{L^B}",{name=U}] \ar[d,"{\Phi_{v,H}}",swap] 
             & \mathcal{L}^H \ar[d,"{\Phi_{v,H}}"] \ar[Rightarrow, dl,"{\rho_{v,B}}",swap]\\
             \left[v\right]^{\ast}\mathcal{L}^H \ar[r,"{[v]^{\ast}L^B}",{name=D},swap] 
             & \left[v\right]^{\ast}\mathcal{L}^H
            \end{tikzcd}
    \end{equation}
    given by
    \begin{align}
    \begin{split}
        \rho_{v,B}: \Phi_{v,H} \circ L^{B} &\Rightarrow [v]^{\ast}L^B \circ \Phi_{v,H} \\
        x &\mapsto \begin{tikzcd}[ampersand replacement = \&] x-v \ar[r,bend left=40,"{(0,exp(-B^{low}(v,x)))}"] \& x-v \end{tikzcd},
    \end{split}
    \end{align}
    which are indeed natural because
    \begin{equation}
        B^{low}(x,k) + H^{low}(v,x,k) - B^{low}(v,x+k) = -B^{low}(v,x) + H^{low}(v,x,k)+B^{low}(x-v,k).
    \end{equation}
    The dependence of $\rho_{v,B}$ on the representative $v$ is computed by observing
    \begin{equation}\label{eq:relation2}
        \begin{tikzcd}[column sep=4ex]
              & \mathcal{L}^H \ar[r,"{L^B}",{name=U}] \ar[dd,"{\Phi_{v+k,H}}"{name=i},pos=0.2] \ar[dl,"{L^{\iota_kH}}",swap]
             & \mathcal{L}^H \ar[dd,"{\Phi_{v+k,H}}"] \ar[Rightarrow, ddl,"{\rho_{v+k,B}}",pos=0.6]  & \\
             \mathcal{L}^H \ar[dr,"{\Phi_{v,H}}",swap] & & \\
             & \left[v\right]^{\ast}\mathcal{L}^H \ar[r,"{[v]^{\ast}L^B}"{name=S},swap,pos=0.5] \ar[Rightarrow,from=i,to=2-1,"{\Psi_{v,k,H}}",shorten=1ex,pos=0.75]
             & \left[v\right]^{\ast}\mathcal{L}^H 
            \end{tikzcd} =
        \begin{tikzcd}[column sep=5ex,row sep = 7ex]
               \mathcal{L}^H \ar[r,"{L^B}",{name=U}] \ar[d,"{L^{\iota_kH}}",swap]
             & \mathcal{L}^H \ar[d,"{L^{\iota_kH}}",swap] \ar[dd,bend left=60,"{\Phi_{v+k,H}}"{name=j},pos=0.85,shorten=1ex] & \\
             \mathcal{L}^H \ar[d,"{\Phi_{v,H}}"{name=i},pos=0.2,swap] \ar[r,"{L^B}"] & \mathcal{L}^H \ar[d,"{\Phi_{v,H}}",swap,pos=0.7] \ar[Rightarrow, dl,"{\rho_{v,B}}",swap,pos=0.6] & \\
             \left[v\right]^{\ast}\mathcal{L}^H \ar[r,"{[v]^{\ast}L^B}"{name=S},swap,pos=0.5] \ar[Rightarrow,from=j,to=2-2,"{\Psi_{v,k,H}}",swap,pos=1]
             & \left[v\right]^{\ast}\mathcal{L}^H 
            \end{tikzcd} exp(-\iota_kB-B^{low}(v,k)).
    \end{equation}
    The multiplicator 2-cells of $\phi$ are then defined by
    \begin{equation}
            \begin{tikzcd}[column sep=14ex]
             \mathcal{L}^{H_1+H_2} \ar[r,"{\Phi_{v_2,H_1} \otimes id}",{name=U}] \ar[d,"{L^{B_1} \otimes id}",swap] 
             & v_2^{\ast}\mathcal{L}^{H_1} \otimes \mathcal{L}^{H_2} \ar[d,"{v_2^{\ast}L^{B_1} \otimes id}"] \\
             \mathcal{L}^{H_1+H_2}  \ar[r,"{\Phi_{v_2,H_1} \otimes id}",{name=D},swap] \ar[Rightarrow, ur,"{\rho_{v_2,B_1}}"] 
             & v_2^{\ast}\mathcal{L}^{H_1} \otimes \mathcal{L}^{H_2}
            \end{tikzcd}, \quad
            \begin{tikzcd}[column sep=14ex]
             \mathcal{L}^{H_1+H_2} \ar[r,"{\Phi_{v_2,H_1} \otimes id}",{name=U}] \ar[d,"{id \otimes L^{-\iota_{k_2}H_1 + B_2}}",swap] 
             & v_2^{\ast}\mathcal{L}^{H_1} \otimes \mathcal{L}^{H_2} \ar[d,"{id \otimes L^{B_2}}"] \\
             \mathcal{L}^{H_1+H_2}  \ar[r,"{\Phi_{v_2+k_2,H_1} \otimes id}",{name=D},swap] \ar[Rightarrow, ur,"{\Psi_{v_2,k_2,H_1}^{-1}}"] 
             & v_2^{\ast}\mathcal{L}^{H_1} \otimes \mathcal{L}^{H_2}
            \end{tikzcd}
            \end{equation}
    Their naturality is immediate, their preservation of composition follows from \eqref{eq:relation1}, and their preservation of interchange cells follows from \eqref{eq:relation2}.

    Next we study the behaviour of $\Phi_{v,H}$ with respect to addition, which will let us define the associator 2-cells of $\phi$. It is clear that $\Phi_{v,H_1+H_2} = \Phi_{v,H_1} \circ \Phi_{v,H_2}$. Now, given $v_1, \,v_2 \in \mathbb R^n$, we define a natural transformation
    \begin{align}
    \begin{split}
        \chi_{v_1,v_2,H}:v_1^{\ast}\Phi_{v_2,H} \circ \Phi_{v_1,H} &\Rightarrow \Phi_{v_1+v_2,H} \\
        x &\mapsto \begin{tikzcd}[ampersand replacement = \&] x-v_1-v_2 \ar[r,bend left=40,"{(0,exp(H^{low}(v_2,v_1,x)))}"] \& x-v_1-v_2 \end{tikzcd},
    \end{split}
    \end{align}
    which is indeed natural because
    \begin{equation}
    H^{low}(v_1,x,k) + H^{low}(v_2,x-v_1,k) + H^{low}(v_2,v_1,x+k) = H^{low}(v_2,v_1,x) + H^{low}(v_1+v_2,x,k).
    \end{equation}
    We note the following relations, where we omit pull-backs by translations from the notation for clarity.
    \begin{equation}\label{eq:relation5}
         \begin{tikzcd}[column sep = 5ex, row sep = 13ex]
            \mathcal{L}^H \ar[r,"{L^B}"] \ar[dr,bend right = 40,"{\Phi_{v_1+v_2,H}}"{name=T},swap] & \mathcal{L}^H \ar[r,"{\Phi_{v_1,H}}"{name=S}]  \ar[rr,bend right = 50,"{}"{name=U},swap] \ar[Rightarrow,from=1-3,to=U,"{\chi_{v_1,v_2,H}}"{name=E},pos=0.25] &\mathcal{L}^H \ar[r,"{\Phi_{v_2,H}}"] &\mathcal{L}^H \\
            & \mathcal{L}^H  \ar[urr,bend right = 40,"{L^B}"{name=B},swap] & & \ar[Rightarrow,from=U,to=2-2,"{\rho_{v_1+v_2,B}}",swap,pos=0.4]
        \end{tikzcd} = 
        \begin{tikzcd}[column sep = 5ex, row sep = 6ex]
            \mathcal{L}^H \ar[ddrr,bend right=60,"{\Phi_{v_1+v_2,H}}"{name=p},swap] \ar[r,"{L^B}"{name=A},pos=0.55] \ar[dr,bend right = 20,"{\Phi_{v_1,H}}"{name=L, font = \tiny}, pos=0.50] & \mathcal{L}^H \ar[r,"{\Phi_{v_1,H}}"] &\mathcal{L}^H \ar[r,"{\Phi_{v_2,H}}"{name=f}] & \mathcal{L}^H \\
             & \mathcal{L}^H  \ar[ur,bend right = 50,"{L^B}"{name=B,font=\tiny},pos=0.4] & \ar[Rightarrow,from=1-2,to=2-2,shorten=1ex,"{\rho_{v_1,B}}",pos=0.4] \\
            & & \mathcal{L}^H \ar[uur, bend right= 60, "{L^{B}}"{name=Y},swap] \ar[from=2-2,to=3-3,"{\Phi_{v_2,H}}",pos=0.5] \ar[Rightarrow,from=2-2,to=p,"{\chi_{v_1,v_2,H}}",shorten=1.5ex] \ar[Rightarrow,from=f,to=3-3,"{\rho_{v_2,B}}",shorten=2ex ]& 
        \end{tikzcd} exp(-B^{low}(v_2,v_1)),
    \end{equation}
    \begin{align}\label{eq:relation4}
    \begin{split}
        \begin{tikzcd}[column sep = 4ex, row sep = 6ex, ampersand replacement = \&]
            \mathcal{L}^H \ar[r,"{\Phi_{v_1,H}}"{name=S}] \ar[ddr,bend right = 30,"{L^{\iota_{k_2}H}}"{name=T},swap]  \&\mathcal{L}^H  \ar[rr,"{\Phi_{v_2+k_2,H}}"{name=A},pos=0.55] \ar[dr,bend right = 20,"{L^{\iota_{k_2H}}}"{name=L, font = \tiny}, pos=0.70] \& \&\mathcal{L}^H \\
            \& \& \mathcal{L}^H  \ar[ur,bend right = 50,"{\Phi_{v_2,H}}"{name=B,font=\tiny},pos=0.4] \& \ar[Rightarrow,from=A,to=2-3,shorten=1ex,"{\Psi_{v_2,k_2,H}}",swap,pos=0.5] \ar[Rightarrow, from=1-2, to=T,"{\rho_{v_1,\iota_{k_2}H}^{-1}}",shorten=2.5ex,pos=0.35] \\
            \& \mathcal{L}^H \ar[ur,bend left=20,"{\Phi_{v_1,H}}",pos=0.2] \ar[uurr, bend right= 60, "{\Phi_{v_1+v_2,H}}"{name=Y},swap,pos=0.4] \ar[Rightarrow,from=2-3,to=Y,"{\chi_{v_1,v_2,H}}",swap,shorten=2ex] \& 
        \end{tikzcd} =
        \begin{tikzcd}[column sep = 8ex, row sep = 9ex, ampersand replacement = \&]
            \mathcal{L}^H \ar[r,"{\Phi_{v_1,H}}"{name=S}] \ar[dr,bend right = 40,"{L^{\iota_{k_2}H}}"{name=T,font=\tiny},swap,pos=0.9] \ar[rr,bend right = 50,"{}"{name=U},swap] \ar[Rightarrow,from=1-2,to=U,"{\chi_{v_1,v_2+k_2,H}}"{name=E},pos=0.25] \& \mathcal{L}^H \ar[r,"{\Phi_{v_2+k_2,H}}"] \& \mathcal{L}^H \\
            \& \mathcal{L}^H  \ar[ur,bend right = 50,"{\Phi_{v_1+v_2,H}}"{name=B,font=\tiny},swap,pos=0.1] \& \ar[Rightarrow,from=U,to=2-2,"{\Psi_{v_1+v_2,k_2,H}}",pos=0.4]
        \end{tikzcd} \alpha(v_1,v_2,k_2,H),
    \end{split}
    \end{align}
    where $\alpha(v_1,v_2,k_2,H) := exp(-H^{low}(v_2+k_2,v_1,k_2) + H^{low}(v_2,k_2,v_1))$,
    \begin{equation}\label{eq:relation3}
        \begin{tikzcd}[column sep = 8ex, row sep = 13ex]
            \mathcal{L}^H \ar[r,"{\Phi_{v_1+k_1,H}}"{name=S}] \ar[dr,bend right = 30,"{L^{\iota_{k_1}H}}"{name=T},swap] \ar[Rightarrow,from=1-1,to=2-2,shorten=2ex,"{\Psi_{v_1,k_1,H}}",pos=0.4] &\mathcal{L}^H \ar[r,"{\Phi_{v_2,H}}"] &\mathcal{L}^H \\
            & \mathcal{L}^H \ar[u,"{\Phi_{v_1,H}}"{name=A},swap,pos=0.2] \ar[ur,bend right = 30,"{\Phi_{v_1+v_2,H}}"{name=B},swap] & \ar[Rightarrow,from=1-2,to=B,shorten=2ex,"{\chi_{v_1,v_2,H}}",pos=0.3]
        \end{tikzcd} = 
        \begin{tikzcd}[column sep = 8ex, row sep = 13ex]
            \mathcal{L}^H \ar[r,"{\Phi_{v_1+k_1,H}}"{name=S}] \ar[dr,bend right = 40,"{L^{\iota_{k_1}H}}"{name=T},swap] \ar[rr,bend right = 50,"{}"{name=U},swap] \ar[Rightarrow,from=1-2,to=U,"{\chi_{v_1+k_1,v_2,H}}"{name=E},pos=0.25] &\mathcal{L}^H \ar[r,"{\Phi_{v_2,H}}"] &\mathcal{L}^H \\
            & \mathcal{L}^H  \ar[ur,bend right = 40,"{\Phi_{v_1+v_2,H}}"{name=B},swap] & \ar[Rightarrow,from=U,to=2-2,"{\Psi_{v_1+v_2,k_1,H}}",pos=0.4]
        \end{tikzcd} exp(-H^{low}(v_2,v_1+k_1,k_1)).
    \end{equation}
    \begin{equation}\label{eq:relation6}
         \begin{tikzcd}[column sep = 5ex, row sep = 13ex]
            \mathcal{L}^H \ar[r,"{\Phi_{v_1,H}}"] \ar[rrr,bend right = 80,"{\Phi_{v_1+v_2+v_3,H}}"{name=T},swap] & \mathcal{L}^H \ar[r,"{\Phi_{v_2,H}}"{name=S}]  \ar[rr,bend right = 50,"{}"{name=U},swap] \ar[Rightarrow,from=1-3,to=U,"{\chi_{v_2,v_3,H}}"{name=E},pos=0.25] &\mathcal{L}^H \ar[r,"{\Phi_{v_3,H}}"] &\mathcal{L}^H  \ar[Rightarrow,from=U,to=T,"{\chi_{v_1,v_2+v_3,H}}",swap,pos=0.4]
        \end{tikzcd} = 
        \begin{tikzcd}[column sep = 6ex, row sep = 13ex]
            \mathcal{L}^H \ar[r,"{\Phi_{v_1,H}}"{name=S}] \ar[rr,bend right = 50,"{}"{name=U},swap] \ar[Rightarrow,from=1-2,to=U,"{\chi_{v_1,v_2,H}}"{name=E},pos=0.25] \ar[rrr,bend right = 80,"{\Phi_{v_1+v_2+v_3,H}}"{name=T},swap] &\mathcal{L}^H \ar[r,"{\Phi_{v_2,H}}"] &\mathcal{L}^H  \ar[r,"{\Phi_{v_3,H}}"] &\mathcal{L}^H \ar[Rightarrow,from=1-3,to=T,"{\chi_{v_1+v_2,v_3,H}}",shorten=1ex,pos=0.35]
        \end{tikzcd} exp(-H^{low}(v_3,v_2,v_1))
    \end{equation}
    Then we define the associator 2-cells as follows.
    \begin{equation}
            \begin{tikzcd}[column sep=18ex]
             \mathcal{L}^{H_1+H_2+H_3} \ar[r,"{id}",{name=U}] \ar[d,"{\Phi_{v_3,H_1+H_2}}",swap] 
             & \mathcal{L}^{H_1+H_2+H_3} \ar[d,"{v_3^{\ast}\Phi_{v_2,H_1}}"] \ar[Rightarrow, ddl,"{\chi_{v_3,v_2,H_1}^{-1}}",swap] \\
             v_3^{\ast}\mathcal{L}^{H_1+H_2} \otimes \mathcal{L}^{H_3} \ar[d,"{\Phi_{v_2+v_3,H_1}}",swap]
             & (v_2+v_3)^{\ast}\mathcal{L}^{H_1} \mathcal{L}^{H_2+H_3} \ar[d,"{\Phi_{v_3,H_2}}"] \\
             (v_2+v_3)^{\ast}\mathcal{L}^{H_1} \otimes v_3^{\ast}\mathcal{L}^{H_1} \otimes \mathcal{L}^{H_3}  \ar[r,"{id}",swap] & (v_2+v_3)^{\ast}\mathcal{L}^{H_1} \otimes v_3^{\ast}\mathcal{L}^{H_1} \otimes \mathcal{L}^{H_3}
            \end{tikzcd}
        \end{equation}
    One can then check that \eqref{eq:relation5}, \eqref{eq:relation4} and \eqref{eq:relation3} yield naturality of the associator 2-cells, while \eqref{eq:relation6} shows they intertwine the pentagonator 2-cells.
    \end{proof}

    \begin{remark}
        It also follows directly from the proof of \cref{th:t2bnf1isgerbes} that $T_2 \mathbb B_n^{F_2}$ is the full sub-3-group of $\mathbb R^n/\mathbb Z^n \ltimes \mathcal{G}(\mathbb R^n/\mathbb Z^n)$ with objects $\mathbb R^n/\mathbb Z^n \times \{1\}$, where $1$ denotes the trivial gerbe over $\mathbb R^n/\mathbb Z^n$.
    \end{remark}

    \subsection{Dimensional reduction to T-duality}\label{sec:dimred}

    We recall the Lie 2-groups controlling T-duality in \cite{Nikolaus:2018qop}. First, $T\mathbb D_n^{F_1}$ has arrows of the form
    \begin{equation}
        \begin{tikzcd}
            (v,A,B_k) \ar[r,bend left = 20, "{(k,A_k,[x])}"] &(v+k,A+A_k,B_k),
        \end{tikzcd}
    \end{equation}
    where $(v,A) \in \mathbb R^n \oplus (\mathbb R^n)^{\ast}$, $B_k \in \Lambda^2 (\mathbb Z^n)^{\ast}$, $(k,A_k) \in \mathbb Z^n \oplus (\mathbb Z^n)^{\ast}$, $[x] \in \mathbb R/ \mathbb Z$. These are composed as follows
    \begin{equation}
        \begin{tikzcd}[every label/.append style = {font = \footnotesize},column sep = 2ex]
            {\scriptstyle  (v,A,B_k)} \ar[r,bend left = 20, "{(k,A_k,[x])}"] & {\scriptstyle (v+k,A+A_k,B_k)} \ar[r,bend left = 20, "{(k',A_k',[x'])}"] & {\scriptstyle (v+k+k',A+A_k+A_k',B_k)}
        \end{tikzcd} = 
        \begin{tikzcd}[every label/.append style = {font = \footnotesize},column sep = 2ex]
            {\scriptstyle (v,A,B_k)} \ar[r,bend left = 20, "{(k+k',A_k+A_k',[x+x'])}"] & {\scriptstyle (v+k+k',A+A_k+A_k',B_k)}.
        \end{tikzcd}
    \end{equation}
    The multiplication functor is
    \begin{align}
    \begin{split}\label{eq:mtdnf1}
        m(&(v_1,A_1,B_{k_1},k_1,[x_1]),(v_2,A_2,B_{k_2},k_2,[x_2])) \\
        &= (v_1+v_2,A_1+A_2+\iota_{v_2}B_{k_1},B_{k_1}+B_{k_2},k_1+k_2,A_{k_1}+A_{k_2} + \iota_{k_2}B_{k_1}, [x_1+x_2-\iota_{v_2}A_{k_1} + B_{k_1}^{low}(v_2,k_2)])
    \end{split}
    \end{align}
    and the associator is
    \begin{align}
    \begin{split}
        \alpha(v_1,A_1,&B_{k_1},v_2,A_2,B_{k_2},v_3,A_3,B_{k_3}) \\
        &= (v_1+v_2+v_3,A_1+A_2+A_3+\iota_{v_2}B_{k_1} + \iota_{v_3}(B_{k_1}+B_{k_2}),B_{k_1} + B_{k_2} + B_{k_3},0,[B_{k_1}^{low}(v_2,v_3)]).
    \end{split}
    \end{align}
    On the other hand, $T\mathbb B_n^{F_1}$ has arrows of the form
    \begin{equation}
        \begin{tikzcd}
            (v,B_k) \ar[r,bend left = 20, "{(k,A_k,[x])}"] &(v+k,B_k),
        \end{tikzcd}
    \end{equation}
    where $v \in \mathbb R^n$, $B_k \in \Lambda^2 (\mathbb Z^n)^{\ast}$, $(k,A_k) \in \mathbb Z^n \oplus (\mathbb Z^n)^{\ast}$, $[x] \in \mathbb R/ \mathbb Z$. These are composed as follows
    \begin{equation}
        \begin{tikzcd}[every label/.append style = {font = \footnotesize},column sep = 2ex]
            {\scriptstyle  (v,B_k)} \ar[r,bend left = 20, "{(k,A_k,[x])}"] & {\scriptstyle (v+k,B_k)} \ar[r,bend left = 20, "{(k',A_k',[x'])}"] & {\scriptstyle (v+k+k',B_k)}
        \end{tikzcd} = 
        \begin{tikzcd}[every label/.append style = {font = \footnotesize},column sep = 2ex]
            {\scriptstyle (v,B_k)} \ar[r,bend left = 20, "{(k+k',A_k+A_k',[x+x'])}"] & {\scriptstyle (v+k+k',B_k)}.
        \end{tikzcd}
    \end{equation}
    The multiplication functor is
    \begin{align}
    \begin{split}\label{eq:mtbnf1}
        m(&(v_1,B_{k_1},k_1,[x_1]),(v_2,B_{k_2},k_2,[x_2])) \\
        &= (v_1+v_2,B_{k_1}+B_{k_2},k_1+k_2,A_{k_1}+A_{k_2} + \iota_{k_2}B_{k_1}, [x_1+x_2-\iota_{v_2}A_{k_1} + B_{k_1}^{low}(v_2,k_2)])
    \end{split}
    \end{align}
    and the associator is
    \begin{align}
    \begin{split}
        \alpha(v_1,&B_{k_1},v_2,B_{k_2},v_3,B_{k_3}) = (v_1+v_2+v_3,B_{k_1} + B_{k_2} + B_{k_3},0,[B_{k_1}^{low}(v_2,v_3)]).
    \end{split}
    \end{align}
    The Lie 2-groups $T \mathbb B_n^{F_2}$ and $T \mathbb D_n^{F_2}$ are the Lie sub-2-groups of $T \mathbb B_n^{F_1}$ and $T \mathbb D_n^{F_1}$ obtained by setting $B_k=0$.

    \begin{theorem}\label{th:dimredf1}
        There are injective homomorphisms of Lie 3-groups
        \begin{equation}
\begin{tikzcd}
	& {T \mathbb B_n^{F_1}} & {T \mathbb D_n^{F_1}} \\
	{T_2\mathbb B_{n+1}^{F_1}} & {T_2\mathbb D_{n+1}^{F_1}}
	\arrow[from=1-2, to=2-1]
	\arrow[from=1-3, to=2-2]
\end{tikzcd}
        \end{equation}
        restricting to their $F_2$ sub-3-groups, and such that the diagram \eqref{eq:tdualspan} commutes.
    \end{theorem}
    \begin{proof}
        First, we resolve the $\mathbb R/\mathbb Z$ component of $T\mathbb D_{n}^{F_1}$ by considering the equivalent 2-groupoid $T\mathbb D_{n}^{F_1,res}$ with 2-cells of the form
        \begin{equation}
        \begin{tikzcd}[ampersand replacement = \& ]
                (v,A,B_k) \ar[r,bend left = 40, "{(k,A_k,x)}"{name=F}] \ar[r,bend right = 40, "{(k,A_k,x+x_k)}"{name=G},swap] \ar[Rightarrow,from=F,to=G,"{x_k}",swap] \& (v+k,A+A_k,B_k)
            \end{tikzcd},
    \end{equation}
    where $(v,A,B_k) \in \mathbb R^n \times (\mathbb R^n)^{\ast} \times  \Lambda^2 (\mathbb Z^n)^{\ast}$, $(k,A_k,x) \in \mathbb Z^n \times (\mathbb Z^n)^{\ast} \times \mathbb R$, and $x_k \in \mathbb Z$, composed in the obvious way. 
    
    Now we describe the 3-group structure that the Lie 2-group structure of $T\mathbb D_{n}^{F_1}$ induces on $T\mathbb D_{n}^{F_1,res}$. The fact that \eqref{eq:mtdnf1} respects composition relies crucially on the fact that arrows are labeled by $\mathbb R/\mathbb Z$ and not just $\mathbb R$; the failure of $m$ to preserve composition when lifting the formulas from $\mathbb R/\mathbb Z$ to $\mathbb R$ determines a compositor and an interchange cell for the induced product on $T\mathbb D_{n}^{F_1,res}$. The resulting Lie 3-group structure is summarized as follows. The multiplication functors are
    \begin{align}
        \begin{split}
            L_{(v_1,A_1,B_{k_1})}: T\mathbb D_n^{F_1,res} &\rightarrow T\mathbb D_n^{F_1,res} \\
            \begin{tikzcd}[ampersand replacement = \&,column sep=1ex ]
                {\scriptstyle (v_2,A_2,B_{k_2})} \ar[r,bend left = 60, "{(k_2,A_{k_2},x_2)}"{name=F}] \ar[r,bend right = 40, "{(k_2,B_{k_2},x_2+x_{k_2})}"{name=G},swap] \ar[Rightarrow,from=F,to=G,"{x_{k_2}}"{swap,pos=0.3,font=\tiny}] \& {\scriptstyle (v_2+k_2,A_2+A_{k_2},B_{k_2}) }
            \end{tikzcd} &\mapsto
            \begin{tikzcd}[ampersand replacement = \&,column sep = 0.5ex ]
                {\scriptstyle (v_1+v_2,A_1+A_2+\iota_{v_2}B_{k_1},B_{k_1}+B_{k_2}) }\ar[r,bend left = 40, "{(k_2,A_{k_2}+\iota_{k_2}B_{k_1},x_2+B_{k_1}^{low}(v_2,k_2))}"{name=F},pos=0.48] \ar[r,bend right = 40, "{(k_2,A_{k_2}+\iota_{k_2}B_{k_1},x_2+B_{k_1}^{low}(v_2,k_2)+x_{k_2})}"{name=G},swap,pos=0.48] \ar[Rightarrow,from=F,to=G,"{x_{k_2}}",swap,pos=0.3] \& {\scriptstyle (v_1+v_2+k_2,A_1+A_2+\iota_{v_2+k_2}B_{k_1}+A_{k_2},B_{k_1}+B_{k_2}) }
            \end{tikzcd} 
        \end{split},
    \end{align}
    \begin{align}
        \begin{split}
            R_{(v_2,A_2,B_{k_2})}: T\mathbb D_n^{F_1,res} &\rightarrow T\mathbb D_n^{F_1,res} \\
            \begin{tikzcd}[ampersand replacement = \&,column sep=1ex  ]
                {\scriptstyle (v_1,A_1,B_{k_1})} \ar[r,bend left = 60, "{(k_1,A_{k_1},x_1)}"{name=F}] \ar[r,bend right = 40, "{(k_1,A_{k_1},x_1+x_{k_1})}"{name=G},swap] \ar[Rightarrow,from=F,to=G,"{x_{k_1}}"{swap,pos=0.3,font=\tiny}] \& {\scriptstyle (v_1+k_1,A_1+A_{k_1},B_{k_1})}
            \end{tikzcd} &\mapsto
            \begin{tikzcd}[ampersand replacement = \&,column sep = 2ex ]
                {\scriptstyle (v_1+v_2,A_1+A_2+\iota_{v_2}B_{k_1},B_{k_1}+B_{k_2}) }\ar[r,bend left = 40, "{(k_1,A_{k_1},x_1-\iota_{v_2}A_{k_1})}"{name=F},pos=0.48] \ar[r,bend right = 40, "{(k_1,A_{k_1},x_1-\iota_{v_2}A_{k_1}+x_{k_1})}"{name=G},swap,pos=0.48] \ar[Rightarrow,from=F,to=G,"{x_{k_1}}"{swap,pos=0.3,font=\footnotesize}] \& {\scriptstyle (v_1+k_1+v_2,A_1+A_{k_1}+A_2 + \iota_{v_2}B_{k_1},B_{k_1}+B_{k_2}) }
            \end{tikzcd}
        \end{split}.
    \end{align}
    While $R_{(v_2,A_2,B_{k_2})}$ has trivial compositor, the compositor for $L_{(v_1,A_1,B_{k_1})}$ acting on 
    $$(v_2,A_2,B_{k_2}) \stackrel{(k_2,A_{k_2},x_2)}{\rightarrow} (v_2+k_2,A_2+A_{k_2},B_{k_2}) \stackrel{(k_2',A_{k_2}',x_2')}{\rightarrow} (v_2+k_2+k_2',A_2+A_{k_2}+A_{k_2}',B_{k_2})$$
    is
    \begin{equation}
    \begin{tikzcd}[column sep = 10ex]
            {\scriptstyle (v_1+v_2,A_1+A_2+\iota_{v_2}B_{k_1},B_{k_1}+B_{k_2})} \ar[rr,bend left=40,"{(k_2+k_2',A_{k_2}+A_{k_2}'+\iota_{k_2+k_2'}B_{k_1},x_2+x_2' + B^{low}_{k_1}(v_2,k_2+k_2'))}"{name=F},pos=0.5] 
              \ar[rr,bend right=10,"{(k_2+k_2',A_{k_2}+\iota_{k_2}B_{k_1}+A_{k_2}'+\iota_{k_2'}B_{k_1},x_2+x_2'+ B^{low}_{k_1}(v_2,k_2) +B^{low}_{k_1}(v_2+k_2,k_2') )}"{name=G},swap,pos=0.5]
             & & {\scriptstyle (v_1+v_2+k_2+k_2',A_1+A_2+\iota_{v_2+k_2+k_2'}B_{k_1}+A_{k_2}+A_{k_2}',B_{k_1}+B_{k_2})} \\ \ar[Rightarrow,from=F,to=G,"{B_{k_1}^{low}(k_2,k_2')}",shorten >=1.5pt,pos=0.4] &
        \end{tikzcd},
    \end{equation}
    The interchange 2-cells are defined for $(v_i,A_i,B_{k_i}) \stackrel{(k_i,A_{k_i},x_i)}{\rightarrow} (v_i+k_i,A_i+A_{k_i},B_{k_i})$, $i=1, \,2$ by 
    \begin{equation}
            \begin{tikzcd}[column sep=15ex,row sep = 10ex]
             {\scriptstyle (v_1+v_2,A_1+A_2+\iota_{v_2}B_{k_1},B_{k_1}+B_{k_2})} \ar[r,"{(k_1,A_{k_1},x_1-\iota_{v_2}A_{k_1})}",{name=U}] \ar[d,"{\substack{(k_2,A_{k_2}+\iota_{k_2}B_{k_1},\\ x_2+B^{low}_{k_1}(v_2,k_2))}}",swap] 
             & {\scriptstyle (v_1+k_1+v_2,A_1+A_{k_1}+A_2+\iota_{v_2}B_{k_1},B_{k_1}+B_{k_2}) } \ar[d,"{\substack{(k_2,A_{k_2}+\iota_{k_2}B_{k_1},\\ x_2+B^{low}_{k_1}(v_2,k_2))}}"] \ar[Rightarrow, dl,"{-\iota_{k_2}A_{k_1}}",swap,pos=0.5]\\
             {\scriptstyle (v_1+v_2+k_2,A_1+A_2+\iota_{v_2+k_2}B_{k_1}+A_{k_2},B_{k_1}+B_{k_2}) }\ar[r,"{(k_1,A_{k_1},x_1-\iota_{v_2+k_2}A_{k_1})}",{name=D},swap] 
             & {\scriptstyle (v_1+k_1+v_2+k_2,A_1+A_{k_1}+A_2+\iota_{v_2+k_2}B_{k_1}+A_{k_2},B_{k_1}+B_{k_2})}.
            \end{tikzcd}
            \end{equation}
    For $(v_i,A_i,B_{k_i})$, $i=1, \,2,\,3$, the associator arrows are
    \begin{equation}
    \begin{tikzcd}
        (v_1+v_2+v_3,A_1+A_2+\iota_{v_2}B_{k_1}+A_3 + \iota_{v_3}(B_{k_1}+B_{k_2}),B_{k_1}+B_{k_2}+B_{k_3}) \ar[d,"{(0,0,B^{low}_{k_1}(v_2,v_3))}"] \\
        (v_1+v_2+v_3,A_1+A_2+A_3+\iota_{v_3}B_{k_2} + \iota_{v_2+v_3}B_{k_1},B_{k_1}+B_{k_2}+B_{k_3})
    \end{tikzcd}
    \end{equation}
    and all other structure 2-cells are trivial. Now by comparison with \cref{sec:t2dnf1} we see that the Lie 3-group $T\mathbb D_n^{F_1,res}$ presented here agrees precisely with the Lie sub-3-group of $T_2\mathbb D_{n+1}^{F_1}$ generated by objects of the form $(v,A \wedge de_{n+1},B_k \wedge de_{n+1})$ with $(v,A,B_k) \in \mathbb R^n \times (\mathbb R^n)^{\ast} \times  \Lambda^2 (\mathbb Z^n)^{\ast}$, arrows of the form $(k,A_k \wedge de_{n+1}, x\,de_{n+1})$ with $(k,A_k,x) \in \mathbb Z^n \times (\mathbb Z^n)^{\ast} \times \mathbb R$, and 2-cells of the form $(x_kde_{n+1},0)$ with $x_k \in \mathbb Z$. Here $e_{n+1}$ is the last basis vector in $\mathbb R^{n+1}$ (according to the order chosen in \cref{not:low}) and we are embedding $\mathbb R^n \rightarrow \mathbb R^{n+1}$ setting $e_{n+1}=0$. This yields the homomorphism $T \mathbb D_n^{F_1} \rightarrow T_2 \mathbb D_{n+1}^{F_1}$; the homomorphism $T \mathbb B_n^{F_1} \rightarrow T_2 \mathbb B_{n+1}^{F_1}$ is similarly constructed.
    \end{proof}

\section{Conclusion and Outlook}\label{sec:conclusion}

The main contributions of this paper are the explicit presentation of the 2-category of gerbes over $\mathbb R^n/\mathbb Z^n$ in terms of representatives, tracing how these behave with respect to tensor product and pull-back by translations (Section \cref{sec:appgerbes}), and axiomatizing the corresponding algebraic structure as a 3-group (Sections \ref{sec:lie3}, \ref{sec:t2bnf1}). The homotopy equivalence (Section \ref{sec:t2dnf1}) and its dimensional reduction to T-duality (Section \ref{sec:dimred}) suggest the existence of a duality for higher-dimensional field theories with underlying non-homeomorphic 2-gerbes over torus fibrations. However, this remains outside of the scope of this paper, and will be explored in future work \cite{Gagliardo:2026inprep}.

\section*{Acknowledgements}

This paper is part of a larger project with Gianni Gagliardo and Christian S\"amann, which started during a stay at Heriot-Watt University. Discussions with them and support from Heriot-Watt University have therefore been very helpful for the results presented here. Part of the work was carried out during affiliation of the author with Instituto de Ciencias Matem\'aticas, and under financial support of Agencia Estatal de Investigaci\'on through grant PRE2019-089916.

\appendix

    \section{Checking that $T_2\mathbb B_n^{F_1}$ and $T_2\mathbb D_n^{F_1}$ are well-defined}\label{sec:welldefined}

    We include here the computations that show that the Lie 3-groups $T_2\mathbb B_n^{F_1}$ and $T_2\mathbb D_n^{F_1}$ from \cref{sec:t2bnf1} and \cref{sec:t2dnf1} are well-defined. All the equations of this section must be understood mod $\mathbb Z$. 
    
    We start with $T_2\mathbb B_n^{F_1}$. It is immediate to see that the multiplication functors $L_{(v_1,H_1)}$ and $R_{(v_2,H_2)}$ preserve the source, target, identity, and vertical composition maps. Naturality of the compositor, interchange, and associator 2-cells is also straightforward. Associativity of the compositors follows from
    \begin{align*}
        &H_1^{low}(v_2,k_2',k_2) - H_1^{low}(\cdot,k_2',k_2) + H_1^{low}(k_2',\cdot,k_2) + H_1^{low}(k_2,k_2',\cdot) \\
        & + H_1^{low}(v_2,k_2'',k_2+k_2') - H_1^{low}(\cdot,k_2'',k_2+k_2') + H_1^{low}(k_2'',\cdot,k_2+k_2') + H_1^{low}(k_2+k_2',k_2'',\cdot) \\
        &\quad = H_1^{low}(v_2,k_2'+k_2'',k_2) - H_1^{low}(\cdot,k_2'+k_2'',k_2) + H_1^{low}(k_2'+k_2'',\cdot,k_2) + H_1^{low}(k_2,k_2'+k_2'',\cdot)\\
        &+ H_1^{low}(v_2+k_2,k_2'',k_2') - H_1^{low}(\cdot,k_2'',k_2') + H_1^{low}(k_2'',\cdot,k_2') + H_1^{low}(k_2',k_2'',\cdot).
    \end{align*}
    Interchange cells respect composition of arrows because
    \begin{align*}
        -\iota_{k_2}B_1 - B_1^{low}(v_2,k_2)  -\iota_{k_2'}B_1 - B_1^{low}(v_2+k_2,k_2') &= -\iota_{k_2+k_2'}B_1 - B_1^{low}(v_2,k_2+k_2'),\\
        -\iota_{k_2}B_1 - B_1^{low}(v_2,k_2) -\iota_{k_2}B_1' - (B_1')^{low}(v_2,k_2) &= -\iota_{k_2}(B_1+B_1') - (B_1+B_1')^{low}(v_2,k_2).
    \end{align*}
    Associator cells respect composition of arrows because
    \begin{align*}
        &-B_1^{low}(v_2,v_3) - (B_1')^{low}(v_2,v_3) = -(B_1+B_1')^{low}(v_2,v_3),\\
        &-H_1^{low}(v_2+k_2,v_3,k_2) + H_1^{low}(v_2,k_2,v_3) -H_1^{low}(v_2+k_2+k_2',v_3,k_2') + H_1^{low}(v_2+k_2+k_2',k_2',v_3)\\
        &+ v_3^{\ast}(H_1^{low}(v_2,k_2',k_2) - H_1^{low}(\cdot,k_2',k_2) + H_1^{low}(k_2',\cdot,k_2) + H_1^{low}(k_2,k_2',\cdot))\\
        &= -H_1^{low}(v_2+k_2+k_2',v_3,k_2+k_2') + H_1^{low}(v_2,k_2+k_2',v_3)\\
        &+ H_1^{low}(v_2+v_3,k_2',k_2) - H_1^{low}(\cdot,k_2',k_2) + H_1^{low}(k_2',\cdot,k_2) + H_1^{low}(k_2,k_2',\cdot),\\
        &-H_1^{low}(v_2,v_3+k_3,k_3) - H_1^{low}(v_2,v_3+k_3+k_3',k_3')\\
        &+ (H_1+H_2)^{low}(v_3,k_3',k_3) - (H_1+H_2)^{low}(\cdot,k_3',k_3) + (H_1+H_2)^{low}(k_3',\cdot,k_3) + (H_1+H_2)^{low}(k_3,k_3',\cdot)\\
        &= -H_1^{low}(v_2,v_3+k_3+k_3',k_3+k_3') + H_2^{low}(v_3,k_3',k_3) - H_2^{low}(\cdot,k_3',k_3) + H_2^{low}(k_3',\cdot,k_3) + H_2^{low}(k_3,k_3',\cdot)\\
        &+ H_1^{low}(v_2+v_3,k_3',k_3) - H_1^{low}(\cdot,k_3',k_3) + H_1^{low}(k_3',\cdot,k_3) + H_1^{low}(k_3,k_3',\cdot).
    \end{align*}
    Associator cells respect interchange cells because
    \begin{align*}
        &v_3^{\ast}(-\iota_{k_2}B_1-B_1^{low}(v_2,k_2)) - (-\iota_{k_2}B_1-B_1^{low}(v_2+v_3,k_2)) = -B_1^{low}(v_2,v_3) + B_1^{low}(v_2+k_2,v_3),\\
        &-\iota_{k_3}(B_2+\iota_{k_2}H_1) - (B_2+\iota_{k_2}H_1)^{low}(v_3,k_3) - (-\iota_{k_3}B_2 - B_2^{low}(v_3,k_3)) \\
        &- (H_1^{low}(v_2+v_3,k_2,k_3) - H_1^{low}(\cdot,k_2,k_3) + H_1^{low}(k_2,\cdot,k_3) + H_1^{low}(k_3,k_2,\cdot))\\
        &+ (H_1^{low}(v_2+v_3,k_3,k_2) - H_1^{low}(\cdot,k_3,k_2) + H_1^{low}(k_3,\cdot,k_2) + H_1^{low}(k_2,k_3,\cdot)) \\
        &= -H_1^{low}(v_2+k_2,v_3,k_2) + H_1^{low}(v_2,k_2,v_3) - (-H_1^{low}(v_2+k_2,v_3+k_3,k_2) + H_1^{low}(v_2,k_2,v_3+k_3))\\
        &-H_1^{low}(v_2+k_2,v_3+k_3,k_3) + H_1^{low}(v_2,v_3+k_3,k_3),\\
        &-\iota_{k_3}B_1-B_1^{low}(v_3,k_3) - (-\iota_{k_3}B_1 - B_1^{low}(v_2+v_3,k_3)) = -B_1^{low}(v_2,v_3) + B_1^{low}(v_2,v_3+k_3).
    \end{align*}
    Pentagonator cells are natural because
    \begin{align*}
        &-B_1^{low}(v_2,v_3) - B_1^{low}(v_2+v_3,v_4) = - B_1^{low}(v_3,v_4) - B_1^{low}(v_2,v_3+v_4),\\
        &-H_1^{low}(v_2+k_2,v_3,k_2) + H_1^{low}(v_2,k_2,v_3) -H_1^{low}(v_2+k_2 + v_3,v_4,k_2) + H_1^{low}(v_2+v_3,k_2,v_4) \\
        &- B_2^{low}(v_3,v_4) - H_1^{low}(v_2+k_2,v_3,v_4)\\
        &= - (B_2+\iota_{k_2}H_1)^{low}(v_3,v_4) - H_1^{low}(v_2+k_2,v_3+v_4,k_2) + H_1^{low}(v_2,k_2,v_3+v_4)-H_1^{low}(v_2,v_3,v_4),\\
        &-H_1^{low}(v_2,v_3+k_3,k_3) - H_1^{low}(v_2+v_3+k_3,v_4,k_3) + H_1^{low}(v_2+v_3,k_3,v_4) - H_2^{low}(v_3+k_3,v_4,k_3)\\
        &+ H_2^{low}(v_3,k_3,v_4)  - H_1^{low}(v_2,v_3+k_3,v_4)\\
        &= -(H_1+H_2)^{low}(v_3+k_3,v_4,k_3) + (H_1+H_2)^{low}(v_3,k_3,v_4) - H_1^{low}(v_2,v_3+k_3+v_4,k_3) - H_1^{low}(v_2,v_3,v_4),\\
        &-H_1^{low}(v_2+v_3,v_4+k_4,k_4) - H_2^{low}(v_3,v_4+k_4,k_4) - H_1^{low}(v_2,v_3,v_4+k_4) \\
        &= - (H_1+H_2)^{low}(v_3,v_4+k_4,k_4) - H_1^{low}(v_2,v_3+v_4+k_4,k_4) - H_1^{low}(v_2,v_3,v_4).
    \end{align*}
    The pentagonator satisfies its cocycle condition because
    \begin{align*}
        H_1^{low}(v_2,v_3,v_4) + &H_1^{low}(v_2,v_3+v_4,v_5) + (H_1+H_2)^{low}(v_3,v_4,v_5) \\
        &= H_1^{low}(v_2,v_3,v_4+v_5) + H_1^{low}(v_2+v_3,v_4,v_5) + H_2^{low}(v_3,v_4,v_5).
    \end{align*}

    This concludes the computations that show that $T_2\mathbb B_n^{F_1}$ is well-defined. It also follows that $T_2\mathbb D_n^{F_1}$ is well-defined. This is because all the identities we have to check are equations between 2-cells, and in fact only the element in $U(1)$ has to agree on both sides, as the element in $(\mathbb Z^n)^{\ast}$ is determined by the source and target of the 2-cell. But since the formulas for the $U(1)$ element of the 2-cells defining $T_2\mathbb D_n^{F_1}$ coincide with those that define $T_2\mathbb B_n^{F_1}$, the same computations prove that $T_2\mathbb D_n^{F_1}$ is well-defined. Alternatively, one can also check that all the extra 2-cells appearing in the axioms of a 3-group when there are non-trivial associator arrows vanish in this case because the associator arrows have $k=0$, $B_k=0$.

\bibliographystyle{latexeu}
\bibliography{bigone}

\end{document}